%% file: main.tex
\documentclass[twoside,11pt]{article}

\usepackage[abbrvbib, preprint]{jmlr2e}

\usepackage{jmlr2e}

\input{macros}

\usepackage{lastpage}
\jmlrheading{}{2023}{1-\pageref{LastPage}}{\today}{9/22}{}{Anish Agarwal, Devavrat Shah, and Dennis Shen}

\ShortHeadings{On Model Identification and Out-of-Sample Prediction of PCR}{Agarwal, Shah, and Shen}
\firstpageno{1}

\begin{document}

\title{On Model Identification and Out-of-Sample Prediction of PCR:  Applications to Synthetic Controls}

\author{\name Anish Agarwal \email aa5194@columbia.edu \\
       \addr Department of Industrial Engineering \& Operations Research\\
       Columbia University
       \AND
       \name Devavrat Shah \email devavrat@mit.edu \\
       \addr Department of Electrical Engineering \& Computer Science\\
       Massachusetts Institute of Technology 
       \AND
       \name Dennis Shen \email dennis.shen@marshall.usc.edu \\
       \addr Department of Data Sciences \& Operations \\
       University of Southern California}

\editor{}

\maketitle

\begin{abstract}
We analyze principal component regression (PCR) in a high-dimensional error-in-variables setting with fixed design.
Under suitable conditions, we show that PCR consistently identifies the unique model with minimum $\ell_2$-norm. 
These results enable us to establish non-asymptotic out-of-sample prediction guarantees that improve upon the best known rates. 
In the course of our analysis, we introduce a natural linear algebraic condition between the in- and out-of-sample covariates, which allows us to avoid distributional assumptions for out-of-sample predictions. 
Our simulations illustrate the importance of this condition for generalization, even under covariate shifts. 
Accordingly, we construct a hypothesis test to check when this condition holds in practice. 
As a byproduct, our results also lead to novel results for the synthetic controls literature, a leading approach for policy evaluation. 
To the best of our knowledge, our prediction guarantees for the fixed design setting have been elusive in both the high-dimensional error-in-variables and synthetic controls literatures. 

\end{abstract}


\input{content/introduction}

\input{content/algorithm}

\input{content/setup}

\input{content/results}

\input{content/exp}

\input{content/hypo}

\input{content/synthetic_controls}

\input{content/discussion_results}
\input{content/discussion}


\acks{We thank Peng Ding and various members within MIT’s Laboratory for Information and Decision Systems (LIDS) for useful discussions and guidance. 
The data and code to reproduce the results in this article are available at \href{https://github.com/deshen24/principal-component-regression}{https://github.com/deshen24/principal-component-regression}.  
}


\newpage

\appendix
\input{content/appendix_simulation}

\input{content/proof_paramest}

\input{content/proof_param_est_lemmas}

\input{content/proof_generalization}

\input{content/proof_prob}

\input{content/proof_hypo_test}
\input{content/appendix_lower_bound}

\input{content/proof_lower_bound}

\vskip 0.2in
\bibliography{bibliography}

\end{document}

%% file: macros.tex
\usepackage{mathtools}
\mathtoolsset{showonlyrefs}
\usepackage{bm}
\usepackage{physics}
\usepackage{enumitem}
\usepackage{bbold}
\usepackage{epigraph}
\usepackage{csquotes}
\usepackage{hyperref}       
\usepackage{booktabs}       
\usepackage{multirow}
\usepackage{amsfonts}       
\usepackage{nicefrac}       
\usepackage{scalerel}
\usepackage[toc,page]{appendix}
\usepackage{dsfont}
\allowdisplaybreaks
\usepackage{float}
\usepackage{graphicx}
\usepackage{subfloat}
\usepackage{subcaption}
\graphicspath{ {images/} }
\usepackage{letltxmacro}%
\usepackage{thmtools}
\usepackage{thm-restate}
\usepackage{relsize}
\usepackage{multirow}
\usepackage[most]{tcolorbox}

\newtheorem{thm}{Theorem}[section]
\newtheorem{cor}{Corollary}[section]
\newtheorem{property}{Property}[section]
\newtheorem{assumption}{Assumption}[section]

\DeclareMathOperator*{\argmin}{arg\,min}

\newcommand{\Ec}{\mathcal{E}}
\newcommand{\Nc}{\mathcal{N}}

\newcommand{\Ex}{\mathbb{E}}
\newcommand{\Rb}{\mathbb{R}}

\newcommand{\Pb}{\mathbb{P}}

\newcommand{\Id}{\mathbf{I}}

\newcommand{\bU}{\boldsymbol{U}}
\newcommand{\bW}{\boldsymbol{W}}
\newcommand{\bV}{\boldsymbol{V}}

\newcommand{\bX}{\boldsymbol{X}}
\newcommand{\bZ}{\boldsymbol{Z}}
\newcommand{\bPi}{\boldsymbol{\Pi}}
\newcommand{\bPip}{\bPi^\prime}
\newcommand{\bXp}{\bX^\prime}
\newcommand{\bWp}{\bW^\prime}
\newcommand{\bZp}{\bZ^\prime}
\newcommand{\bHp}{\bH^\prime}

\newcommand{\bhH}{\widehat{\bH}}

\newcommand{\btheta}{\boldsymbol{\theta}} 
\newcommand{\bz}{\boldsymbol{z}}
\newcommand{\bx}{\boldsymbol{x}} 
\newcommand{\bv}{\boldsymbol{v}}
\newcommand{\bu}{\boldsymbol{u}} 
\newcommand{\bw}{\boldsymbol{w}} 
 
\newcommand{\bhv}{\widehat{\bv}}
\newcommand{\bhu}{\widehat{\bu}} 
\newcommand{\bhy}{\widehat{\by}} 
\newcommand{\bhup}{\bhu^\prime} 
\newcommand{\bhvp}{\bhv^\prime} 
\newcommand{\bhyp}{\bhy^\prime} 
\newcommand{\bbeta}{\boldsymbol{\beta}}
\newcommand{\bhbeta}{\widehat{\bbeta}} 
\newcommand{\bvarepsilon}{\boldsymbol{\varepsilon}} 
\newcommand{\bpi}{\boldsymbol{\pi}} 

\newcommand{\enoise}{\emph{noise}}

\newcommand{\bSp}{\bS^\prime}
\newcommand{\bVp}{{\bV^\prime}}
\newcommand{\bUp}{{\bU^\prime}}
\newcommand{\rp}{r^\prime}
\newcommand{\hsp}{{\hs}^{\prime}}

\newcommand{\bzeta}{\boldsymbol{\zeta}}
\newcommand{\be}{\boldsymbol{e}}
\newcommand{\basis}{\boldsymbol{e}}
\newcommand{\by}{\boldsymbol{y}}
\newcommand{\bY}{\boldsymbol{Y}}

\newcommand{\bQ}{\boldsymbol{Q}}
\newcommand{\bH}{\boldsymbol{H}}

\newcommand{\bA}{\boldsymbol{A}}
\newcommand{\bB}{\boldsymbol{B}}
\newcommand{\bD}{\boldsymbol{D}}

\newcommand{\bP}{\boldsymbol{P}}
\newcommand{\bI}{\boldsymbol{I}}
\newcommand{\bS}{\boldsymbol{S}}
\newcommand{\bzero}{\boldsymbol{0}}
\newcommand{\bone}{\boldsymbol{1}}
\newcommand{\bSigma}{\boldsymbol{\Sigma}}

\newcommand{\btZ}{\widetilde{\bZ}}
\newcommand{\btZp}{{\btZ^\prime}}

\newcommand{\hs}{\widehat{s}}

\newcommand{\hy}{\widehat{y}}

\newcommand{\hrho}{\widehat{\rho}}
\newcommand{\hrhop}{\hrho^{\prime}}

\newcommand{\cE}{\mathcal{E}}

\newcommand{\Reals}{\mathbb{R}}
\newcommand{\Prob}[1]{\mathbb{P}\left( #1 \right) }

\newcommand{\bhU}{\widehat{\bU}}

\newcommand{\bhS}{\widehat{\bS}}

\newcommand{\bhV}{\widehat{\bV}}

\newcommand{\test}{\text{test}}
\newcommand{\etest}{\emph{test}}

\newcommand{\diag}{\text{diag}}

\newcommand{\epost}{\emph{post}}
\newcommand{\htau}{\widehat{\tau}}
\newcommand{\Dc}{\mathcal{D}}

\newcommand{\Var}{\mathbb{V}\text{ar}}
\newcommand{\btW}{\tilde{\bW}}

\newcommand{\tO}{\tilde{O}}

\newcommand{\distas}[1]{\mathbin{\overset{#1}{\kern\z@\sim}}}%
\newsavebox{\mybox}\newsavebox{\mysim}
\newcommand{\distras}[1]{%
  \savebox{\mybox}{\hbox{\kern3pt$\scriptstyle#1$\kern3pt}}%
  \savebox{\mysim}{\hbox{$\sim$}}%
  \mathbin{\overset{#1}{\kern\z@\resizebox{\wd\mybox}{\ht\mysim}{$\sim$}}}%
}

\newcommand{\btbeta}{\tilde{\bbeta}}

\newcommand{\bEta}{\boldsymbol{\eta}}

\newcommand{\trunc}{\text{trunc}}
\newcommand{\etrunc}{\emph{trunc}}

\newcommand{\snr}{{\mathsf {snr}}}

\renewcommand{\tO}{\widetilde{O}}
\newcommand{\NA}{\mathsf{NA}} 

\newcommand{\bbhtheta}{\boldsymbol{\widehat{\theta}}}

\newcommand{\ba}{\boldsymbol{a}}
\newcommand{\bg}{\boldsymbol{g}}

%% file: content/introduction.tex
\section{Introduction} \label{sec:intro}
%
%
We consider error-in-variables regression in a high-dimensional setting with fixed design. 
Formally, we observe a labeled dataset of size $n$, denoted as $\{(y_i, \bz_i): i \le n \}$.
Here, $y_i \in \Reals$ is the response variable and $\bz_i \in \Rb^p$ is the observed covariate. 
For any $i \geq 1$, we posit that 
\begin{align}\label{eq:response_vector}  
	y_i & = \langle \bx_i, \bbeta^* \rangle + \varepsilon_i,
\end{align}
where $\bbeta^* \in \mathbb{R}^{p}$ is the unknown model parameter, $\bx_i \in \Rb^p$ is a fixed  covariate, and $\varepsilon_i \in \Reals$ is the response noise. 
Unlike traditional settings where $\bz_i = \bx_i$, the error-in-variables (EiV) setting reveals a corrupted version of the covariate $\bx_i$. 
Precisely, for any $i \geq 1$, let 
\begin{align}\label{eq:noisy_covariate}  
	\bz_i & = (\bx_i + \bw_i) \circ \bpi_i,
\end{align}
where $\bw_i \in \Rb^p$ is the covariate measurement noise, $\bpi_i \in \{1, \NA\}^p$ is a binary mask with $\NA$ denoting a missing value, and $\circ$ is the Hadamard product.
Further, we consider a high-dimensional setting where $n$ and $p$ are growing with $n$ possibly smaller than $p$. 

We analyze the classical method of principal component regression (PCR) within this framework. 
PCR is a two-stage process: 
first, PCR ``de-noises'' the observed in-sample covariate matrix $\bZ = [\bz^T_i] \in \Reals^{n \times p}$ via principal component analysis (PCA), i.e., PCR replaces $\bZ$ by its low-rank approximation. 
Then, PCR regresses $\by = [y_i] \in \Reals^n$ on the low-rank approximation to produce the model estimate $\bhbeta$.  
This focus of this work is to answer the following questions about PCR:
%
%
\begin{tcolorbox}[colback=gray!10!white,colframe=black!75!black]
	\begin{center}
	{
		{\bf Q1:} ``When $p > n$, is there a model parameter that PCR consistently identifies?''
		\\ 
		{\bf Q2:} ``Given deterministic, corrupted, and partially observed out-of-sample covariates, can PCR recover the expected responses?''
	} 
	\end{center} 
\end{tcolorbox}

%

\subsection{Contributions}\label{ssec:contributions}
{\em Model identification.} 
Regarding Q1, we prove that PCR consistently identifies the  projection of the model parameter onto the linear space generated by the underlying covariates.
This corresponds to the unique minimum $\ell_2$-norm model, which is arguably sufficient for valid statistical inference \citep{shao_deng}. 

\smallskip
\noindent 
{\em Out-of-sample prediction.} 
For Q2, we leverage our results for Q1 to establish non-asymptotic out-of-sample prediction guarantees that improve upon the best known rates. 
Notably, these results are novel for the fixed design setting.
In the course of our analysis, we introduce a natural linear algebraic condition between the in- and out-of-sample data that supplants distributional assumptions on the underlying covariates that are common in the literature. 
We construct a hypothesis test to check when this condition holds in practice. 
We also illustrate the importance of this condition through extensive simulations. 

\smallskip
\noindent
{\em Applications to synthetic controls.} 
Our responses to Q1--Q2 lead to novel results for the synthetic controls literature, a popular framework for policy evaluation \citep{abadie1, abadie2}. 
In particular, our results provide theoretical guarantees for several PCR based methods, namely \cite{rsc, mrsc}. 
To the best of our knowledge, we provide the first counterfactual $\ell_2$-prediction guarantees for the entire counterfactual trajectory in a fixed design setting for the synthetic controls literature. 
We apply our hypothesis test to two widely analyzed studies in the synthetic controls literature.

\subsection{Organization}
Section~\ref{sec:alg} details the PCR algorithm. 
Section~\ref{sec:setup_eiv_regression} describes our problem setup and assumptions. 
Section~\ref{sec:results} provides formal statistical guarantees on Q1--Q2. 
Section~\ref{sec:simulations} reports on simulation studies. 
Section~\ref{sec:hypo} presents a hypothesis test to check when a key assumption that enables PCR to generalize holds in practice. 
Section~\ref{sec:synthetic_controls} contextualizes our findings within the synthetic controls framework. 
Section~\ref{sec:discussion} discusses related works from the error-in-variables, PCR, and functional PCA/PCR literatures. 
Section~\ref{sec:conclusion} offers directions for future research. 
We relegate all mathematical proofs to the Appendix. 

\subsection{Notation} 
For a matrix $\bA \in \Rb^{a \times b}$, we denote its operator (spectral), Frobenius, and max element-wise norms as $\|\bA\|_{2}$, $\|\bA\|_F$, and $\|\bA\|_{\max}$, respectively. 
By $\text{rowspan}(\bA)$, we denote the
subspace of $\Rb^b$ spanned by the rows of $\bA$. 
Let $\bA^\dagger$ denote the pseudoinverse of $\bA$. 
For a vector $\bv \in \Rb^a$, let $\|\bv\|_p$ denote its 
$\ell_p$-norm.
We define the sub-gaussian (Orlicz) norm as $\|\bv\|_{\psi_2}$. 
Let $\langle \cdot, \cdot \rangle$ and $\otimes$ denote the inner and outer products, respectively.
For any two numbers $a, b \in\Rb$, we use $a \wedge b$ to denote $\min(a, b)$ and $a \vee b$ to denote $\max(a, b)$. 
Let $[a] = \{1, \dots, a\}$ for any positive integer $a$.

Let $f$ and $g$ be two functions defined on the same space. We say that $f(n)$ = $O(g(n))$ if and only if
there exists a positive real number $M$ and a real number $n_0$ such that for all $n \ge n_0, |f (n)| \le M|g(n)|$.
Analogously we say
$f (n) = \Theta(g(n))$ if and only if there exists positive real numbers $m, M$ such that for all $n \ge n_0, \ m|g(n)| \le |f(n)| \le M|g(n)|$;
$f (n) = o(g(n))$ if for any $m > 0$, there exists $n_0$ such that for all 
$n \ge n_0, |f(n)| \le m|g(n)|$;
$f (n) = \omega(g(n))$ if for any $m > 0$, there exists $n_0$ such that for all 
$n \ge n_0, |f(n)| \ge m|g(n)|$.
$\tO(\cdot)$ is defined analogously to $O(\cdot)$, but ignores $\log$ dependencies.

%% file: content/algorithm.tex
\section{Principal Component Regression} \label{sec:alg}

\subsection{Observations} 
As described in Section \ref{sec:intro}, our in-sample (train) data consists of $n$ {\em labeled} observations $\{(y_i, \bz_i): i \le n\}$. 
By contrast, our out-of-sample (test) data consists of $m \ge 1$ {\em unlabeled} observations. 
That is, for $i > n$, we observe the covariates $\bz_i$ but do not observe the associated response variables $y_i$. 
Let $\bZ = [\bz_i^T: i \le n] \in \Rb^{n \times p}$ and 
$\bZp = [\bz_i^T: i > n] \in \Rb^{m \times p}$ denote the matrices of in- and out-of-sample covariates, respectively. 

\subsection{Description of Algorithm} \label{sec:pcr_description} 
We describe PCR, as introduced in \cite{pcr_jolliffe}, with a variation to handle missing data.  

\medskip
\noindent 
{\em I: Model identification.} 
Let $\hrho$ denote the fraction of observed entries in $\bZ$.
Replace all missing values ($\NA$) in the covariate matrices with zero. 
Let $\btZ = (1/\hrho) \bZ = \sum_{i=1}^{n \wedge p} \hs_i \bhu_i \otimes \bhv_i$,
where $\hs_i \in \Rb$ are the singular values and $\bhu_i \in \Rb^n, \bhv_i \in \Rb^p$ are the left and right singular vectors, respectively.  
For a hyperparameter $k \in [n \wedge p]$, 
let $\btZ^k = \sum_{i=1}^k \hs_i \bhu_i \otimes \bhv_i$ 
and define the estimated model parameter as 
\begin{align} \label{eq:PCR_model.threshold}
    \bhbeta 
    &= ( \btZ^k )^\dagger \by
    = \Big( \sum_{i=1}^k (1/\hs_i) \bhv_i \otimes \bhu_i \Big) \by. 
\end{align}

\noindent
{\em II: Out-of-sample prediction.} 
Let $\hrhop$ denote the proportion of observed entries in $\bZp$. 
Let $\btZp = (1/\hrhop) \bZp = \sum_{i=1}^{m \wedge p} \hsp_i \bhup_i \otimes {\bhvp}_i{\,}$, 
where $\hsp_i \in \Rb$ are the singular values and $\bhup_i \in \Rb^m, \bhvp_i \in \Rb^p$ are the left and right singular vectors, respectively.  
Given algorithmic parameter $\ell \in [m \wedge p]$, let $\btZ^{\prime \ell} = \sum_{i=1}^{\ell} \hsp_i \bhup_i \otimes {\bhvp_i}{\,}$,
and define the test response estimates as $\bhy^\prime = \btZ^{\prime \ell} \bhbeta$. 

If the expected responses are known to belong to a bounded interval, say $[-b, b]$ for some $b > 0$, then the entries of $\bhy^\prime$ are truncated as follows:
for every $i > n$, 
\begin{align}\label{eq:response.truncate}
\hy^\trunc_i & = \begin{cases}
			-b & \mbox{~if~}\hy_i < -b, \\
			\hy_i & \mbox{~if~} -b \le \hy_i \le b, \\
			b & \mbox{~if~} \hy_i > b.
	       \end{cases}
\end{align}

\subsection{Additional Useful Properties of PCR} 
We state a few useful properties of PCR that we use extensively. 
These are well-known results that are discussed in Chapter 17 of \cite{sroman} and Chapter 6.3 of \cite{strang}.

\begin{property}\label{property:equiv}
The PCR solution, $\bhbeta$, as given in \eqref{eq:PCR_model.threshold}, is  
\begin{itemize}
\item[1.] 
the unique solution to the following program: 
\begin{align}
	\mbox{minimize} & \quad  \norm{\bbeta}_2  \quad \mbox{over}\quad \bbeta \in \Rb^p \nonumber \\
	\mbox{such~that}~& ~\bbeta \in \argmin_{\bbeta^\prime \in \Rb^p} 
	~\| \by - \btZ^k \bbeta^\prime \|_2^2. 
\end{align} 

\item[2.] embedded within the $\emph{rowspan}(\btZ^k)$.
\end{itemize}
\end{property}

\subsection{Applying PCR in Practice} 

\subsubsection{Imputing Missing Covariate Values} \label{sec:zero_imputation} 
As shown in \cite{robust_pcr, pcr_jasa}, PCR can equivalently be interpreted as first applying the matrix completion algorithm, hard singular value thresholding (HSVT), on $\btZ$ to obtain $\btZ^k$, and then performing OLS with this de-noised output matrix. 
Accordingly, this work utilizes the simple imputation method of replacing $\NA$ values with zero to enable HSVT. 
We justify this imputation approach as follows: by setting $\NA$ values to zero, it follows that $\Ex[Z_{ij}] = \rho X_{ij} + (1 - \rho) 0 = \rho X_{ij}$; recalling $\tilde{Z}_{ij} = (1 / \hrho) Z_{ij}$, we then obtain $\Ex[\tilde{Z}_{ij}] = X_{ij}$.
Indeed, constructing $\btZ$ such that $\Ex[\btZ] = \bX$ is a crucial step that enables the HSVT subroutine of PCR to produce a good estimate of $\bX$ through $\btZ^k$. 

Naturally, there are other matrix completion methods such as nearest neighbors or alternative least squares  that do not first impute missing values. 
As long as the approach taken yields a sufficiently good estimator, cf. Lemma~\ref{lem:thm1.2} of Appendix~\ref{sec:proof_thm_4.1}, our main results on model parameter identification and generalization would naturally extend to these settings. 

\subsubsection{Choosing the Number of Principal Components} \label{sec:choose_k} 
The ideal number of principal components $k$ is rarely known a priori. 
As such, the problem of choosing $k$ has become a well-studied problem in the low-rank matrix completion literature and there exists a suite of principled methods.
These include visual inspections of the plotted singular values \citep{cattell}, cross-validation \citep{wold1978, owen2009}, 
Bayesian methods \citep{hoff}, and ``universal'' thresholding schemes that preserve singular values above a precomputed threshold \citep{Gavish_2014, usvt}. 

A common argument for these approaches is rooted in the underlying assumption that the smallest non-zero singular value of the ``signal'' $\bX$ is well-separated from the largest singular value of the ``noise'' $\bW$.
Under reasonable ``signal-to-noise'' ($\snr$) scenarios, Weyl's inequality implies that a sharp threshold or gap should exist between the top $r$ singular values and remaining singular values of the observed data $\btZ$. 
This gives rise to a natural ``elbow'' point, shown in Figure~\ref{fig:elbow}, and suggests choosing a threshold within this gap. 
As such, a researcher can simply plot the singular values of $\btZ$ and look for the elbow structure to decide if PCR is suitable for the application at hand.
We formalize a notion of $\snr$ in \eqref{eq:snr} and establish our results in the following section with respect to this quantity.
%

\begin{figure} 
     \centering
     \includegraphics[width=0.40\linewidth]
     {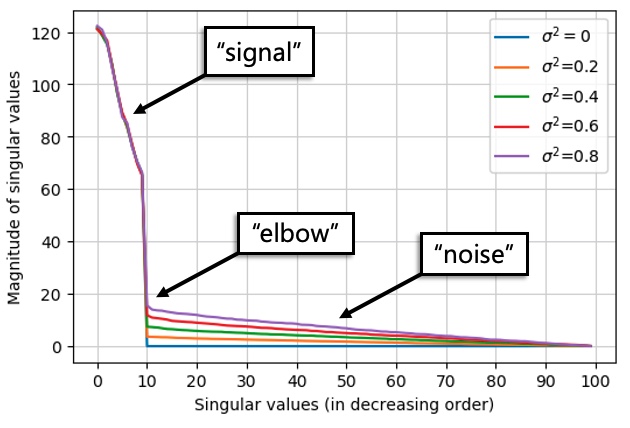}
     \caption{
     {\smaller
     Spectrum of $\bZ = \bX + \bW \in \Rb^{100 \times 100}$. 
     Here, $\bX = \bU \bV^T$, where entries of $\bU, \bV \in \Rb^{100 \times 10}$ are sampled independently from $\mathcal{N}(0,1)$;  
     entries of $\bW$ are sampled independently from $\mathcal{N}(0, \sigma^2)$ with $\sigma^2 \in \{0, 0.2, \dots, 0.8\}$. 
     We see a steep drop-off in magnitude in the singular values across all noise levels---this marks the ``elbow'' point. 
     Top singular values of $\bZ$ correspond closely with those of $\bX$ ($\sigma^2=0$). The remaining singular values are induced by $\bW$.
     Thus, the ``effective rank'' of $\bZ$ is the rank of $\bX$.
     }
     }
     \label{fig:elbow}
 \end{figure}

%% file: content/setup.tex
\section{Problem Setup} \label{sec:setup_eiv_regression}
This section formalizes our problem setup. 
Let $\bX = [\bx_i^T: i \le n] \in \Rb^{n \times p}$ and $\bXp = [\bx_i^T: i > n] \in \Rb^{m \times p}$ represent the underlying in- and out-of-sample covariates, respectively.

\subsection{Assumptions}  
Collectively, we assume \eqref{eq:response_vector} and \eqref{eq:noisy_covariate} are satisfied. 
We make the additional assumptions. 

\begin{assumption}[Response noise]\label{assumption:response_noise} 
Let $\{\bvarepsilon_i: i \le n \}$ be a sequence of independent mean zero subgaussian random variables with $\norm{\bvarepsilon_i}_{\psi_2} \le \sigma$.
\end{assumption} 
Assumption~\ref{assumption:response_noise} is a standard assumption in the regression literature that posits the idiosyncratic response noise to be independent draws from a subgaussian distribution. 
%

\begin{assumption} [Covariate noise and missing values] \label{assumption:cov_noise} 
Let $\{\bw_i: i \le n+m\}$ be a sequence of independent mean zero subgaussian random vectors with $\norm{\bw_i}_{\psi_2} \le K$ and $\| \Ex[\bw_i \otimes \bw_i] \|_{2} \le \gamma^2$.
Let $\bpi_i \in \{1, \NA \}^p$, where $\NA$ denotes a missing value, be a vector of independent Bernoulli variables with parameter $\rho \in (0,1]$. 
Further, let $\bvarepsilon_i$, $\bw_i$, $\bpi_i$ be mutually independent. 
\end{assumption} 
Consistent with standard assumptions in the error-in-variables (EiV) regression literature, Assumption~\ref{assumption:cov_noise} posits the idiosyncratic EiV vector-valued noise $\bw_i$ to be subgaussian and independent across measurements; note, however, that the noise is allowed to be dependent within a measurement, i.e., the coordinates of $\bw_i$ can be correlated.
Finally, we require missing entries in the observed covariate vector to be missing completely at random (MCAR). 
In Section \ref{sec:het_missigness}, we discuss ways to allow for more heterogeneous missingness patterns. 

\begin{assumption}[Bounded covariates] \label{assumption:bounded}  
Let $\norm{\bX}_{\max} \le 1$ and $\norm{\bXp}_{\max}  \le 1$.
\end{assumption}
Assumption~\ref{assumption:bounded} bounds the magnitude of the underlying noiseless covariates, not the {\em observed} noisy covariates. 
This assumption is made to simplify our analysis and it can be generalized to hold for any $C$ that is an absolute constant.
Our theoretical results will correspondingly only change by an absolute constant as well.

%% file: content/results.tex
\section{Main Results} \label{sec:results} 
This section responds to Q1--Q2. 
For ease of notation, let $C, c >0$ be absolute constants whose values may change from line to line or even within a line. 
Let $\bH = \bX^\dagger \bX \in \Rb^{p \times p}$ and $ \bH_\perp = \bI - \bH$ denote the projection matrices onto the rowspace and nullspace of $\bX$, respectively. 
Let $\bHp, \bHp_\perp \in \Rb^{p \times p}$ be defined analogously with respect to $\bXp$. 
We define $\btbeta^* = \bH \bbeta^*$ as the projection of $\bbeta^*$ onto the linear space spanned by the rows of $\bX$. 

%
\subsection{Model Identification} \label{sec:param_est_results}
\vspace{2mm}
\begin{tcolorbox}[colback=gray!10!white,colframe=black!75!black]
	\begin{center}
	{
		{\bf Q1:} ``When $p > n$, is there a model parameter that PCR consistently identifies?''
	} 
	\end{center} 
\end{tcolorbox}

The model parameter $\bbeta^*$ is not identifiable in the high-dimensional regime as infinitely many solutions satisfy \eqref{eq:response_vector}. 
Among all feasible parameters, we show that PCR recovers $\btbeta^*$, the unique parameter with minimum $\ell_2$-norm that is entirely embedded in the rowspace of $\bX$, provided the number of principal components $k$ is aptly chosen.

From Property \ref{property:equiv}, recall that PCR enforces $\bhbeta \in \text{rowspan}(\btZ^k)$.
Hence, if $k=r$ and the rowspace of $\btZ^r$ is ``close'' to the rowspace of $\bX$, then $\bhbeta \approx \btbeta^*$. 
The ``noise'' in $\bZ$ arises from the missingness pattern induced by $\bpi$ and the measurement error $\bW$; meanwhile, the ``signal'' in $\bZ$ arises from $\bX$, where its strength is captured by the magnitude of its singular values. 
Accordingly, we define the $\snr$ as 
\begin{align}\label{eq:snr}
\snr & \coloneqq \frac{\rho s_r}{\sqrt{n} + \sqrt{p}}. 
\end{align}
Here, $s_r$ is the smallest nonzero singular value of $\bX$, $\rho$ determines the fraction of observed entries, 
and $\sqrt{n} + \sqrt{p}$ is induced by the perturbation in the singular values from $\bW$.
As one would expect, the signal strength $s_r$ scales linearly with $\rho$. 
From standard concentration results for sub-gaussian matrices, it follows that $\| \bW \|_2 = \tO( \sqrt{n} + \sqrt{p})$ (see Lemma~\ref{lem.perturb.1}).
With this notation, we state the main result on model identification. 

\begin{thm} \label{thm:param_est}
Let Assumptions \ref{assumption:response_noise}--\ref{assumption:bounded} hold. 
Consider 
(i) PCR with $k = r = \rank(\bX)$, 
(ii) $\rho \geq c (np)^{-1}\log^2 (np)$, 
and
(iii) $\snr \geq C(K+1)(\gamma+1)$.  
Then w.p. at least $1-O((np)^{-10})$, 
\begin{align} 
    \| \bhbeta - \btbeta^* \|_2^2 
    &\le 
    C_{\enoise} \log(np) 
    \cdot 
    \left\{
       	    \frac{\| \btbeta^* \|_2^2}{\snr^2}  
            + \frac{\sqrt{n} \| \btbeta^* \|_1}{(n \vee p) \snr^2} 
            + \frac{r (1 \vee \| \btbeta^* \|_1^2)}{(n \vee p) \snr^2}
       	   + \frac{\| \btbeta^* \|_1^2}{\snr^4}
    \right\},  
    \label{eq:thm1.main}
\end{align} 
where $C_\enoise = C (K+1)^4 (\gamma + 1)^2 (\sigma^2 + 1)$. 
Further, if $\langle \bx_i,\bbeta^* \rangle \in [-d, d]$ for all $i \le n$, then 
\begin{align}\label{eq:bound_linear_coeff}
\|\btbeta^*\|_2 \le s_r^{-1} \cdot d \sqrt{n}, 
\quad
\|\btbeta^*\|_1 \le s_r^{-1} \cdot d \sqrt{np}. 
\end{align}
\end{thm} 


%
\noindent 
{\em Interpretation.} 
We make a few remarks on Theorem~\ref{thm:param_est}. 
First, condition (iii) is not necessary but we impose it to simplify the parameter estimation bound in \eqref{eq:thm1.main}. 
Please refer to \eqref{eq:snr.simplify} in Appendix~\ref{sec:proof_thm_4.1} for details. 
We now briefly discuss why the $\ell_1$-norm of $\btbeta^*$ shows up in the bound.
Our analysis of the parameter estimation error involves an EiV error term of the form $\| (\bX - \btZ^k) \btbeta^*\|_2$, which can be bounded as follows: 
\begin{align}
	\| (\bX - \btZ^k) \btbeta^*\|_2 \le \| \bX - \btZ^k \|_{2,\infty} \|\btbeta^*\|_1. \label{eq:eiv_mixed_l1}
\end{align} 
See \eqref{eq:thm1.2.5} in Appendix~\ref{sec:proof_thm_4.1} for details. 
%
%
From \eqref{eq:bound_linear_coeff}, it is clear that $\|\btbeta^*\|_1$ is controlled if $s_r$ is sufficiently large.
Indeed, Assumption \ref{assumption:balanced_spectra} below is one such natural condition on $s_r$. 

To gain a better view on Theorem~\ref{thm:param_est} regarding consistency, let us suppress dependencies on $(K, \gamma, \sigma)$ for the following discussion.  
Theorem~\ref{thm:param_est} implies that a sufficient condition for consistency is given by 
\begin{align}
 \frac{\snr^2}
 {
  \log(np) \cdot \max\{\| \btbeta^* \|_2^2,
	~n^{1/2} (n \vee p)^{-1}  \| \btbeta^* \|_1,
	~r  (n \vee p)^{-1} (1 \vee \| \btbeta^* \|_1^2),
	~\| \btbeta^* \|_1\}
  } \rightarrow \infty. 
\end{align}
That is, PCR recovers $\btbeta^*$ provided $\snr$ grows sufficiently fast. 
Finally, \eqref{eq:bound_linear_coeff} implies that \eqref{eq:thm1.main} can be purely expressed through the smallest nonzero singular value of $\bX$. 

We now describe a natural setting for which we can provide an explicit bound on the $\snr$.
Towards this, we introduce the following assumption and discuss its meaning in Section~\ref{sec:eiv_lit}.

\begin{assumption}[Balanced spectra: in-sample covariates]\label{assumption:balanced_spectra}
The $r$ nonzero singular values $s_i$ of $\bX$ satisfy $s_i = \Theta(\sqrt{np/r})$. 
\end{assumption}

\begin{cor}\label{cor:param_est}
Let the setup of Theorem \ref{thm:param_est} and Assumption \ref{assumption:balanced_spectra} hold. 
If $\langle \bx_i,\bbeta^* \rangle \in [-d, d]$ for all $i \le n$, then w.p. at least $1-O((np)^{-10})$, 
\begin{align}\label{eq:thm1.main.1}
    \| \bhbeta - \btbeta^*\|_2^2 & \leq 
    C_\enoise \log (np) 
    \cdot 
    \left\{
        \frac{d r^{3/2}}{\rho^2 \sqrt{n} p} + \frac{d^2 r^3}{\rho^4 (n \wedge p)^2}
    \right\}. 
\end{align}
\end{cor}

\begin{proof}
By Assumption \ref{assumption:balanced_spectra}, we have $s_r = \Theta(\sqrt{np/r})$. 
This yields
\begin{align}
    \snr & = \frac{\rho s_r}{\sqrt{n} + \sqrt{p}} 
    ~\geq \frac{c \rho \sqrt{np}}{\sqrt{r(n + p)}} 
    ~\geq c\rho \sqrt{(n \wedge p) / r}, 
\end{align}
i.e., $\snr = \Omega(\rho \sqrt{(n \wedge p) / r} )$. 
Further, we have from \eqref{eq:bound_linear_coeff} that 
\begin{align}
\|\btbeta^*\|_2 \le d \sqrt{r / p}, \quad \|\btbeta^*\|_1 \le d \sqrt{r}. \label{eq:well_balanced_coeff}
\end{align}
Inserting \eqref{eq:well_balanced_coeff} into \eqref{eq:thm1.main} and simplifying completes the proof.
\end{proof}
Ignoring dependencies on $(\rho, r, d)$, Corollary \ref{cor:param_est} implies that the model identification error scales as $\min\{1 / \sqrt{n}p,  1 / (n \wedge p)^2\}$. Hence, the error vanishes as $\min\{n, p\} \to \infty$.
The requirement that $p$ grows arises from the error-in-variables problem; more specifically, in the PCA subroutine, we show that $\btZ^k$ is a good estimate of $\bX$ provided both $n$ and $p$ grow (see Lemmas~\ref{eq:thm1.lem1.f2.0} and \ref{lem:thm1.2} in Appendix~\ref{sec:proof_thm_4.1} for details). 
%

%

%
\subsection{Out-of-sample Prediction} \label{sec:test_error_results} 
\vspace{2mm}
\begin{tcolorbox}[colback=gray!10!white,colframe=black!75!black]
	\begin{center}
	{
		{\bf Q2:} ``Given deterministic, corrupted, and partially observed out-of-sample covariates, can PCR recover the expected responses?''
	} 
	\end{center} 
\end{tcolorbox}
Towards answering Q2, we define PCR's out-of-sample (test) prediction errors with respect to $\bhy$ and $\bhy^\trunc$ as 
\begin{align}\label{eq:mse_test}
	&\text{MSE}_{\test} \coloneqq \frac{1}{m}\sum^m_{i = 1}(\hy_{n+i} - \langle \bx_{n+i}, \bbeta^*\rangle)^2
	\\
	&\text{MSE}^\trunc_{\test} \coloneqq \frac{1}{m}\sum^m_{i = 1}(\hy^\trunc_{n+i} - \langle \bx_{n+i}, \bbeta^*\rangle)^2, 
\end{align}
respectively. 
Let $s_{\ell}, s'_{\ell} \in \Rb$ be the $\ell$-th largest singular values of $\bX$ and $\bXp$, respectively.
Recall from Section \ref{sec:alg} that $\hs_\ell, \hsp_\ell$ are defined analogously for $\btZ$ and $\btZp$, respectively. 
Analogous to \eqref{eq:snr}, we define a signal-to-noise ratio for the out-of-sample covariates as 
\begin{align}\label{eq:snr_test}
	\snr_{\test} &\coloneqq  \frac{\rho s'_{r'}}{\sqrt{m} + \sqrt{p}}.
\end{align}
Next, we bound $\text{MSE}_{\test}$ in probability and $\text{MSE}^\trunc_{\test}$ in expectation with respect to $\snr$ and $\snr_\test$.  
For ease of notation, we define $n_{\min} = n \wedge m$ and $n_{\max} = n \vee m$. 

\begin{thm} \label{thm:test_error} 
Let the setup of Theorem \ref{thm:param_est} hold with $\rho \geq c (mp)^{-1} \log^2 (mp)$.
Consider 
(i) PCR with $\ell = r' = \rank(\bXp)$
and
(ii) $\| \btbeta^* \|_1 = \Omega(1)$.  
Then w.p. at least $1- O((n_{\min} p)^{-10})$,  
\begin{align} \label{eq:test_error.1}
    \emph{MSE}_{\etest} 
    &\le \Delta_1 + \Delta_2, 
\end{align} 
where 
\begin{align}
    \Delta_1 &= 
    C \cdot p \cdot 
   \delta_\beta
    \cdot \| \bH_\perp  \bHp  \|_2^2, 
    \\ 	
    \Delta_2 
    &= C'_\enoise \log(n_{\max} p) 
    \cdot 
    \Bigg\{
    \frac{\sqrt{n}}{\snr^2}\| \btbeta^* \|_1
    + 
    \Delta_3
    \Bigg\}, 
    \\
    \Delta_3 &=     \left(
    	\frac{r(1 \vee p/m)}{\rho^2 \snr^2}
	+
	\frac{r'}{\snr^2_\etest \wedge m}
	+ 
     	\frac{ n \vee p }{\snr^4}
    \right) \| \btbeta^* \|_1^2;    
\end{align} 
here, $\delta_\beta$ is given by the righthand side of \eqref{eq:thm1.main} 
and $C'_\enoise = C (K+1)^6 (\gamma+1)^4 (\sigma^2 + 1)$. 
Further, if $\langle \bx_i,\bbeta^* \rangle \in [-b, b]$ for all $i > n$, then 
\begin{align} \label{eq:test_error.2}
\Ex [\emph{MSE}^\etrunc_{\etest}] &\le \Delta_1 + \Delta_4 + \Delta_5, 
\end{align} 
where
\begin{align}
    \Delta_4 &= C'_\enoise \log(n_{\max} p) 
    \cdot 
    \left\{
    \frac{\sqrt{n}}{\snr^2} \left( \frac{1}{\snr^2} + \frac{1 \vee p/m}{\rho^2 (n \vee p)} \right) \| \btbeta^* \|_1
    + \Delta_3 \right\},
    \\
    \Delta_5 &= \frac{Cb^2}{(n_{\min} p)^{10}}.
\end{align}
\end{thm} 

\noindent {\em Interpretation.} 
Let us briefly dissect Theorem~\ref{thm:test_error}. 
Firstly, condition (ii) is not necessary but made to simplify the resulting bound. 
On a more interesting note, it is well known that generalization error bounds rely on some notion of ``closeness'' between the in- and out-of-sample covariates. 
A canonical assumption within the statistical learning theory literature considers the two sets of covariates to be drawn from the same underlying distribution a la i.i.d. samples.  
As seen in \eqref{eq:test_error.1} and \eqref{eq:test_error.2}, we consider a complementary notion of covariate closeness that is captured by the term $\| \bH_\perp  \bHp  \|_2$ in $\Delta_1$. 
In words, it measures the size of the linear subspace spanned by the out-of-sample covariates that is not contained within the linear subspace spanned by the in-sample covariates. 
Effectively, this term quantifies the $\ell_2$-distance, or $\ell_2$-similarity, between the in- and out-of-sample covariates. 
If each out-of-sample covariate is some linear combination of the in-sample covariates, then this error term vanishes and the out-of-sample prediction error decreases. 
We formalize this concept in Assumption~\ref{assumption:subspace_incl} below. 

\begin{assumption} [Subspace inclusion] \label{assumption:subspace_incl} 
Let $\emph{rowspan}(\bXp) \subseteq \emph{rowspan}(\bX)$. 
\end{assumption} 

To aid our intuition of Assumption~\ref{assumption:subspace_incl}, consider \eqref{eq:response_vector} in the classical regime where $n > p$. 
The canonical assumption within this paradigm considers $\bX$ to have full column rank, i.e., $\rank(\bX) = p$. 
Accordingly, the in-sample covariates span $\Rb^p$ so the subspace spanned by the out-of-sample covariates necessarily lies within that spanned by the in-sample covariates, yielding $\| \bH_\perp  \bHp  \|_2 = 0$. 
In this view, Assumption~\ref{assumption:subspace_incl} generalizes the full column rank assumption in the classical regime to the collinear setting in the high-dimensional regime. 

\begin{cor} \label{cor:test_error.0}
Let the setup of Theorem \ref{thm:test_error} and Assumption \ref{assumption:subspace_incl} hold. 
Then, $\Delta_1 = 0$. 
\end{cor}

\begin{proof}
Under Assumption~\ref{assumption:subspace_incl}, we have $\| \bHp \bH_\perp \|_2^2 = 0$. 
\end{proof} 

For interpretability,  we suppress dependencies on $(K, \gamma, \sigma)$, and assume $p = \Theta(m)$ with $m \to \infty$.
One can then verify that Corollary \ref{cor:test_error.0} implies that sufficient conditions for PCR's expected test prediction error to vanish are
\begin{align}
	&\frac{\snr^2}
	{
	\log(n_{\max}p)
	\cdot
	\max\{
		n^{1/4} \| \btbeta^* \|_1^{1/2},
		~ (n \vee p)^{1/2} \| \btbeta^* \|_1
	\}
	} \rightarrow \infty, 
	\\
	&\frac{\rho^2 ~ \snr^2}
	{
	\log(n_{\max}p)
	\cdot
	\max\{
		~ n^{1/2} (1 \vee p/m) (n \vee p)^{-1} \| \btbeta^* \|_1,
		~ r (1 \vee p/m) \| \btbeta^* \|_1^2
	\}
	} \rightarrow \infty,
	\\
	 &\frac{\snr^2_{\test}}
	{
	\log(n_{\max}p) \cdot r' \| \btbeta^* \|_1^2
	} \rightarrow \infty.  
\end{align} 
As with Theorem \ref{thm:param_est}, we specialize Theorem \ref{thm:test_error} in Corollary \ref{cor:test_error} to the setting where $\snr = \Omega(\rho \sqrt{(n \wedge p)/r})$ and $\snr_\test = \Omega(\rho \sqrt{(m \wedge p)/r'})$.
A sufficient condition for the lower bound on $\snr_\test$ is provided in Assumption~\ref{assumption:balanced_spectra_test}. 

\begin{assumption}[Balanced spectra: out-of-sample covariates]\label{assumption:balanced_spectra_test}
The $\rp$ nonzero singular values $s'_i$ of $\bXp$ satisfy $s'_i = \Theta(\sqrt{mp / r'})$. 
\end{assumption}
%
\begin{cor}\label{cor:test_error}
Let the setups of Corollaries \ref{cor:param_est}--\ref{cor:test_error.0} and Assumption \ref{assumption:balanced_spectra_test} hold. 
Then w.p. at least $1- O((n_{\min} p)^{-10})$,
\begin{align}
	\emph{MSE}_{\etest} 
	&\le 
	C'_\enoise \log(n_{\max} p)
	\cdot
	\left\{
		\frac{d r^{3/2} \sqrt{n}}{\rho^2 (n \wedge p)}
		+
		\Delta
	\right\},
\end{align}
where
\begin{align}
	\Delta = \frac{d^2 r^3 (1 \vee p/m)}{\rho^4(n \wedge p)}
		+
		\frac{d^2 r^2}{\rho^2 (m \wedge p)}
		+
		\frac{d^2 r (n \vee p)}{\rho^4 (n \wedge p)^2}. 
\end{align}
Further, if $\langle \bx_i,\bbeta^* \rangle \in [-b, b]$ for all $i > n$, then 
\begin{align}
\Ex [\emph{MSE}^\etrunc_{\etest}] 
&\le 
C'_\enoise \log( n_{\max} p) \cdot 
\left\{
	\frac{d r^{5/2} \sqrt{n}}{\rho^4 (n \wedge p) (n \wedge m \wedge p) }
	+ \Delta
\right\} 
+ \frac{Cb^2}{(n_{\min} p)^{10}}.
\end{align}
\end{cor}

\begin{proof}
Using identical arguments to those made in the proof of Corollary \ref{cor:param_est} and noting $r' \le r$, it follows that Assumption~\ref{assumption:balanced_spectra_test} gives $\snr_{\test} \geq c\rho \sqrt{(m \wedge p) / r}$. 
Plugging the bounds on $\snr$, $\snr_\test$, and \eqref{eq:well_balanced_coeff} into Theorem \ref{thm:test_error} completes the proof. 
\end{proof}

For the following discussion, we suppress dependencies on 
$(K, \gamma, \sigma, r)$ and log factors, assume $\rho = \Theta(1)$, 
and only consider the scaling with respect to $(n, m, p)$.
Corollary \ref{cor:test_error} implies that if $p = o(n n_{\min})$ and $n = o(p^2)$,\footnote{Practically speaking, this condition is not binding. If $n = \Omega(p^2)$, then we can sample a subset of the training data to satisfy the condition. Hence, this condition is likely an artifact of our analysis.}
then the out-of-sample prediction error vanishes to zero both in expectation and w.h.p., as $n, m, p \rightarrow \infty$. 
If we make the additional assumption that $n = \Theta(p)$ and $p = \Theta(m)$, then the error scales as $\tO(1/n)$ in expectation.
This improves upon the best known rate of $\tO(1/\sqrt{n})$, established in \cite{robust_pcr, pcr_jasa}; 
notably, these works do not provide a high probability bound. 
Additionally, \cite{robust_pcr, pcr_jasa} require i.i.d. covariates to leverage standard Rademacher tools for their out-of-sample analyses.
In contrast, we consider fixed design points, thus our generalization error bounds do not rely on distributional assumptions regarding $\bX$ and $\bXp$. 
Finding the optimal relative scalings of $(n,m,p)$ to achieve consistency remains future work. 

\subsection{Discussion} 
%

\subsubsection{Heterogeneous Missingness Patterns}\label{sec:het_missigness}
Assumption~\ref{assumption:cov_noise} considers MCAR patterns in the observed covariate matrix $\bZ$. 
This is motivated by the HSVT subroutine of PCR, as discussed in Section~\ref{sec:zero_imputation}. 
If the missingness pattern is instead heterogeneous, other matrix completion methods designed for such settings can be utilized to more accurately recover the underlying covariates.
Matrix completion with heterogeneous missingness patterns is an active area of research and there has been a recent emergence of exciting results, including \cite{schnabelfwang16, ma2019missing, sportisse2020estimation_PCA} and \cite{bhattacharya2021matrix} to name a few. 

At a high-level, these algorithms follow a two-step approach: 
(i) construct estimates $\hrho_{ij}$ of $\rho_{ij}$;
(ii) use $\hrho_{ij}$ and $\bZ$ to estimate $X_{ij}$. 
With regards to step (i), let $\bPi \in \{0,1\}^{n \times p}$ denote the binary mask matrix with $\Ex[\pi_{ij}] = \rho_{ij}$. 
The common assumption driving these approaches is that $\Ex[\bPi]$ is a low-rank matrix; note if $\Ex[\pi_{ij}] = \rho$ (MCAR), then $\rank(\Ex[\bPi]) = 1$. 
As such, matrix completion algorithms can be first applied to $\bPi$ to obtain the estimates $\hrho_{ij}$. 
Then, $\bX$ can be estimated using $\hrho_{ij}$ and $\bZ$.
%
Within the context of this work, if the matrix completion algorithm can faithfully recover the underlying covariates, cf. Lemma~\ref{lem:thm1.2} of Appendix~\ref{sec:proof_thm_4.1}, then our main results in Section~\ref{sec:results} would naturally extend. 
A formal analysis of this more general estimator is left as interesting future work.


For the specific setting where there is a different probability of missingness $\{\rho_j\}_{j \in [p]}$ for each of the $p$ covariates, we propose a straightforward extension of PCR. 
Let $\hrho_j$ be the fraction of observed entries in the $j$-th column of $\bZ$. 
Let $\widehat{\bP} \in \Rb^{p \times p}$ be a diagonal matrix with the $j$-th diagonal element given by $\hrho_j$.
After setting the $\NA$ values of $\bZ$ to zero, we now redefine $\btZ$ as $\btZ = \bZ \widehat{\bP}$.
In words, rather than uniformly re-weighting the $\bZ$ by $1 / \hrho$, we now re-weight the $j$-th column of $\bZ$ by $1 / \hrho_j$.
As a result, our theoretical results will go through in an analogous manner with the scaling now depending on $\rho_{\min} = \min_{j \in [p]} \rho_j$. 


\subsubsection{PCR Theory with Misspecified Number of Principal Components} 
The results of this section rely on an oracle version of PCR that has access to the true ranks of $\bX$ and $\bXp$. 
We leave a formal treatment of PCR when the number of principal components is misspecified as an important future line of inquiry. 
With that said, we remark that the universal data-driven approach of \cite{Gavish_2014}, as mentioned in Section~\ref{sec:choose_k}, often performs remarkably well in practice. 
We apply this approach in our simulation studies on PCR's generalization performance in Sections~\ref{sec:distribution_shift}--\ref{sec:sim.mcar}. 

\subsubsection{Towards a Lower Bound on Model Identification} 
To the best of our knowledge, Theorem~\ref{thm:param_est} provides the first upper bound on PCR's model parameter estimation error in the high-dimensional EiV setting with fixed design. 
In Lemma~\ref{lemma:lower_bound} of Appendix~\ref{sec:lower_bound}, we take the first step towards establishing a complementary lower bound to better understand the limitations of PCR in such a setting. 
%

\subsubsection{Viewing Generalization through Assumption~\ref{assumption:subspace_incl}} 
As discussed, our out-of-sample guarantees do {\em not} rely on any distributional assumptions between the in- and out-of-sample covariates.
Rather, our results rely on a purely linear algebraic condition given by Assumption \ref{assumption:subspace_incl}.
In this view, Assumption \ref{assumption:subspace_incl} offers a complementary, distribution-free perspective on generalization and has possible implications to learning under covariate shifts. 
We examine the role of Assumption \ref{assumption:subspace_incl} in our simulations in Section~\ref{sec:simulations}. 
As a preview, our results provide empirical evidence that PCR can generalize even when the in- and out-of-sample covariates obey different distributions provided Assumption \ref{assumption:subspace_incl} holds. 
In light of these findings, we furnish a data-driven diagnostic in Section~\ref{sec:hypo} to check when Assumption~\ref{assumption:subspace_incl} may hold in practice.

%% file: content/exp.tex
\section{Illustrative Simulations} \label{sec:simulations} 
In this section, we present illustrative simulations to support our theoretical results. 
We provide details of the simulations in Appendix~\ref{sec:simulation_details}. 

\subsection{PCR Identifies the Minimum $\ell_2$-norm Model Parameter}
\label{sec:experiment_param_exp} 
To see how Theorem~\ref{thm:param_est} plays out in practice, we design a simulation on model identification. 

\medskip 
\noindent 
{\em Setup.} 
We consider $p = 512$ and $r = 15$. 
We generate $\bbeta^*$ and set it to have unit norm. 
For each $n \in \{30, 98, 167, \dots, p\}$, we generate the $\bX$ and define the minimum $\ell_2$-norm solution as $\btbeta = \bX^\dagger \bX \bbeta^*$. 
We conduct $1000$ simulation repeats per sample size $n$. 
For each repeat, we sample $(\bvarepsilon, \bW)$ to construct $\by = \bX \bbeta^* + \bvarepsilon$ and $\bZ = \bX + \bW$. 


\medskip 
\noindent 
{\em Results.} 
For each simulation repeat, we apply PCR on $(\by, \bZ)$ to learn a {\em single} $\bhbeta$ with $k = r$ chosen correctly. 
Figure~\ref{fig:pcr_param_est} visualizes the root-MSE (RMSE) of $\bhbeta$ with respect to $\btbeta^*$ and $\bbeta^*$. 
As predicted by Theorem \ref{thm:param_est}, the RMSE with respect to $\btbeta^*$ decays to zero as the sample size increases. 
In contrast, the RMSE with respect to $\bbeta^*$ stays roughly constant across different sample sizes. 
This reaffirms that PCR identifies the minimum $\ell_2$-norm solution amongst all feasible solutions. 

\begin{figure}
	\centering
	\includegraphics[width=0.45\linewidth]{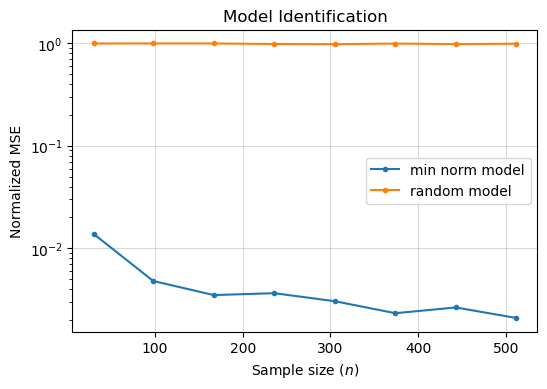}
	\caption{ \smaller
	    Plots of $\ell_2$-norm errors (log scale) for two cases: (i) $\| \bhbeta - \btbeta^* \|_2$ and (ii) $\|\bhbeta - \bbeta^* \|_2$. The error decays for case (i) but remains stagnant for case (ii).
	}
	\label{fig:pcr_param_est}
\end{figure}

\subsection{PCR is Robust to Covariate Shifts}
\label{sec:distribution_shift} 
We study the PCR's generalization properties, as predicted by Theorem~\ref{thm:test_error}, in the presence of covariate shifts, i.e., the in- and out-of-sample covariates follow different distributions. 

\medskip 
\noindent 
{\em Setup.} 
We adopt the same considerations on $(n, p, r)$ and generate $\bbeta^*$ as in Section~\ref{sec:experiment_param_exp}.  
Let $m = n$. 
For each $n$, we generate $\bX$ as per distribution $\Dc_1$. 
We then generate four different out-of-sample covariates as follows: 
(i) $\bXp_1 \sim \Dc_1$,
(ii) $\bXp_2 \sim \Dc_2$,
(iii) $\bXp_3 \sim \Dc_3$,
and (iv) $\bXp_4 \sim \Dc_4$, where $\Dc_2, \Dc_3, \Dc_4$ are distinct distributions from one another and from $\Dc_1$. 
Critically, Assumption~\ref{assumption:subspace_incl} is satisfied between $\bX$ and $\bXp_i$ for every $i \in \{1, \dots, 4\}$. 
We define $\btheta'_i = \bXp_i \bbeta^*$. 
We conduct $1000$ simulation repeats. 
For each repeat, we sample $(\bvarepsilon, \bW, \bWp)$ to construct $\by = \bX \bbeta^* + \bvarepsilon$, $\bZ = \bX + \bW$, and $\bZp_i = \bXp_i + \bWp$. 

\medskip 
\noindent 
{\em Results.} 
For each simulation repeat, we apply PCR on $(\by, \bZ)$ to learn a single $\bhbeta$ by choosing $k$ via the universal data-driven approach of \cite{Gavish_2014}. 
For each $i$, we construct $\bhyp_i$ from the de-noised version of $\bZp_i$ and $\bhbeta$. 
Figure \ref{fig:pcr_distribution_shift} displays the MSE of $\bhyp_i$ with respect to $\btheta'_i$. 
As predicted by Corollary~\ref{cor:test_error}, the out-of-sample prediction error decays as the sample size increases for each covariate shift. 
Hence, our results provide further evidence that PCR is robust to corrupted out-of-sample covariates and, perhaps more importantly, covariate shifts provided Assumption \ref{assumption:subspace_incl} holds. 

\begin{figure}
	\centering
	\includegraphics[width=0.45\linewidth]{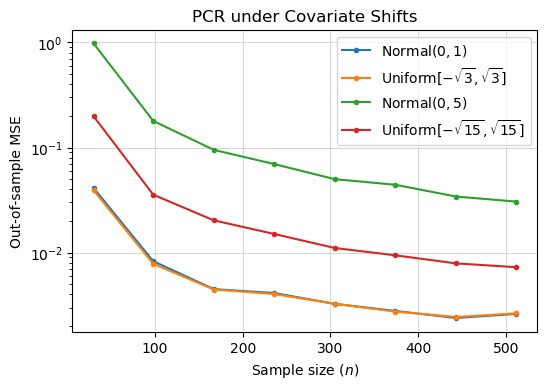}
	\caption{ \smaller
	    Plot of PCR's MSE (log scale) under various covariate shifts with Assumption \ref{assumption:subspace_incl} satisfied in each case. The MSE decays as the sample size increases for each covariate shift.
	}
	\label{fig:pcr_distribution_shift}
\end{figure}

\subsection{PCR Generalizes under Assumption \ref{assumption:subspace_incl}} \label{sec:sim.subspace_incl}
This simulation further examines the role of Assumption \ref{assumption:subspace_incl}. 
Specifically, we compare PCR's generalization error under two settings: (i) there is covariate shift but Assumption \ref{assumption:subspace_incl} holds; (ii) there is distributional invariance (i.e., the in- and out-of-sample covariates obey the same distribution) but Assumption \ref{assumption:subspace_incl} is violated. 

\medskip 
\noindent 
{\em Setup.} 
We adopt the same considerations on $(n, m, p, r)$ and generate $\bbeta^*$ as in Section~\ref{sec:distribution_shift}. 
For each $n$, we generate $\bX \sim \Dc_1$. 
We then generate two out-of-sample covariates: 
(i) $\bXp_1 \sim \Dc_1$ that violates Assumption~\ref{assumption:subspace_incl}; 
(ii) $\bXp_2 \sim \Dc_2$ with $\Dc_2 \neq \Dc_1$ that obeys Assumption~\ref{assumption:subspace_incl}. 
Next, we define $\btheta'_i = \bXp_i \bbeta^*$ for $i \in \{1,2\}$. 
We conduct $1000$ simulation repeats. 
For each repeat, we sample $(\bvarepsilon, \bW, \bWp)$ to construct $\by = \bX \bbeta^* + \bvarepsilon$, $\bZ = \bX + \bW$, and $\bZp_i = \bXp_i + \bWp$. 

\medskip 
\noindent 
{\em Results.} 
For each simulation repeat, we apply PCR on $(\by, \bZ)$ to learn a single $\bhbeta$ by choosing $k$ via the universal data-driven approach of \cite{Gavish_2014}. 
For each $i$, we construct $\bhyp_i$ from the de-noised version of $\bZp_i$ and $\bhbeta$. 
Figure \ref{fig:subspace_v_iid} displays the MSE of $\bhyp_i$ with respect to $\btheta'_i$. 
When Assumption~\ref{assumption:subspace_incl} holds, the MSE decays as the sample size increases; by contrast, when Assumption~\ref{assumption:subspace_incl} fails, the MSE is stagnant across varying sample sizes. 
Our findings reinforce the importance of Assumption~\ref{assumption:subspace_incl} for PCR's ability generalize. 
%

\begin{figure}
	\centering
	\includegraphics[width=0.45\linewidth]{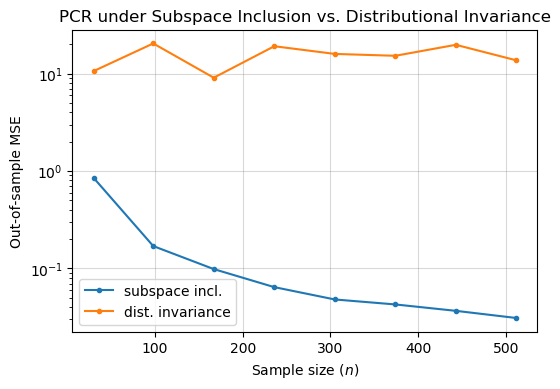}
	\caption{
	   { \smaller
	    Plots of PCR's MSE (log scale) under two cases: (i) Assumption \ref{assumption:subspace_incl} holds but distributional invariance is violated (blue); (ii) Assumption \ref{assumption:subspace_incl} is violated but distributional invariance holds (orange).
	    Case (i) achieves a vanishing MSE while case (ii) suffers from non-vanishing MSE.}
	}
	\label{fig:subspace_v_iid}
\end{figure}

\subsection{PCR Generalizes with MCAR Entries} \label{sec:sim.mcar} 
This simulation investigates PCR's out-of-sample performance under varying intensities of MCAR patterns in the observed covariate matrices. 

\medskip 
\noindent 
{\em Setup.} 
We adopt the same considerations on $(n, m, p, r)$ and generate $\bbeta^*$ as in Section~\ref{sec:distribution_shift}. 
For each $n$, we generate $\bX, \bXp \sim \Dc_1$ with Assumption~\ref{assumption:subspace_incl} satisfied.  
Next, we define $\btheta' = \bXp \bbeta^*$. 
We consider varying intensities of MCAR entries with $\rho \in \{0.4, 0.6, 0.8, 0.99\}$.
We conduct $1000$ simulation repeats for each $(\rho, n)$ pair. 
For each repeat, we sample $(\bvarepsilon, \bW, \bWp, \bPi, \bPip)$ to construct $\by = \bX \bbeta^* + \bvarepsilon$, $\bZ = (\bX + \bW) \circ \bPi$, and $\bZp_i = (\bXp_i + \bWp) \circ \bPip$. 
Note that there are $\rho$ entries in $\bPi, \bPip$ that are randomly assigned the value $1$, and each iteration considers a different permutation of revealed entries. 

\medskip 
\noindent 
{\em Results.} 
For each simulation repeat, we apply PCR on $(\by, \bZ)$ to learn $\bhbeta$ by choosing $k$ via the universal data-driven approach of \cite{Gavish_2014}. 
We construct $\bhyp$ from the de-noised version of $\bZp$ and $\bhbeta$. 
Figure \ref{fig:mcar} displays the MSE of $\bhyp$ with respect to $\btheta'$. 
Across varying intensities of $\rho$, the MSE decays as the sample size increases, which suggests that PCR can generalize when entries in the observed covariate matrices are MCAR. 

\begin{figure}
	\centering
	\includegraphics[width=0.45\linewidth]{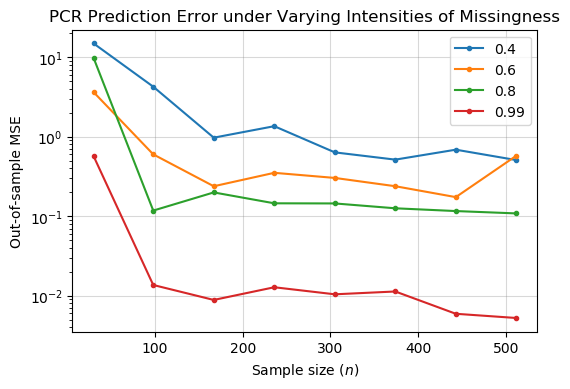}
	\caption{
	   { \smaller
	    Plots of PCR's MSE (log scale) under varying intensities of MCAR as controlled by $\rho$. The MSE decays as the sample size increases for each value of $\rho$.}
	}
	\label{fig:mcar}
\end{figure}

%% file: content/hypo.tex
\section{A Hypothesis Test for Assumption \ref{assumption:subspace_incl}} \label{sec:hypo} 
Our theoretical and empirical results highlight the importance of Assumption~\ref{assumption:subspace_incl}. 
Accordingly, we present a hypothesis test to check when Assumption~\ref{assumption:subspace_incl} holds in practice. 
Recall the definitions of $(\bH, \bH_\perp)$ and $(\bHp, \bHp_\perp)$ as defined at the start of Section~\ref{sec:results}. 

We consider the hypotheses
\begin{align}
	H_0: ~ \text{rowspan}(\bXp) \subseteq \text{rowspan}(\bX)
	\quad \text{and} \quad 
	H_1: ~ \text{rowspan}(\bXp) \nsubseteq \text{rowspan}(\bX).
\end{align} 
Since $(\bX, \bXp)$ are unobserved, we use $(\bZ, \bZp)$ as proxies. 
To this end, let $\bhH^k$ and $\bhH^{\prime \ell}$ denote the projection matrices formed by the right singular vectors of $\btZ^k$ and $\btZ^{\prime \ell}$, respectively; see Section~\ref{sec:pcr_description} for a recall of relevant notation.  
We then define our test statistic as 
\begin{align}
	\htau &= \| (\bI - \bhH^k) \bhH^{\prime \ell} \|_F^2. 
\end{align}
In words, $\htau$ measures the $\ell_2$-distance between the in- and out-of-sample covariates represented by the rowspaces of $\btZ^k$ and $\btZ^{\prime \ell}$, respectively. 

We define the test as follows: for any significance level $\alpha \in (0,1)$ and corresponding critical value $\tau(\alpha)$, retain $H_0$ if $\htau \le \tau(\alpha)$ and reject $H_0$ if $\htau > \tau(\alpha)$. 
%
%
In Sections~\ref{sec:hypo.theory} and \ref{sec:hypo_practical} below, we discuss two approaches to perform the hypothesis test.

\subsection{A Theory-Based Approach} \label{sec:hypo.theory}
We first provide a theory-based approach in defining $\tau(\alpha)$. 
Formally, let 
\begin{align}\label{eq:crit_value}
\tau(\alpha) &= 
r' \left( 
\frac{C\varsigma^2 \phi^2 (\alpha/2)}{s^2_{r}}
+ \frac{C\varsigma^2 (\phi' (\alpha/2))^2}{(s'_{r'})^2}
+ \frac{C\varsigma \phi (\alpha/2)}{s_{r}}
\right), 
\end{align}
where $C \ge 0$ is an absolute constant, 
$\text{Var}(w_{ij}) \le \varsigma^2$, 
$\phi(a) = \sqrt{n} + \sqrt{p} + \sqrt{\log(1/a)}$;
$\phi'(a) = \sqrt{m} + \sqrt{p} + \sqrt{\log(1/a)}$;  
and 
recall that $s_\ell, s'_\ell$ are the $\ell$-th largest singular values of 
$\bX$ and $\bXp$, respectively. 
See Appendix~\ref{sec:proof.hypo_test} for the derivation of \eqref{eq:crit_value}. 

\subsubsection{Type I and Type II Guarantees} 
Given our choice of $\htau$ and $\tau(\alpha)$, we control both Type I and Type II 
errors of our test. 
For ease of exposition, we will consider a more restrictive form of Assumption~\ref{assumption:cov_noise}, namely the entries of the covariate noise are independent and $(\bZ, \bZp)$ are fully observed. 

\begin{thm} \label{thm:hypo}
Consider Assumption~\ref{assumption:cov_noise} with (i) the entries of $\{\bw_i: i \le n + m\}$ as independent random variables satisfying $\emph{Var}(w_{ij}) = \varsigma^2$ and (ii) $\rho = 1$. 
Suppose $k = r$ and $\ell = r'$.
Fix any $\alpha \in (0,1)$. 
Then, there exists an absolute constant $C \ge 0$, defined in \eqref{eq:crit_value}, such that 
the Type I error is bounded as 
$
\Pb(\htau > \tau(\alpha) | H_0) \le \alpha. 
$
To bound the Type II error, suppose the additional condition holds:
\begin{align} \label{eq:type2_cond} 
    r' > \| \bH \bHp \|_F^2 
    + 2 \tau(\alpha) 
    + 
    \frac{C\varsigma r' \phi'(\alpha/2)}{s'_{r'}}.
\end{align}
Then, the Type II error is bounded as 
$
\Pb(\htau \le \tau(\alpha) | H_1) \le \alpha.
$
\end{thm}
The particular $C$ for which Theorem \ref{thm:hypo} holds depends on the underlying distribution of the covariate noise $\bw_i$. 
$C$ can be made explicit for certain classes of distributions; as an example, Corollary \ref{cor:hypo} specializes Theorem \ref{thm:hypo} to when $\bw_i$ are normally distributed. 
\begin{cor}\label{cor:hypo}
Consider the setup of Theorem \ref{thm:hypo} with $C=4$.  
Let $\bw_i$ be normally distributed for all $i \le n + m$. 
Then, $\Pb(\htau > \tau(\alpha) | H_0) \le \alpha$ and  $\Pb(\htau \leq \tau(\alpha) | H_1) \le \alpha.$
\end{cor}
We now argue \eqref{eq:type2_cond} is not a restrictive condition.
Conditioned on $H_1$, observe that $ r' > \|\bH \bHp \|_F^2$ always holds. 
If Assumptions~\ref{assumption:balanced_spectra} and \ref{assumption:balanced_spectra_test} hold, then one can easily verify that the latter two terms on the right-hand side of \eqref{eq:type2_cond} decay to zero as $(n, m, p)$ grow.

\subsubsection{Computing the Critical Value} 
Computing $\tau(\alpha)$ requires estimating (i) $\varsigma^2$;  (ii) $r, r'$; (iii) $s_{r}, s'_{r'}$. 
Under our assumptions, the covariance of $\bw$ can be estimated from the sample covariance matrices of $(\bZ, \bZp)$. 
By standard random matrix theory, the singular values of $\bZ$ and $\bX$ are close. 
Thus, as discussed in Section~\ref{sec:choose_k}, the spectrum of $\bZ$ serves as a good proxy to estimate $(r, s_r)$. 
Analogous arguments hold for $\bZp$ with respect to $\bXp$. 
Corollary~\ref{cor:critical_value} specializes $\tau(\alpha)$ under Assumptions~\ref{assumption:balanced_spectra} and \ref{assumption:balanced_spectra_test}. 
\begin{cor}\label{cor:critical_value}
Let the setup of Theorem \ref{thm:hypo} hold.
Suppose Assumptions~\ref{assumption:balanced_spectra} and \ref{assumption:balanced_spectra_test} hold.
Then, $\tau(\alpha) = O\Big( \frac{\sqrt{\log(1 / \alpha)}}{\min\{\sqrt{n}, \sqrt{m}, \sqrt{p} \}} \Big).$
\end{cor} 
If we consider the noiseless case, $\bw_i = \bzero$, then $\tau(\alpha) = 0$.
More generally, if the spectrum of $\bX$ and $\bXp$ are well-balanced, then Corollary~\ref{cor:critical_value} establishes that $\tau(\alpha) = o(1)$, even in the presence of noise. 
We remark that Corollary \ref{cor:hypo} allows for exact constants in the definition of $\tau(\alpha)$ under the Gaussian noise model.

\subsection{A Practical Approach} \label{sec:hypo_practical} 
We now provide a practical approach to computing $\tau(\alpha)$. 
To build intuition, observe that $\htau$ represents the remaining spectral energy of $\bHp$ not contained within $\bH$. 
Further, we note $\htau$ is trivially bounded by $r'$. 
Thus, one can fix some fraction $\alpha \in (0,1)$ and reject $H_0$ if $\htau > \tau(\alpha)$, where $\tau(\alpha) = r' \alpha$. 
In words, if more than $\alpha$ fraction of the spectral energy of $\bHp$ lies outside the span of $\bH$, then the alternative test rejects $H_0$. 
We remark that this variant is likely more robust compared to its exact computation counterpart in \eqref{eq:crit_value}, which requires estimating several ``nuisance'' quantities and varies with the underlying modeling assumptions on the covariate noise and singular values. 
Accordingly, without knowledge of these quantities, we recommend the practical approach. 
To see how the practical heuristic plays out in practice, see Section~\ref{sec:sc_studies} and \cite{broad}.

%% file: content/synthetic_controls.tex
\section{Synthetic Controls} \label{sec:synthetic_controls}
This section contextualizes our results in Section~\ref{sec:results} for synthetic controls \citep{abadie1, abadie2}, which has emerged as a leading approach for policy evaluation with observational data \citep{athey_imbens}.  
Towards this, we connect synthetic controls to (high-dimensional) error-in-variables regression with fixed design. 

\subsection{Synthetic Controls Framework}
%
%
Consider a panel data format where observations of $p+1$ units, indexed as $\{0, \dots, p\}$, are collected over $n+m$ time periods. 
Each unit $i$ at time $t$ is characterized by two potential outcomes, $Y_{ti}(1)$ and $Y_{ti}(0)$, corresponding to the outcomes under treatment and absence of treatment (i.e., control), respectively \citep{neyman, rubin}. 
For each unit, we observe their potential outcomes according to their treatment status, i.e., we either observe $Y_{ti}(0)$ or $Y_{ti}(1)$, never both. 
%
%
Let $Y_{ti}$ denote the observed outcome. 
For ease of exposition, we consider a single treated unit indexed by the zeroth unit and referred to as the target. 
We refer to the remaining units as the control group. 

We observe all $p+1$ units under control for the first $n$ time periods. 
In the remaining $m$ time periods, we continue to observe the control group without treatment but observe the target unit with treatment. 
Precisely, 
\begin{align}
	Y_{ti} = \begin{cases}
		Y_{ti}(0) &\text{for all } t \le n \text{ and } i \ge 0,
		\\
		Y_{ti}(0) &\text{for all } t > n \text{ and } i \ge 1,
		\\
		Y_{ti}(1) &\text{for all } t > n \text{ and } i = 0. 
	\end{cases}
\end{align}
We call the first $n$ and final $m$ time steps the pre- and post-treatment periods, respectively. 
We encode the control units' pre- and post-treatment observations into $\bZ = [Y_{ti}: t \le n, i \ge 1] \in \Rb^{n \times p}$ and $\bZp = [Y_{ti}: t > n, i \ge 1] \in \Rb^{m \times p}$, respectively. 
We encode the target unit's pretreatment observations into $\by = [Y_{t0}: t \le n] \in \Rb^n$. 
With these concepts in mind, we connect the synthetic controls framework to our setting of interest. 

\subsubsection{Out-of-Sample Prediction}
Synthetic controls tackles the counterfactual question: ``what would have happened to the target unit in the absence of treatment?''
Formally, the goal is to estimate the (expected) counterfactual vector $\Ex[\by'(0)]$, where $\by'(0) = [Y_{t0}(0): t > n] \in \Rb^m$.
Methodologically, this is answered by regressing $\by$ on $\bZ$ and applying the regression coefficients $\bhbeta$ to $\bZp$ to  
estimate the treated unit's expected potential outcomes under control during the post-treatment period.
%
From this perspective, we identify that counterfactual estimation is precisely out-of-sample prediction. 

\subsubsection{Error-in-Variables} 
As is typical in panel studies, potential outcomes are modeled as the addition of a latent factor model and a random variable in order to model measurement error and/or misspecification \citep{abadie_survey}. 
That is,
$Y_{t0}(0) = \langle \bu_t, \bv_i \rangle + \varepsilon_{t0}$,
where $\bu_t, \bv_i \in \Rb^r$ represent latent time and unit features with $r$ much smaller than $(n, m, p)$, and $\varepsilon_{t0} \in \Rb$ models the stochasticity.
This is also known as an interactive fixed effects model \citep{bai_interactive_fixed_effects}.
Put differently, the observed matrices $\bZ$ and $\bZp$ are viewed as noisy instantiations of $\bX = \Ex[\bZ]$ and $\bXp = \Ex[\bZp]$, where $\bX, \bXp$ are low-rank matrices. 
They represent the matrices of latent expected potential outcomes, which are a function of the latent time and unit factors. 
Since $\bhbeta$ is learned using $\bZ$ not $\bX$, synthetic controls is an instance of error-in-variables regression. 

\begin{remark} [Clarifying MCAR entries] 
As described in Section~\ref{sec:setup_eiv_regression}, we allow the entries in $\bZ$ and $\bZp$ to be missing completely at random (MCAR). 
We emphasize that these missing elements do {\em not} correspond to our counterfactual estimands of interest.  
Readers who find the MCAR setting to be implausible can proceed with the balanced panel data setting in mind. 
\end{remark}

\subsubsection{Linear Model} 
The underlying premise behind synthetic controls is that the target unit is a weighted composition of control units. 
In our setup, this translates more formally as the existence of a linear model $\bbeta^* \in \Rb^p$ satisfying 
\begin{align} \label{eq:sc.eq}
	\Ex[Y_{t0}(0)] = \sum_{i \ge 1} \beta_i^* \Ex[Y_{ti}(0)] \implies Y_{t0}(0) = \sum_{i \ge 1} \beta_i^* \Ex[Y_{ti}(0)] + \varepsilon_{t0} 
\end{align}  
for every $t \in [n+m]$, i.e., $\by = \bX \bbeta^* + \bvarepsilon$, where $\bvarepsilon = [\varepsilon_{t0}: t \le n]$.
%
%
We note that \citep{pcr_jasa} establish that such a $\bbeta^*$ exists w.h.p. if $r$ is much smaller than $(n, m, p)$.

\subsubsection{Fixed Design} 
Several works in the literature, e.g., \cite{pcr_jasa}, enforce the latent time factors to be sampled i.i.d.
Subsequently, the pre- and post-treatment data under control are also i.i.d. 
In contrast, we consider a fixed design setting that avoids distributional assumptions on the expected potential outcomes. 
This allows us to model settings with underlying time trends or shifting ideologies, which are likely present in many panel studies. 

\subsection{Novel Guarantees for the Synthetic Controls Literature}
With our connection established, we transfer our theoretical results to the synthetic controls framework. 
In particular, we analyze the robust synthetic controls (RSC) estimator of \cite{rsc} and its extension in \cite{mrsc}, which learns $\bhbeta$ via PCR. 

\subsubsection{Model Identification} 
Intuitively, $\bbeta^*$ defines the synthetic control group. 
That is, the magnitude (and sign) of the $i$th entry specifies the contribution of the $i$th control unit in the construction of the target unit. 
Theorem~\ref{thm:param_est} establishes that RSC consistently identifies the unique synthetic control group with minimum $\ell_2$-norm. 
%


\subsubsection{Counterfactual Estimation} 
We denote RSC's estimate of the expected counterfactual trajectory as $\widehat{\Ex}[\by'(0)] = [\widehat{\Ex}[Y_{t0}(0)]: t > n]$.
The counterfactual estimation error is then 
\begin{align}\label{eq:sc_error}
\frac{1}{m} \| \widehat{\Ex}[\by'(0)] -  \Ex[\by'(0)] \|^2_2, 
\end{align}
which precisely corresponds to \eqref{eq:mse_test}. 
Theorem~\ref{thm:test_error} immediately leads to a vanishing bound on \eqref{eq:sc_error} as $(n, m, p)$ grow. 
The exact finite-sample rates given in Theorem~\ref{thm:test_error} improve upon the best known rate provided in \cite{pcr_jasa}, which is only established in expectation and for random designs.  
To the best of our knowledge, Theorem~\ref{thm:test_error} is also the first guarantee for fixed designs in the synthetic controls literature. 

\subsection{Examining Assumption~\ref{assumption:subspace_incl} for Two Synthetic Controls Studies}\label{sec:sc_studies}
We revisit two canonical synthetic controls case studies: 
(i) terrorism in Basque Country \citep{abadie1} and 
(ii) California's Proposition 99 \citep{abadie2}. 
These studies have been used extensively to explain the utility of the synthetic controls method.
We apply the practical variant of our hypothesis test for Assumption~\ref{assumption:subspace_incl} in Section~\ref{sec:hypo_practical} to study the potential feasibility of counterfactual inference in both studies. 

\subsubsection{Terrorism in Basque Country} 
\noindent {\em Background \& setup.} 
%
%
Our first study evaluates the economic ramifications of 
terrorism on the Basque Country of Spain.
Our data is comprised of the per-capita GDP associated with 17 Spanish regions over 43 years. 
Basque Country is the sole treated unit that is affected by terrorism; the remaining $p=16$ regions are the control regions that are relatively unaffected by terrorism. 
The pre- and post-intervention durations are $n = 14$ and $m = 29$ years, respectively. 
We note that the original work of \cite{abadie1} uses 13 additional predictor variables for each region, including demographic information pertaining to one's educational status, and average shares for six industrial sectors.
We only use information related to the outcome of interest, i.e., the per-capita GDP. 

\medskip \noindent 
{\em Hypothesis test results.} 
We consider $\alpha = 0.05$. 
We estimate $r' = 3$ via the universal data-driven approach of \cite{Gavish_2014}. 
%
%
This sets $\tau(\alpha) = 0.15$.  
Estimating $r$ analogously to $r'$, we obtain $r = 5$ and $\htau = 0.61$.  
Since $\htau > \tau(\alpha)$, our test suggests that the PCR-based method of \cite{rsc} may not be suitable for this study under our assumptions. 
In fact, our test only passes for (effectively) $\alpha > 0.21$, which roughly translates to allowing for over $21\%$ of the spectral energy of $\bHp$ to fall outside of $\bH$. 

%

\subsubsection{California Proposition 99} 
\noindent {\em Background \& setup.} 
Our second study evaluates the effect of California's Proposition 99 on the consumption of tobacco. 
Our data is comprised of annual per-capita cigarette sales at the state level for 39 U.S. states for 31 years. 
With the exception of California, the other states in this study neither adopted an anti-tobacco program or raised cigarette sales taxes by 50 cents or more.
As such, the remaining $p=38$ states are considered the control states and California is considered the treated state. 
The pre- and post-intervention durations are $n = 18$ and $m = 13$ years, respectively. 
The original work of \cite{abadie2} uses six additional covariates per state. 
We do not include these variables in our study. 

\medskip \noindent {\em Hypothesis test results.} 
We consider $\alpha = 0.05$. 
Estimating $(r, r')$ as above, we obtain $r = 4$ and $r' = 3$, which yields $\tau(\alpha) = 0.15$ and $\htau = 1.63$. 
Again, we have $\htau > \tau(\alpha)$, which suggests that PCR-based methods may be ill-suited to produce reliable counterfactual estimates under our assumptions. 
Our test, therefore, only passes for (effectively) $\alpha > 0.55$. 

\subsubsection{Discussion of Findings} 
Although our tests do not pass for either study, our results are not meant to discredit the previous conclusions drawn in \cite{rsc} and \cite{pcr_jasa}.
Rather, our tests highlight that these studies warrant further investigation. 
We hope our findings not only motivate the usage of this test, but also spark the development of new robustness tests to stress investigate the assumptions that underlie statistical methods and thus the associated causal conclusions drawn from these methods. 


%% file: content/discussion_results.tex
\section{Related works} \label{sec:discussion} 
This section discusses related prior works from several literatures. 
%

\subsection{Principal Component Regression} \label{sec:pcr_lit} 
Since its introduction in \cite{pcr_jolliffe}, there have been several notable works analyzing PCR, including \cite{pcr_tibshirani, robust_pcr, pcr_jasa, recent-survey}. 
We pay particular attention to \cite{robust_pcr, pcr_jasa} given their closeness to this article.

\subsubsection{Model Identification} 
\cite{robust_pcr, pcr_jasa} purely focuses on prediction and thus, do not provide any results for model identification. 
This work proves that PCR identifies the unique minimum $\ell_2$-norm model with non-asymptotic rates of convergence. 

\subsubsection{Out-of-Sample Prediction} 
\cite{robust_pcr, pcr_jasa} shows that
PCR's out-of-sample prediction error decays as $\tO(1 / \sqrt{n})$ when $m, p = \Theta(n)$. 
\cite{robust_pcr, pcr_jasa} conjecture that their ``slow'' rate is an artefact of their Rademacher complexity arguments.  
By leveraging our model identification result in Theorem~\ref{thm:param_est}, we establish the ``fast'' rate of $\tO(1 / n)$.  

\subsubsection{Framework}
\noindent {\em Learning setup.} 
\cite{robust_pcr, pcr_jasa} considers a transductive learning setting, where {\em both} the in- and out-of-sample covariates are accessible upfront. 
This work, in comparison, considers the classical supervised learning setup, where the out-of-sample covariates are not revealed during training. 

\medskip 
\noindent 
{\em Covariate design.} 
\cite{robust_pcr, pcr_jasa} considers a random design setting with i.i.d. covariates.
By contrast, we consider a fixed design setting. 
As \cite{shao_deng} notes, estimation in high-dimensional regimes with fixed designs is very different from those with random designs due to the identifiability of the model parameter. 
Additionally, since we treat the covariates as deterministic, we do {\em not} impose that the in- and out-of-sample covariates obey the same distribution. 
Under the linear algebraic condition of Assumption \ref{assumption:subspace_incl}, we prove that PCR achieves consistent out-of-sample prediction in Corollary~\ref{cor:test_error.0}. 

\subsection{Functional Principal Component Analysis}\label{sec:fpca}
We consider functional principal component analysis (fPCA), which generalizes PCA to infinite-dimensional operators \citep{Yao_Muller_Wing_Aos, Hall_Muller_Wang_Aos, Li_Hsing_Aos, descary2019functional}.
This literature often assumes access to $n$ randomly sampled trajectories at $p$ locations, which are carefully chosen from a grid with minor perturbations, forming an $n \times p$ data matrix, $\bD$. 
Thus, $\bD^T \bD$ is the empirical proxy of the underlying covariance kernel that corresponds to these random trajectories. 
Under appropriate assumptions on the trajectories, the $\bD^T \bD$ matrix can be represented as the additive sum of a low-rank matrix with a noise matrix. 
This resembles the low-rank matrix estimation problem with a key difference being that {\em all} entries here are fully observed. 
In \cite{descary2019functional}, the low-rank component is estimated by  performing an explicit rank minimization, which is known to be computationally hard.  
The functional (or trajectory) approximation from this low-rank estimation
is obtained by smoothing (or interpolation)---this is where the careful choice of locations in a grid plays an important role. 
The estimation error is provided with respect to the normalized Frobenius norm (i.e., Hilbert-Schmidt norm when discretized). 
Finally, we remark that the fPCA literature has thus far considered diverging $n$ with fixed $p$ or $n \gg p$. 

%

In comparison, PCR utilizes hard singular value thresholding (HSVT), a popular method in the matrix estimation toolkit, to recover the low-rank matrix;  
such an approach is computationally efficient and even yields a closed form solution.
As shown in \cite{pcr_jasa}, PCR can be equivalently interpreted as HSVT followed by ordinary least squares. 
Hence, unlike the standard fPCA setup, PCR allows for missing values in the covariate matrix since HSVT recovers the underlying matrix in the presence of noisy and missing entries.
Analytically, our model identification and prediction error guarantees rely on matrix recovery bounds with respect to the $\ell_{2, \infty}$-norm, which is stronger than the Frobenius norm, i.e., $(np)^{-1/2} \| \bA \|_F \le n^{-1/2}  \| \bA \|_{2, \infty}$. 
Put differently, the typical Frobenius norm bound is insufficient to provide guarantees for PCR with error-in-variables. 
%
%
Finally, our setting allows for both $n \ll p$ and $n \gg p$; the current fPCA literature only allows for $n \gg p$.   

In this view, our work offers several directions for research within the fPCA literature: 
(i) allow the sampling locations to be different across the $n$ measurements, provided there is sufficient overlap;
(ii) consider settings beyond $n \gg p$; 
%
(iii) extend fPCA guarantees for computationally efficient methods like HSVT. 

%
There has also been work on functional principal component regression (fPCR), which allows $\bbeta^*$ to be an infinite-dimensional parameter. 
Notable works include \cite{Hall_Horowitz_Aos} and \cite{Call_Hall_Aos}, which consider the problems of model identification and prediction error, respectively.
These works, however, do {\em not} allow for error-in-variables. 
As noted above, model identification and out-of-sample guarantees at the fast rate of $\tO(1/n)$ for PCR with error-in-variables in the finite-dimensional case has remained elusive.
Extending these results for fPCR with error-in-variables remains interesting future work.
%

%
%
%


\subsection{Error-in-Variables}\label{sec:eiv_lit}

There are numerous prominent works in the high-dimensional error-in-variables literature, including \cite{tsybakov_1, tsybakov_2, orthogonal_1, orthogonal_2, loh_wainwright, weighted_l1, tsybakov_3, tsybakov_4, cocolasso}.   
Below, we highlight a few key points of comparison.

\subsubsection{Out-of-Sample Prediction} 
By and large, this literature has focused on model identification. 
Accordingly, the algorithms in the works above are ill-equipped to produce reliable predictions given corrupted and partially observed out-of-sample covariates. 
Therefore, even if the true model parameter $\bbeta^*$ is known, it is unclear how prior results can be extended to establish generalization error bounds. 
This work shows PCR can be easily adapted to handle these cases. 

\subsubsection{Knowledge of Noise Distribution} 
Many existing algorithms explicitly utilize knowledge of the underlying noise distribution to recover $\bbeta^*$.
%
Typically, these algorithms perform corrections of the form $\bZ^T \bZ - \Ex[\bW^T \bW]$. 
%
To carry out this computation, one must assume access to either oracle knowledge 
of $\Ex[\bW^T \bW]$ or obtain a good data-driven estimator for it.
As \cite{orthogonal_2} note, such an estimator can be costly or simply infeasible in many practical settings. 
PCR does not require any such knowledge. 
Instead, the PCA subroutine within PCR {\em implicitly} de-noises the covariates. 
The trade-off is that our results only hold if the number of retained singular components $k$ is chosen to be the rank of $\bX$.
Although there are numerous heuristics to aptly choose $k$, we leave a formal analysis of PCR when $k$ is misspecified as important future work.

\subsubsection{Operating Assumptions}
We compare our primary assumptions with canonical assumptions in the literature. 

\medskip
\noindent
{\em I: Low-rank vis-\'{a}-vis sparsity.} 
The most popularly endowed structure in high-dimensional regression is that the model parameter $\bbeta^*$ is $r$-sparse. 
This work posits that the in-sample covariate matrix $\bX$ is described by $r$ nonzero singular values. 
These two notions are related. 
If $\rank(\bX) = r$, then there exists an $r$-sparse $\btbeta$ such that $\bX \bbeta^* = \bX \btbeta$; see Proposition 3.4 of \cite{pcr_jasa}.  
Meanwhile, if $\bbeta^{*}$ is $r$-sparse, then there exists a $\tilde{\bX}$ of rank $r$ that also provides equivalent responses. 
In this view, the two perspectives are complementary. 


With that said, it is difficult to verify the sparsity of $\bbeta^*$, but the low-rank assumption on $\bX$ can be examined through the singular values of $\bZ$, as described in Section~\ref{sec:choose_k}. 
It is also well-established that (approximately) low-rank matrices are abundant in real-world data science applications \citep{xu2017rates, Udell2017NiceLV, udell2018big}. 

\medskip
\noindent
{\em II: Well-balanced spectra vis-\'{a}-vis restricted eigenvalue condition.}
The second common condition in the literature captures the amount of ``information spread'' across the rows and columns of $\bX$, which leads to a bound on its smallest singular value.
This is referred to as the restricted eigenvalue condition (see Definitions 1 and 2 in \cite{loh_wainwright}), which is imposed on the empirical estimate of the covariance of $\bX$. 
This work assumes the spectrum of 
$\bX$ is well-balanced (Assumption~\ref{assumption:balanced_spectra}).
This assumption is {\em not} necessary for consistent estimation. 
Rather, it is one condition that yields a reasonable $\snr$, which guarantees both model identification {\em and} vanishing out-of-sample prediction errors.

In many prior works, the restricted eigenvalue condition (and its variants) are shown to hold w.h.p. if the rows of $\bX$ are i.i.d. (or at least, independent) samples from a mean zero sub-gaussian distribution. 
This data generating process implies that the smallest and largest singular values of $\bX$ are of order $\tO(\sqrt{n} + \sqrt{p})$.
However, under the assumptions $\text{rank}(\bX) = r$ and $|X_{ij}| = \Theta(1)$, one can verify that $\|\bX \|_2 = \Omega(\sqrt{np / r} )$. 
The difference in the typical magnitude of the largest singular value reflects the difference in applications in which a restricted eigenvalue assumption versus a low-rank assumption is likely to hold.
The restricted eigenvalue assumption is particularly suited in applications such as compressed sensing where researchers {\em design} $\bX$. 
The applications arising in the social or life sciences primarily involve observational data. 
In such settings, a low-rank assumption on $\bX$ is arguably more suitable to capture the latent structure amongst the covariates. 
Ultimately, the Assumption~\ref{assumption:balanced_spectra} is similar to the restricted eigenvalue condition in that it requires the smallest and largest nonzero singular values of $\bX$ to be of the same order.

It turns out that analogous assumptions are pervasive across many fields.
Within the econometrics factor model literature, it is standard to assume that the factor structure is separated from the idiosyncratic errors, e.g., Assumption A of \cite{bai2020matrix}; 
within the robust covariance estimation literature, this assumption is closely related to the notion of pervasiveness, e.g., Proposition 3.2 of \cite{fan2017ellinfty}; 
within the matrix/tensor completion literature, it is assumed that the nonzero singular values are of the same order to achieve minimax 
optimal rates, e.g., \cite{cai2019nonconvex}.
Assumption~\ref{assumption:balanced_spectra} has also been shown to hold w.h.p. for the embedded Gaussians model, which is a canonical probabilistic generating process 
used to analyze probabilistic PCA \citep{probpca, bayesianpca, pcr_jasa}. 
%
%
Finally, like the low-rank assumption, a practical benefit of the well-balanced spectra assumption is that it can be empirically examined via the same procedure outlined in Section \ref{sec:choose_k}. 

\subsection{Linear Regression with Hidden Confounding}\label{ssec:lr.confound}
The problem of high-dimensional error-in-variables regression is related to linear regression with hidden confounding, a common model within the causal inference and econometrics literatures \citep{guo2020doubly, spectral_deconfounding}. 
As noted by \cite{guo2020doubly}, a particular class of error-in-variables models can be reformulated as linear regression with hidden confounding.
Using our notation, they consider a high-dimensional model where the rows of $\bX$ are sampled i.i.d. 
As such, $\bX$ can be full-rank, but $\bW$ is assumed to have low-rank structure. 
The aim of this work is to estimate a sparse $\bbeta^*$.
In comparison, we place the low-rank assumption on $\bX$, and assume the rows of $\bW$ are sampled independently and thus, can be full-rank.
Notably, for this setup, \cite{spectral_deconfounding} ``deconfounds'' the observed covariates $\bZ$ by a spectral transformation of its singular values.
It is interesting future work to analyze PCR for this important and closely related scenario. 

	
	
	
	
	


%% file: content/discussion.tex
\section{Conclusion}\label{sec:conclusion}
The most immediate direction for future work is to establish bounds when the covariates are approximately low-rank. 
Within this context, our analysis suggests PCR induces an additional error of the form $\|(\bI - \bV_r \bV_r^T) \btbeta^*\|_2$, where $\bV_r$ is formed from the top $r$ principal components of $\bX$. 
This is the unavoidable model misspecification error that results from taking a rank $r$ approximation of $\bX$. 
It stands to reason that {\em soft} singular value thresholding (SVT), which appropriately down-weights the singular values of $\btZ$, may be a more appropriate algorithmic approach as opposed to the {\em hard} SVT. 

Another future line of research is to bridge our out-of-sample prediction analysis with recent works that analyze over-parameterized estimators.
\cite{ols1}, for instance, demonstrates that the minimum $\ell_2$-norm linear regression solution predicts well out-of-sample despite a perfect fit to noisy in-sample data; this phenomena is known as ``benign overfitting''. 
To establish their result, \cite{ols1} introduces two notions of ``effective rank'' of the data covariance, and characterize linear regression problems that exhibit benign overfitting with respect to these quantities. 
In comparison, this work characterizes the out-of-sample prediction performance of PCR with respect to the $\ell_2$-distance between the in- and out-of-sample covariates (see Assumption~\ref{assumption:subspace_incl}).
Accordingly, one exciting research agenda is to explore the interplay of these two conceptions for over-parameterized linear estimators. 
This may also have implications for approximately low-rank settings. 



%% file: content/appendix_simulation.tex
\section{Illustrative Simulations: Details} \label{sec:simulation_details} 
We present the generative models in our simulation studies in Section~\ref{sec:simulations}.

\subsection{PCR Identifies the Minimum $\ell_2$-norm Model} \label{sec:sim1.details}
We generate $\bX = \bU \bV^T$, where the entries of $\bU, \bV$ are sampled independently from a standard normal distribution. 
Next, we generate $\bbeta^* \in \Rb^p$ by sampling from a multivariate standard normal vector with independent entries, and normalize it by $\| \bbeta^* \|_2$ so that it has unit norm.
We define $\btbeta^* = \bX^\dagger \bX \bbeta^*$. 
For each simulation repeat, we independently sample the entries of $\bvarepsilon \in \Rb^{n}$
from a normal distribution with mean $0$ and variance $\sigma^2 = 0.2$. 
The entries of $\bW \in \Rb^{n \times p}$ are sampled in an identical fashion. 
We then define our observed response vector as $\by = \bX \bbeta + \bvarepsilon$ and observed covariate matrix as $\bZ = \bX + \bW$. 
For simplicity, we do not mask any of the entries. 

\subsection{PCR is Robust to Covariate Shifts} \label{sec:sample}
We generate $\bX = \bU \bV^T$ as in Appendix~\ref{sec:sim1.details}. 
Next, we generate four different out-of-sample covariates, defined as $\bXp_1, \bXp_2, \bXp_3, \bXp_4$ via the following procedure: 
We independently sample the entries of $\bUp_1$ from a standard normal distribution, and define $\bXp_1 = \bUp_1 \bV^T$. 
We define $\bXp_2 = \bUp_2 \bV^T$ similarly with the entries of $\bUp_2$ sampled from $\mathcal{N}(0, 5)$. 
Next, we independently sample the entries of $\bUp_3$ from $\text{Uniform}[-\sqrt{3}, \sqrt{3}]$,
and define $\bXp_3 = \bUp_3 \bV^T$. 
We define $\bXp_4 = \bUp_4 \bV^T$ similarly
with the entries of $\bUp_4$ sampled from $\text{Uniform}[-\sqrt{15}, \sqrt{15}]$.  

\medskip
By construction, the mean and variance of the entries in $\bXp_3$ match that of $\bXp_1$; an analogous relationship holds between $\bXp_4$ and $\bXp_2$. 
While $\bXp_1$ follows the same distribution as that of $\bX$, there is a clear distribution shift from $\bX$ to $\bXp_3, \bXp_2, \bXp_4$. 

\medskip
We proceed to generate $\bbeta^*$ from a standard multivariate normal.
We define $\btheta'_1 = \bXp_1 \bbeta^*$, and define 
$\btheta'_2, \btheta'_3, \btheta'_4$ analogously. 
Further, the entries of $\bvarepsilon$ and $\bW, \bWp$ are independently sampled from a normal distribution with variance $\sigma^2 = 0.2$.
We define the training responses as $\by = \bX \bbeta^* + \bvarepsilon$ and observed training covariates as $\bZ = \bX + \bW$. 
The first set of observed testing covariates is defined as $\bZp_1 = \bXp_1 + \bWp$, with analogous definitions for $\bZp_2, \bZp_3, \bZp_4$.

\subsection{PCR Generalizes under Assumption~\ref{assumption:subspace_incl}}
We generate $\bX = \bU \bV^T$ as in Appendix~\ref{sec:sim1.details}. 
We now generate two different testing covariates. 
First, we generate $\bXp_1 = \bUp \bV^T$, where the entries of $\bUp$ are independently sampled from a normal distribution with mean zero and variance $5$.
As such, it follows that Assumption \ref{assumption:subspace_incl} immediately holds between $\bXp_1$ and $\bX$, though they do not obey the same distribution.
Next, we generate $\bXp_2 = \bU \bV^{\prime T}$, where the entries of $\bVp$ are independently sampled from a standard normal (just as in $\bV$). 
In doing so, we ensure that $\bXp_2$ and $\bX$ follow the same distribution, though Assumption \ref{assumption:subspace_incl} no longer holds. 

\medskip
We generate $\bbeta^*$ as in Appendix~\ref{sec:sample}, and define $\btheta'_1 = \bXp_1 \bbeta^*$ and $\btheta'_2 = \bXp_2 \bbeta^*$. 
We also generate $\bvarepsilon, \bW, \bWp$ as in Appendix~\ref{sec:sample}.
In turn, we define the training data as $\by = \bX \bbeta^* + \bvarepsilon$ and $\bZ = \bX + \bW$, and testing data as $\bZp_1 = \bXp_1 + \bWp$ and $\bZp_2 = \bXp_2 + \bWp$. 

\subsection{PCR Generalizes with MCAR Entries}
We generate $\bX = \bU \bV^T$ as in Appendix~\ref{sec:sim1.details} and generate $\bXp = \bUp \bV^T$, where the entries of $\bUp$ are independently sampled from a standard normal. 
As such, it follows that Assumption \ref{assumption:subspace_incl} immediately holds between $\bXp$ and $\bX$. 

\medskip
We generate $\bbeta^*$ as in Appendix~\ref{sec:sample}, and define $\btheta = \bXp \bbeta^*$.
We also generate $(\bvarepsilon, \bW, \bWp)$ as in Appendix~\ref{sec:sample}. 
There are $\rho$ entries in $\bPi, \bPip$ that are randomly assigned the value $1$, and each iteration considers a different permutation of revealed entries. 
Putting everything together, we define the training data as $\by = \bX \bbeta^* + \bvarepsilon$ and $\bZ = (\bX + \bW) \circ \bPi$, and testing data as $\bZp = (\bXp + \bWp) \circ \bPip$.

%% file: content/proof_paramest.tex
\section{Proof of Theorem \ref{thm:param_est}}\label{sec:proof_thm_4.1} 
We start with some useful notation. 
Note $\bX \bbeta^* = \bX \btbeta^*$.
Let $\by = \bX \btbeta^* + \bvarepsilon$ be the vector notation of \eqref{eq:response_vector} with $\by = [y_i: i \le n] \in \Rb^n$, $\bvarepsilon = [\varepsilon_i: i \le n] \in \Rb^n$.  
Throughout, let $\bX  = \bU \bS \bV^T$ denote the singular value decomposition (SVD) of $\bX$.
Recall that we write $\btZ = \hrho^{-1} \bZ = \bhU \bhS \bhV^T$ for the SVD of $\btZ$.
Its truncation using the top $k$ singular components is denoted as $\btZ^k = \bhU_k \bhS_k \bhV_k^T$. 

Further, we will often use the following bound: 
for any $\bA \in \Rb^{a \times b}$, $\bv \in \Rb^{b}$, 
\begin{align}\label{eq:2.inf.ineq}
\| \bA \bv \|_2 & = \| \sum_{j=1}^b \bA_{\cdot j} v_j \|_2  \leq \big(\max_{j \le b} \| \bA_{\cdot j}\|_2\big) \big(\sum_{j=1}^b |v_j|\big) 
= \| \bA\|_{2, \infty} \|\bv\|_1,
\end{align}
where $\|\bA\|_{2,\infty} = \max_{j} \| \bA_{\cdot j}\|_2$ with $\bA_{\cdot j}$ representing the $j$-th column of $\bA$. 

As discussed in Section \ref{sec:param_est_results}, we will consider $\btbeta^*$ as our model parameter of interest. 
This corresponds to the  unique minimum $\ell_2$-norm model parameter satisfying \eqref{eq:response_vector} for $i \le n$.
As a result, it follows that 
\begin{align}\label{eq:2.norm.min}
\bV_{\perp}^T \btbeta^* &= \bzero,
\end{align} 
where $\bV_\perp$ represents a matrix of orthornormal basis vectors that span the nullspace of $\bX$.

Similarly, let $\bhV_{k, \perp} \in \Rb^{p \times (p-k)}$ be a matrix of orthonormal basis vectors that span the nullspace of $\btZ^k$; thus, $\bhV_{k, \perp}$ is orthogonal to $\bhV_k$. 
Then, 
\begin{align}
\| \bhbeta - \btbeta^*\|_2^2 & =  \| \bhV_k \bhV_k^T (\bhbeta - \btbeta^*) + \bhV_{k, \perp}\bhV_{k, \perp}^T (\bhbeta - \btbeta^*)\|_2^2 \nonumber \\
& = \| \bhV_k \bhV_k^T (\bhbeta - \btbeta^*)  \|_2^2 ~+~ \| \bhV_{k, \perp}\bhV_{k, \perp}^T (\bhbeta - \btbeta^*)\|_2^2 \nonumber \\
& =  \| \bhV_k \bhV_k^T (\bhbeta - \btbeta^*)  \|_2^2 ~+~ \| \bhV_{k, \perp}\bhV_{k, \perp}^T\btbeta^*\|_2^2. \label{eq:thm1.1}
\end{align}
Note that in the last equality we have used Property \ref{property:equiv}, which states that $\bhV_{k, \perp}^T \bhbeta = \bzero$. 
Next, we bound the two terms in \eqref{eq:thm1.1}. 

\vspace{10pt} \noindent 
{\em Bounding $ \| \bhV_k \bhV_k^T (\bhbeta - \btbeta^*)  \|_2^2$. } 
To begin, note that 
\begin{align}\label{eq:thm1.2.1}
\| \bhV_k \bhV_k^T (\bhbeta - \btbeta^*)  \|_2^2 & 
= \|\bhV_k^T (\bhbeta - \btbeta^*)  \|_2^2, 
\end{align}
since $\bhV_k$ has orthonormal columns. 
Next, consider
\begin{align}
\| \btZ^k(\bhbeta - \btbeta^*)\|_2^2 & \leq 2\| \btZ^k \bhbeta -\bX \btbeta^* \|_2^2 + 2 \| \bX \btbeta^* - \btZ^k \btbeta^*\|_2^2 \nonumber \\
& \leq 2\| \btZ^k \bhbeta -\bX \btbeta^* \|_2^2 + 2 \| \bX - \btZ^k \|_{2,\infty}^2 \|\btbeta^*\|_1^2, 
\end{align}
where we used \eqref{eq:2.inf.ineq}. Recall that $\btZ^k = \bhU_k \bhS_k \bhV_k^T$. 
Therefore, 
\begin{align}
\| \btZ^k(\bhbeta - \btbeta^*)\|_2^2 
& =  (\bhbeta - \btbeta^*)^T \bhV_k \bhS_k^2 \bhV_k^T (\bhbeta - \btbeta^*) \nonumber \\
& = (\bhV_k^T (\bhbeta - \btbeta^*))^T \bhS_k^2 (\bhV_k^T (\bhbeta - \btbeta^*)) \nonumber \\
& \geq \hs_k^2 \| \bhV_k^T (\bhbeta - \btbeta^*) \|_2^2. 
\end{align}
Therefore using \eqref{eq:thm1.2.1}, we conclude that 
\begin{align}\label{eq:thm1.2.2}
\| \bhV_k \bhV_k^T (\bhbeta - \btbeta^*)  \|_2^2 & \leq \frac{2}{\hs_k^2} \Big(\| \btZ^k \bhbeta -\bX \btbeta^* \|_2^2 +  \| \bX - \btZ^k \|_{2,\infty}^2 \|\btbeta^*\|_1^2\Big).
\end{align}
Next, we bound $\| \btZ^k \bhbeta -\bX \btbeta^* \|_2$. 
\begin{align}\label{eq:thm1.2.3}
\| \btZ^k \bhbeta - \by \|_2^2 & = \| \btZ^k \bhbeta - \bX \btbeta^* - \bvarepsilon\|_2^2 \nonumber \\
& = \| \btZ^k \bhbeta - \bX \btbeta^*\|_2^2 + \|\bvarepsilon\|_2^2 - 2 \langle \btZ^k \bhbeta - \bX \btbeta^*, \bvarepsilon\rangle.
\end{align}
By Property \ref{property:equiv} we have, 
\begin{align}\label{eq:thm1.2.4}
\| \btZ^k \bhbeta - \by \|_2^2 & \leq \| \btZ^k \btbeta^* - \by \|_2^2 ~=~\| (\btZ^k - \bX) \btbeta^* - \bvarepsilon\|_2^2 \nonumber \\
& = \| (\btZ^k - \bX) \btbeta^*\|_2^2 + \| \bvarepsilon\|_2^2 - 2 \langle (\btZ^k - \bX) \btbeta^*, \bvarepsilon\rangle. 
\end{align}
From \eqref{eq:thm1.2.3} and \eqref{eq:thm1.2.4}, we have 
\begin{align}
\| \btZ^k \bhbeta - \bX \btbeta^*\|_2^2 & \leq \| (\btZ^k - \bX) \btbeta^*\|_2^2 + 2 \langle \btZ^k (\bhbeta - \btbeta^*), \bvarepsilon\rangle \nonumber \\
& \leq \| \bX - \btZ^k \|_{2,\infty}^2 \|\btbeta^*\|_1^2 + 2 \langle \btZ^k (\bhbeta - \btbeta^*), \bvarepsilon\rangle, 
\label{eq:thm1.2.5}
\end{align}
where we used \eqref{eq:2.inf.ineq}.  From \eqref{eq:thm1.2.2} and \eqref{eq:thm1.2.5}, we conclude that 
\begin{align}\label{eq:thm1.2}
\| \bhV_k \bhV_k^T (\bhbeta - \btbeta^*)  \|_2^2 & \leq \frac{4}{\hs_k^2} \Big( \| \bX - \btZ^k \|_{2,\infty}^2 \|\btbeta^*\|_1^2 +\langle \btZ^k (\bhbeta - \btbeta^*), \bvarepsilon\rangle \Big).
\end{align}

\vspace{10pt} \noindent 
{\em {Bounding $ \| \bhV_{k, \perp}\bhV_{k, \perp}^T\btbeta^*\|_2^2$. }} 
Consider
\begin{align}\label{eq:thm1.3.1}
\| \bhV_{k, \perp}\bhV_{k, \perp}^T\btbeta^*\|_2 
	& = \| (\bhV_{k, \perp}\bhV_{k, \perp}^T - \bV_{\perp}\bV_{\perp}^T) \btbeta^* + \bV_{\perp}\bV_{\perp}^T \btbeta^*\|_2 \nonumber \\
	& \stackrel{(a)}{=} \| (\bhV_{k, \perp}\bhV_{k, \perp}^T - \bV_{\perp}\bV_{\perp}^T) \btbeta^* \|_2 \nonumber \\
	& \leq \| \bhV_{k, \perp}\bhV_{k, \perp}^T - \bV_{\perp}\bV_{\perp}^T \|_2 ~ \|\btbeta^*\|_2,
\end{align}
where (a) follows from $\bV_\perp^T \btbeta^* = \bzero$ due to \eqref{eq:2.norm.min}. 
Then,
\begin{align}\label{eq:thm1.3.2}
\bhV_{k, \perp}\bhV_{k, \perp}^T - \bV_{\perp}\bV_{\perp}^T & = (\Id - \bV_{\perp}\bV_{\perp}^T) - (\Id - \bhV_{k, \perp}\bhV_{k, \perp}^T ) \nonumber \\
& = \bV \bV^T - \bhV_{k} \bhV_{k}^T.
\end{align}
From \eqref{eq:thm1.3.1} and \eqref{eq:thm1.3.2}, it follows that 
\begin{align}\label{eq:thm1.3}
\| \bhV_{k, \perp}\bhV_{k, \perp}^T\btbeta^*\|_2 & \leq  \| \bV \bV^T - \bhV_{k} \bhV_{k}^T\|_2 ~ \|\btbeta^*\|_2.
\end{align}

\vspace{10pt} \noindent 
{\em {Bringing together \eqref{eq:thm1.1}, \eqref{eq:thm1.2}, and \eqref{eq:thm1.3}. }} 
Collectively, we obtain 
\begin{align}\label{eq:thm1.4}
\| \bhbeta - \btbeta^*\|_2^2 &\leq  \| \bV \bV^T - \bhV_{k} \bhV_{k}^T\|_2^2 ~ \|\btbeta^*\|_2^2 \nonumber
\\ &\quad + \frac{4}{\hs_k^2} \Big( \| \bX - \btZ^k \|_{2,\infty}^2 ~ \|\btbeta^*\|_1^2 +\langle \btZ^k (\bhbeta - \btbeta^*), \bvarepsilon\rangle \Big).
\end{align}

\vspace{10pt} \noindent 
{\em Key lemmas. } 
We state the key lemmas bounding each of the terms on the right hand side of \eqref{eq:thm1.4}. 
This will help us conclude the proof of Theorem \ref{thm:param_est}. 
The proofs of these lemmas are presented in Sections \ref{ssec:thm1.lem1}, \ref{ssec:thm1.lem2}, \ref{ssec:thm1.lem3}, \ref{ssec:thm1.lem4}. 

\vspace{5pt}
\begin{lemma}\label{lem:thm1.1}
Consider the setup of Theorem \ref{thm:param_est}, and PCR with parameter $k = r$.
Then, for any $t > 0$, the following holds w.p. at least $1-\exp(-t^2)$: 
\begin{align}\label{eq:thm1.lem1.f2.0} 
\| \bU \bU^T - \bhU_r \bhU_r^T\|_2 & \leq
C(K + 1) ( \gamma + 1) \frac{\sqrt{n} + \sqrt{p} + t}{\rho s_r}, 
\\ \| \bV \bV^T - \bhV_r \bhV_r^T\|_2 & \leq 
C(K + 1) ( \gamma + 1)  \frac{\sqrt{n} + \sqrt{p} + t}{\rho s_r}. 
\end{align}
Here, $s_r>0$ represents the $r$-th largest singular value of $\bX$. 
\end{lemma}

\begin{lemma}\label{lem:thm1.2}
Consider PCR with parameter $k= r$ and $\rho \geq c (np)^{-1} \log^2 (np)$.
Then w.p. at least $1-O(1/(np)^{10})$,
\begin{align}
&\| \bX - \btZ^r \|_{2,\infty}^2 
	\\ &\quad \leq C(K+1)^4 (\gamma + 1)^2  
		\left( \frac{(n+p)(n +\sqrt{n}\, \log(np))}{\rho^4 s_r^2 } + \frac{r + \sqrt{r}\, \log(np)}{\rho^2} \right) + C \frac{\log(np)}{\rho \, p}.
\end{align}
\end{lemma}

\begin{lemma}\label{lem:thm1.3}
If $\rho \geq c (np)^{-1} \log^2 (np)$, then for any $k$, we have w.p. at least  $1-O(1/(np)^{10})$,
\begin{align}\label{eq:lem.thm.1.3.0}
| \hs_k - s_k | & \leq C(K + 1) ( \gamma + 1) \frac{\sqrt{n} + \sqrt{p}}{\rho} + C \frac{\sqrt{\log(np)}}{\sqrt{\rho \,np}} s_k.
\end{align}
\end{lemma}

\begin{lemma}\label{lem:thm1.4}
Given $\btZ^r$, the following holds w.p. at least $1-O(1/(np)^{10})$ with respect to the randomness in $\bvarepsilon$:  
{\small
\begin{align}\label{eq:lem.thm.1.4.0}
\langle \btZ^r (\bhbeta - \btbeta^*), \bvarepsilon\rangle  & \leq 
\sigma^2 r + C \sigma \sqrt{\log(np)} \left({ \sigma} \sqrt{r} + \sigma \sqrt{\log(np)} + \|\btbeta^*\|_1 (\sqrt{n}   +  \| \btZ^r - \bX\|_{2,\infty} ) \right).
 \end{align}
 } 
\end{lemma}

\vspace{10pt} \noindent 
{\em Completing the proof of Theorem \ref{thm:param_est}. } 
Using Lemma \ref{lem:thm1.4}, the following holds w.p. at least
$1-O(1/(np)^{10})$:
\begin{align}\label{eq:thm1.f.1}
& \| \bX - \btZ^r \|_{2,\infty}^2 \|\btbeta^*\|_1^2 +\langle \btZ^r (\bhbeta - \btbeta^*), \bvarepsilon\rangle \nonumber \\
&\quad \leq  \| \bX - \btZ^r \|_{2,\infty}^2 \|\btbeta^*\|_1^2 
+ C \sigma \sqrt{\log(np)} \| \bX - \btZ^r \|_{2,\infty} \|\btbeta^*\|_1 + C \sigma^2 \log(np) \nonumber \\
&\quad \qquad  + C \sigma \sqrt{\log(np)} (\sqrt{n} \|\btbeta^*\|_1 + {s \sigma}\sqrt{r})
+ \sigma^2 r \nonumber \\
& \quad\leq C \big(\| \bX - \btZ^k \|_{2,\infty} \|\btbeta^*\|_1 + \sigma \sqrt{\log(np)}\big)^2 + C \sigma \sqrt{\log(np)} (\sqrt{n} \|\btbeta^*\|_1 +{ \sigma} \sqrt{r})
+ \sigma^2 r \nonumber \\
&\quad \leq C \| \bX - \btZ^k \|_{2,\infty}^2  \|\btbeta^*\|_1^2 +  C \sigma^2 (\log(np) + r) + C \sigma \sqrt{n \log(np)}  \|\btbeta^*\|_1.
\end{align}
Using \eqref{eq:thm1.4} and \eqref{eq:thm1.f.1}, 
we have w.p. at least $1-O(1/(np)^{10})$,
\begin{align}
\| \bhbeta - \btbeta^*\|_2^2 
& \leq \| \bV \bV^T - \bhV_{k} \bhV_{k}^T\|_2^2 ~\|\btbeta^*\|_2^2  
+ C \frac{\| \bX - \btZ^k \|_{2,\infty}^2} {\hs_r^2} \|\btbeta^*\|_1^2
\\
&\quad 
+ C \frac{\sigma^2 (\log(np) + r)} {\hs_r^2}
+ C \frac{\sigma \sqrt{n \log(np)}} {\hs_r^2} \|\btbeta^*\|_1. 
\label{eq:thm1.f.2}
\end{align}
Using Lemma \ref{lem:thm1.1} in \eqref{eq:thm1.f.2}, we have w.p. at least $1-O(1/(np)^{10})$,
\begin{align}
\| \bhbeta - \btbeta^*\|_2^2 
& \leq  
C(K+1)^2(\gamma+1)^2 \frac{n + p }{\rho^2 s_r^2} \|\btbeta^*\|_2^2  + C \frac{\| \bX - \btZ^k \|_{2,\infty}^2 }{\hs_r^2} \|\btbeta^*\|_1^2
\\ &\quad 
+ C \frac{\sigma^2 (\log(np) + r)} {\hs_r^2}
+ C \frac{\sigma \sqrt{n \log(np)}} {\hs_r^2} \|\btbeta^*\|_1. 
\label{eq:thm1.f.3}
\end{align}
Applying Lemma \ref{lem:thm1.3} with $k =r$ and recalling  
$\rho \geq c (np)^{-1} \log^2 np$ and $\snr \geq C(K+1)(\gamma+1)$, 
\begin{align}
\frac{| \hs_r - s_r |}{s_r} & \leq C(K+1)(\gamma+1) \frac{\sqrt{n} + \sqrt{p}}{\rho s_r} + C \frac{\sqrt{\log(np)}}{\sqrt{\rho \,np}} \nonumber \\
& = \frac{C(K+1)(\gamma+1)}{\snr} + C \frac{\sqrt{\log(np)}}{\sqrt{\rho \,np}}
~\leq~ \frac{1}{2}. \label{eq:snr.simplify} 
\end{align}
%
As a result, 
\begin{align}\label{eq:singualr_val_emp_pop}
    s_r/2 \leq \hs_r \leq 3s_r/2.
\end{align}
Using the definition of $\snr$ as per \eqref{eq:snr} and $\sqrt{a + b} \le \sqrt{a} + \sqrt{b}$ for $a, b \ge 0$, we have  
\begin{align}\label{eq:snr_basic_ineq}
    \frac{n + p }{\rho^2 s_r^2} & \leq \frac{1}{\snr^2}.   
\end{align}
Using \eqref{eq:singualr_val_emp_pop}, \eqref{eq:snr_basic_ineq} we obtain 
\begin{align}
\frac{ \sigma^2 (\log(np) + r)}{\hs_r^2} ~&\leq~ C\frac{ \sigma^2 \rho^2 r \log(np)}{\snr^2 (n + p) } \le C\frac{ \sigma^2  r \log(np)}{\snr^2 (n \vee p) }, \label{eq:thm1.f.1x} \\
\frac{\sigma \sqrt{n \log(np)}}{\hs_r^2} \|\btbeta^*\|_1 
~&\leq C \frac{ \sigma \rho^2  \sqrt{n  \log(np)}}{\snr^2 (n + p)} \|\btbeta^*\|_1 \le C \frac{ \sigma \sqrt{n  \log(np)}}{\snr^2 (n \vee p)} \|\btbeta^*\|_1,  \label{eq:thm1.f.2x} \\
\frac{(n+p)(n +\sqrt{n}\, \log(np))}{\rho^4 s_r^2 \hs_r^2} \|\btbeta^*\|^2_1
& \leq C \frac{ n \log(np)}{\snr^4 (n + p)} \|\btbeta^*\|^2_1 \le C \frac{  \log(np)}{\snr^4} \|\btbeta^*\|^2_1, \label{eq:thm1.f.3x} \\
\frac{r + \sqrt{r}\, \log(np)}{\rho^2 \hs_r^2} \|\btbeta^*\|^2_1 & \leq 
C\frac{ r \log(np)}{\snr^2 (n + p)}  \|\btbeta^*\|^2_1,  \label{eq:thm1.f.4x} \\
\frac{\log(np)}{\rho \hs_r^2 p} \|\btbeta^*\|^2_1 & \leq C \frac{\log(np)}{\snr^2 p (n + p)} \|\btbeta^*\|^2_1. \label{eq:thm1.f.5x}
\end{align}
Plugging Lemma~\ref{lem:thm1.2}, \eqref{eq:thm1.f.1x}, \eqref{eq:thm1.f.2x}, \eqref{eq:thm1.f.3x}, \eqref{eq:thm1.f.4x}, \eqref{eq:thm1.f.5x} into \eqref{eq:thm1.f.3},  and simplifying completes the proof of \eqref{eq:thm1.main} in Theorem \ref{thm:param_est}.

It remains to establish \eqref{eq:bound_linear_coeff}. 
This result is first proved in \cite{synth_combo}, cf. Lemma 19; we state a similar proof for completeness. 
By definition, $\btbeta^* = \bX^\dagger \bX \bbeta^*$. 
As such, it immediately follows that 
\begin{align}
	\| \btbeta^* \|_2 
	&= \| \bX^\dagger \bX \bbeta^* \|_2
	\le \| \bX^\dagger \|_2 \cdot \| \bX \bbeta^* \|_2
	\le s_r^{-1} \cdot d \sqrt{n}. 
\end{align}
The last inequality follows from our boundedness assumption on $\langle \bx_i, \bbeta^* \rangle$ for all $i \le n$. 

Finally, the second part of \eqref{eq:bound_linear_coeff} follows from the property $\| \bv \|_1 \le \sqrt{p} \| \bv \|_2$ for any $\bv \in \Rb^p$. 


%% file: content/proof_param_est_lemmas.tex

\subsection{Proof of Lemma \ref{lem:thm1.1}}\label{ssec:thm1.lem1} 
Recall that $\bU, \bV$ denote the left and right singular vectors of $\bX$ (equivalently, $\rho \bX$), respectively;
meanwhile, $\bhU_k, \bhV_k$ denote the top $k$ left and right singular vectors of $\btZ$ (equivalently, $\bZ$), respectively.  
Further, observe that $\Ex[\bZ] = \rho \bX$ and let $\btW = \bZ - \rho \bX$. 
To arrive at our result, we recall Wedin's Theorem \citep{Wedin1972PerturbationBI}. 

\vspace{5pt}
\begin{thm} [Wedin's Theorem] \label{thm:wedin}
Given $\bA, \bB \in \Rb^{n \times p}$, let $\bA = \bU \bS \bV^T$ and $\bB = \bhU \bhS \bhV^T$ be their respective SVDs. 
Let $\bU_k, \bV_k$ (respectively, $\bhU_k, \bhV_k$) correspond to the truncation of $\bU, \bV$ (respectively, $\bhU, \bhV$) that retains the columns corresponding to the top $k$ singular values of $\bA$ (respectively, $\bB$). 
Let $s_k$ denote the $k$-th singular value of $\bA$. 
Then, 
\begin{align}
	\max \Big(\| \bU_k \bU_k^T - \bhU_k \bhU_k^T\|_2,  \| \bV_k \bV_k^T - \bhV_k \bhV_k^T\|_2 \Big) & \leq \frac{2\norm{\bA - \bB}_2}{ s_k - s_{k+1}}.
\end{align}
\end{thm}
Using Theorem \ref{thm:wedin} for $k = r$, it follows that 
\begin{align}\label{eq:thm1.lem1.1}
\max\Big( \| \bU \bU^T - \bhU_r \bhU_r^T\|_2, \| \bV \bV^T - \bhV_r \bhV_r^T\|_2\Big) & \leq \frac{2 \| \btW \|_2}{\rho s_r}, 
\end{align}
where $s_r$ is the smallest nonzero singular value of $\bX$.  Next, we obtain a high probability bound on $\| \btW\|_2$. 
To that end, 
\begin{align}\label{eq:thm1.lem1.2}
\frac1n \| \btW\|^2_2 & = \frac1n \| \btW^T \btW\|_2 \leq \frac1n \|\btW^T \btW - \Ex[\btW^T \btW]\|_2 +\frac1n \| \Ex[\btW^T \btW]\|_2.
\end{align}
We bound the two terms in \eqref{eq:thm1.lem1.2} separately. 
We recall the following lemma, which is a direct extension of Theorem 4.6.1 of \cite{vershynin2018high} for the non-isotropic setting, and we present its proof for completeness in Section \ref{ssec:matrix.conc}. 

\vspace{5pt}
\begin{lemma} [Independent sub-gaussian rows] \label{lem:subg_matrix} 
Let $\bA$ be an $n \times p$ matrix whose rows $A_i$ are independent, mean zero, sub-gaussian random vectors in $\Rb^p$ with second moment matrix $\bSigma = (1/n) \Ex [ \bA^T \bA]$. 
Then for any $t \ge 0$, the following holds w.p. at least $1 - \exp(-t^2)$: 
\begin{align}\label{eq:matrix.conc}
	\| \frac{1}{n} \bA^T \bA - \bSigma \|_2 & \le K^2 \max(\delta, \delta^2), \quad \text{where } \delta = C \sqrt{\frac{p}{n}} + \frac{t}{\sqrt{n}};
\end{align}
here, $K = \max_i \norm{ A_i }_{\psi_2}$. 
\end{lemma}
The matrix $\btW = \bZ - \rho \bX$ has independent rows by Assumption \ref{assumption:cov_noise}. 
We state the following Lemma about the distribution property of the rows of $\btW$, the proof of which can be found in Section \ref{ssec:subG.z}. 

\vspace{5pt}
\begin{lemma}\label{lem:subG.z}
Let Assumption \ref{assumption:cov_noise} hold. 
Then, $\bz_i - \rho \bx_i$ is a sequence of independent, mean zero, sub-gaussian random vectors satisfying $\| \bz_i - \rho \bx_i \|_{\psi_2} \leq C(K+1)$.
\end{lemma}
From Lemmas \ref{lem:subg_matrix} and \ref{lem:subG.z}, w.p. at least $1 - \exp(-t^2)$, 
\begin{align}\label{eq:thm1.lem1.3}
\frac1n \|\btW^T \btW - \Ex[\btW^T \btW]\|_2 & \leq C(K+1)^2 \left(1 + \frac{p}{n} + \frac{t^2}{n}\right).
\end{align}
Finally, we claim the following bound on $\| \Ex[\btW^T \btW]\|_2$, the proof of which is in Section \ref{ssec:lem:bound.cov}.

\vspace{5pt}
\begin{lemma}\label{lem:bound.cov}
Let Assumption \ref{assumption:cov_noise} hold.
Then, we have
\begin{align}
\| \Ex[\btW^T \btW]\|_2 & \leq C(K+1)^2 n (\rho - \rho^2) + n \rho^2 \gamma^2.
\end{align}
\end{lemma}
From \eqref{eq:thm1.lem1.2}, \eqref{eq:thm1.lem1.3} and Lemma \ref{lem:bound.cov}, we have w.p. at least $1-\exp(-t^2)$ for any $t > 0$
\begin{align}
\| \btW\|_2^2 & \leq C(K+1)^2 (n + p + t^2) + n (\rho(1-\rho)(K+1)^2 + \rho^2 \gamma^2). 
\end{align}
For this, we conclude the following lemma.

\vspace{5pt}
\begin{lemma}\label{lem.perturb.1}
For any $t > 0$, the following holds w.p. at least $1- \exp(-t^2)$: 
\begin{align}\label{eq:thm1.lem1.3.2}
\| \bZ - \rho \bX \|_2 & \leq 
C(K + 1) ( \gamma + 1)  (\sqrt{n} + \sqrt{p} + t).
%
\end{align}
\end{lemma}
Using the above and \eqref{eq:thm1.lem1.1}, we conclude the proof of Lemma \ref{lem:thm1.1}.

\subsection{Proof of Lemma \ref{lem:thm1.2}}\label{ssec:thm1.lem2}
We want to bound $\| \bX - \btZ^k \|_{2,\infty}^2$. 
To that end, let $\Delta_j = \bX_{\cdot j} - \btZ^k_{\cdot j}$ for any $j \in [p]$. 
Our interest is in bounding $\|\Delta_j\|_2^2$ for all $j \in [p]$. 
Consider,
\begin{align}
 \btZ^k_{\cdot j} - \bX_{\cdot j}  & =  (\btZ^k_{\cdot j} - \bhU_k \bhU_k^T\bX_{\cdot j}) + (\bhU_k \bhU_k^T\bX_{\cdot j} - \bX_{\cdot j}).
\end{align}
Now, note that $\btZ^k_{\cdot j} - \bhU_k \bhU_k^T\bX_{\cdot j}$ belongs to the subspace spanned by column vectors of $\bhU_k$, while $\bhU_k \bhU_k^T\bX_{\cdot j} - \bX_{\cdot j}$ belongs to its orthogonal complement with respect to $\Rb^n$. 
As a result, 
\begin{align}\label{eq:num.0}
\|\btZ^k_{\cdot j} - \bX_{\cdot j}\|^2_2 & = \| \btZ^k_{\cdot j} - \bhU_k \bhU_k^T\bX_{\cdot j}\|_2^2 + \| \bhU_k \bhU_k^T\bX_{\cdot j} - \bX_{\cdot j}\|_2^2.
\end{align}

\vspace{10pt} \noindent
{\em Bounding $\| \btZ^k_{\cdot j} - \bhU_k \bhU_k^T\bX_{\cdot j}\|_2^2$. }
Recall that $\btZ = (1/\hrho) \bZ = \bhU \bhS \bhV^T$, and hence $\bZ = \hrho \bhU \bhS \bhV^T$.  
Consequently, 
\begin{align}
\frac{1}{\hrho}  \bhU_k \bhU_k^T \bZ_{\cdot j} &= 
\frac{1}{\hrho}  \bhU_k \bhU_k^T \bZ \basis_j 
=  \bhU_k \bhU_k^T \bhU \bhS \bhV^T \basis_j \nonumber \\
& =\bhU_k \bhS_k \bhV_k^T \basis_j = \btZ^k_{\cdot j}.
\end{align}
Therefore, we have 
\begin{align}
\btZ^k_{\cdot j} - \bhU_k \bhU_k^T\bX_{\cdot j} & = \frac{1}{\hrho}  \bhU_k \bhU_k^T \bZ_{\cdot j} - \bhU_k \bhU_k^T\bX_{\cdot j} \nonumber \\
& =  \frac{1}{\hrho} \bhU_k \bhU_k^T ( \bZ_{\cdot j} - \rho \bX_{\cdot j} ) + \Big(\frac{\rho - \hrho}{\hrho} \Big) \bhU_k \bhU_k^T\bX_{\cdot j}. 
\end{align}
Therefore, 
\begin{align}
\| \btZ^k_{\cdot j} - \bhU_k \bhU_k^T\bX_{\cdot j} \|_2^2 & \leq \frac{2}{\hrho^2} \| \bhU_k \bhU_k^T ( \bZ_{\cdot j} - \rho \bX_{\cdot j} )\|_2^2 + 2 \Big(\frac{\rho - \hrho}{\hrho} \Big)^2 \| \bhU_k \bhU_k^T\bX_{\cdot j}\|_2^2 \nonumber \\
& \leq \frac{2}{\hrho^2} \|\bhU_k \bhU_k^T ( \bZ_{\cdot j} - \rho \bX_{\cdot j} )\|_2^2 + 
2 \Big(\frac{\rho - \hrho}{\hrho} \Big)^2 \| \bX_{\cdot j}\|_2^2,
\end{align}
where we have used the fact that $\| \bhU_k \bhU_k^T \|_2 = 1$. 
Recall that $\bU \in \Rb^{n \times r}$ represents the left singular vectors of $\bX$. 
Thus,
\begin{align}
\| \bhU_k \bhU_k^T ( \bZ_{\cdot j} - \rho \bX_{\cdot j} )\|_2^2 
	& \leq 2 \| (\bhU_k \bhU_k^T - \bU \bU^T) ( \bZ_{\cdot j} - \rho \bX_{\cdot j})\|_2^2 + 2 \| \bU \bU^T ( \bZ_{\cdot j} - \rho \bX_{\cdot j})\|_2^2 \nonumber \\
	& \leq 2 \| \bhU_k \bhU_k^T - \bU \bU^T\|_2^2 ~ \| \bZ_{\cdot j} - \rho \bX_{\cdot j} \|_2^2 + 2 \| \bU \bU^T ( \bZ_{\cdot j} - \rho \bX_{\cdot j})\|_2^2.
\end{align}
By Assumption \ref{assumption:bounded}, we have that $\| \bX_{\cdot j}\|_2^2 \leq n$. 
This yields 
\begin{align}\label{eq:num.1}
\| \btZ^k_{\cdot j} - \bhU_k \bhU_k^T\bX_{\cdot j} \|_2^2 
	& \leq \frac{4}{\hrho^2}  \| \bhU_k \bhU_k^T - \bU \bU^T\|_2^2 ~ \| \bZ_{\cdot j} - \rho \bX_{\cdot j} \|_2^2  \nonumber \\
	& \quad + \frac{4}{\hrho^2}\| \bU \bU^T ( \bZ_{\cdot j} - \rho \bX_{\cdot j})\|_2^2 + 2n  \Big(\frac{\rho - \hrho}{\hrho}\Big)^2.
\end{align}
We now state Lemmas \ref{lem:est.rho} and \ref{lem:projection}. Their proofs are in Sections \ref{ssec:lem.est.rho} and \ref{ssec:lem.projection}, respectively.

\vspace{5pt}
\begin{lemma}\label{lem:est.rho}
For any $\alpha > 1$, 
\begin{align}
\Prob{  \rho / \alpha  \le \widehat{\rho} \le \alpha \rho } & \geq 1-2 \exp\Big( -\frac{(\alpha-1)^2 n p \rho}{2\alpha^2}\Big).
\end{align}
Therefore, for $\rho \geq c \frac{\log^2 np}{np}$, we have w.p. $1-O(1/(np)^{10})$
\begin{align}
\frac{\rho}{2} \leq \hrho \leq 2\rho & \quad \mbox{and} \quad \Big(\frac{\rho - \hrho}{\hrho}\Big)^2  \leq C \frac{\log(np)}{\rho np}.
\end{align}
\end{lemma}

\begin{lemma}\label{lem:projection}
Consider any matrix $\bQ \in \Rb^{n \times \ell}$ with $1 \leq \ell \leq n$ such that its columns $\bQ_{\cdot j}$ for $j \in [\ell]$ are orthonormal vectors. 
Then for any $t > 0$, 
\begin{align}\label{eq:lem.proj.1}
	& \Pb \Big( \max_{j \in [p]} \, \norm{\bQ\bQ^T(\bZ_{\cdot j} - \rho \bX_{\cdot j}) }_{2}^2  \ge \ell C(K+1)^2 + t \Big) \nonumber \\
	& \qquad \qquad \le p  \cdot  \exp\Big(-c\min\Big( \frac{t^2}{C(K+1)^4 \ell}, \frac{t}{C(K+1)^2} \Big)\Big).  
\end{align}
%
%
Subsequently, w.p. $1-O(1/(np)^{10})$, 
\begin{align}\label{eq:lem.proj.2}
\max_{j \in [p]} \, \norm{\bQ\bQ^T(\bZ_{\cdot j} - \rho \bX_{\cdot j}) }_{2}^2 & \leq C(K+1)^2 (\ell + {\sqrt{\ell}}\, {\log(np)}).
\end{align}
\end{lemma}
\noindent Both terms $\| \bZ_{\cdot j} - \rho \bX_{\cdot j} \|_2^2$ and $\| \bU \bU^T ( \bZ_{\cdot j} - \rho \bX_{\cdot j})\|_2^2$ can be bounded by Lemma \ref{lem:projection}: for the first term $\bQ = \Id$, and for the second term $\bQ = \bU$. 
In summary, w.p. $1-O(1/(np)^{10})$, we have 
\begin{align}\label{eq:lem.proj.1.1}
\max_{j \in [p]} \, \norm{\bZ_{\cdot j} - \rho \bX_{\cdot j} }_{2}^2 & \leq C(K+1)^2 (n+ { \sqrt{n}}\, {\log(np)}),
\end{align}
and 
\begin{align}\label{eq:lem.proj.1.2}
\max_{j \in [p]} \, \| \bU \bU^T ( \bZ_{\cdot j} - \rho \bX_{\cdot j})\|_2^2 
& \leq C(K+1)^2 ( r+ {\sqrt{r}}\,{\log(np)}).
\end{align}
Using \eqref{eq:num.1}, \eqref{eq:lem.proj.1.1}, \eqref{eq:lem.proj.1.2}, and Lemmas \ref{lem:thm1.1} and \ref{lem:est.rho} with $k = r$, we conclude that w.p. $1-O(1/(np)^{10})$, 
\begin{align} 
&\max_{j \in [p]}  \| \btZ^k_{\cdot j} - \bhU_k \bhU_k^T\bX_{\cdot j} \|_2^2 
\\
& \leq C (K+1)^4 (\gamma + 1)^2
		\left(\frac{(n+p)(n + {\sqrt{n}}\, \log(np))}{\rho^4 s_r^2 } + \frac{r + {\sqrt{r}}\, \log(np)}{\rho^2}\right) + C \frac{\log(np)}{\rho\, p}. \label{eq:num.2}
\end{align}
%

\vspace{10pt} \noindent
{\em Bounding $\| \bhU_k \bhU_k^T\bX_{\cdot j} - \bX_{\cdot j}\|_2^2$. } 
Recalling $\bX = \bU \bS \bV^T$, we obtain $\bU \bU^T \bX_{\cdot j} = \bX_{\cdot j}$ since $\bU \bU^T$ is the projection onto the column space of $\bX$. 
Therefore,
\begin{align}
\| \bhU_k \bhU_k^T\bX_{\cdot j} - \bX_{\cdot j}\|_2^2 & = \| \bhU_k \bhU_k^T\bX_{\cdot j} - \bU \bU^T \bX_{\cdot j}\|_2^2 \nonumber \\
& \leq  \| \bhU_k \bhU_k^T - \bU \bU^T\|_2^2 ~ \|\bX_{\cdot j}\|_2^2.
\end{align}
Using Property \ref{assumption:bounded}, note that $\|\bX_{\cdot j}\|_2^2 \leq n$. Thus using Lemma \ref{lem:thm1.1} with $k = r$, we have that w.p. at least
$1-O(1/(np)^{10})$, we have
\begin{align}\label{eq:num.3}
\| \bhU_k \bhU_k^T\bX_{\cdot j} - \bX_{\cdot j}\|_2^2 & \leq C \frac{n(n + p)}{\rho^2 s_r^2}.
\end{align} 

\vspace{10pt} \noindent
{\em Concluding. } 
From \eqref{eq:num.0}, \eqref{eq:num.2}, and \eqref{eq:num.3}, we claim w.p. at least $1-O(1/(np)^{10})$
\begin{align}
&\| \bX - \btZ^k \|_{2,\infty}^2 
\\ &\quad \leq C(K+1)^4 (\gamma + 1)^2  
		\left( \frac{(n+p)(n +{\sqrt{n}}\, \log(np))}{\rho^4 s_r^2 } + \frac{r + {\sqrt{r}}\, \log(np)}{\rho^2}\right) + C \frac{\log(np)}{\rho\,p}.
\end{align}
This completes the proof of Lemma \ref{lem:thm1.2}.

{
\subsection{Proof of Lemma \ref{lem:thm1.3}}\label{ssec:thm1.lem3}
To bound $\hs_k$, we recall Weyl's inequality. 

\vspace{5pt}
\begin{lemma}  [Weyl's inequality] \label{lemma:weyl}
Given $\bA, \bB \in \Rb^{m \times n}$, let $\sigma_i$ and $\widehat{\sigma}_i$ be the $i$-th singular values of $\bA$ and $\bB$, respectively, in decreasing order and repeated by multiplicities. 
Then for all $i \in [m \wedge n]$,
\begin{align*}
\abs{ \sigma_i - \widehat{\sigma}_i} &\le \norm{\bA - \bB}_2.
\end{align*}
\end{lemma}
Let $\tilde{s}_k$ be the $k$-th singular value of $\bZ$.  
Then, $\hs_k = (1/\hrho) \tilde{s}_k$ since it is the $k$-th singular value of $\btZ = (1/\hrho) \bZ$. 
By Lemma \ref{lemma:weyl}, we have 
\begin{align}
| \tilde{s}_k - \rho s_k | & \leq \| \bZ - \rho \bX\|_2; 
\end{align}
recall that $s_k$ is the $k$-th singular value of $\bX$.  
As a result, 
\begin{align}\label{eq:weyl.1}
| \hs_k - s_k | & =  \frac{1}{\hrho}  |  \tilde{s}_k - \hrho s_k | \nonumber \\
& \leq  \frac{1}{\hrho}  |  \tilde{s}_k - \rho s_k | + \frac{|\rho - \hrho|}{\hrho} s_k \nonumber \\
& \leq \frac{\| \bZ - \rho \bX\|_2}{\hrho} + \frac{|\rho - \hrho|}{\hrho} s_k.
\end{align}
From Lemma \ref{lem.perturb.1} and Lemma \ref{lem:est.rho}, it follows that w.p. at least $1-O(1/(np)^{10})$, 
\begin{align}\label{eq:weyl.2}
| \hs_k - s_k | & \leq C(K + 1) ( \gamma + 1) \frac{\sqrt{n} + \sqrt{p}}{\rho} + C \frac{\sqrt{\log(np)}}{\sqrt{\rho\, np}} s_k.
\end{align}
This completes the proof of Lemma \ref{lem:thm1.3}.

\subsection{Proof of Lemma \ref{lem:thm1.4}}\label{ssec:thm1.lem4}
We need to bound $\langle \btZ^k (\bhbeta - \btbeta^*), \bvarepsilon\rangle$. To that end, we recall that $\bhbeta = \bhV_k \bhS_k^{-1} \bhU_k^T y$, $\btZ^k =  \bhU_k\bhS_k \bhV_k^T$, and $\by = \bX \btbeta^* + \bvarepsilon$. 
Thus, 
\begin{align}
\btZ^k \bhbeta & = \bhU_k\bhS_k \bhV_k^T \bhV_k \bhS_k^{-1} \bhU_k^T \by  =  \bhU_k \bhU_k^T \bX \btbeta^* +  \bhU_k \bhU_k^T \bvarepsilon.
\end{align}
Therefore, 
\begin{align}\label{eq:lem4.0}
 \langle \btZ^k 
 (\bhbeta - \btbeta^*), \bvarepsilon\rangle 
 &= \langle  \bhU_k\bhU_k^T \bX \btbeta^*, \bvarepsilon\rangle + \langle  \bhU_k\bhU_k^T \bvarepsilon, \bvarepsilon \rangle - \langle  \bhU_k\bhS_k \bhV_k^T\btbeta^*, \bvarepsilon \rangle. 
\end{align}
Now, $\bvarepsilon$ is independent of $\bhU_k, \bhS_k,  \bhV_k$ since $ \btZ^k$ is determined by $\bZ$, which is independent of $\bvarepsilon$. 
As a result, 
\begin{align}\label{eq:thm1.lem3.0}
\Ex\big[\langle \bhU_k\bhU_k^T \bvarepsilon, \bvarepsilon\rangle\big] & = \Ex\big[ \bvarepsilon^T \bhU_k\bhU_k^T \bvarepsilon \big] \nonumber \\
& = \Ex \big[\tr(\bvarepsilon^T \bhU_k\bhU_k^T \bvarepsilon)\big] = \Ex\big[\tr(\bvarepsilon \bvarepsilon^T \bhU_k\bhU_k^T)\big] \nonumber \\
& = \tr(\Ex\big[\bvarepsilon \bvarepsilon^T\big] \bhU_k\bhU_k^T) 
\le C\tr(\sigma^2 \bhU_k\bhU_k^T) \nonumber \\
& = C \sigma^2 \|\bhU_k\|_F^2 = C \sigma^2 k.  
\end{align}
Therefore, it follows that 
\begin{align}\label{eq:thm1.lem3.1}
\Ex [\langle \btZ^k (\bhbeta - \btbeta^*), \bvarepsilon\rangle ] & \le C\sigma^2 k, 
\end{align}
where we used the fact $\Ex[\bvarepsilon] = \bzero$.  
To obtain a high probability bound, using Lemma \ref{lemma:hoeffding_random} it follows that for any $t > 0$
\begin{align}\label{eq:lem4.1}
\Pb\left(  \langle \bhU_k \bhU_k^T \bX \btbeta^*, \bvarepsilon \rangle  \geq t \right) & \leq  \exp\Big( - \frac{c t^2}{n \|\btbeta^*\|_1^2 \sigma^2 } \Big)
\end{align}
due to Assumption \ref{assumption:response_noise}, and 
\begin{align}
\| \bhU_k\bhU_k^T \bX \btbeta^*\|_2 & 
\le \|\bX \btbeta^*\|_2 
 \leq \|\bX\|_{2, \infty} \|\btbeta^*\|_1 \leq \sqrt{n} \|\btbeta^*\|_1;
\end{align}
note that we have used the fact that $\bhU_k\bhU_k^T$ is a projection matrix and $\|\bX\|_{2, \infty} \leq \sqrt{n}$ due to Assumption \ref{assumption:bounded}. 
Similarly, for any $t > 0$
\begin{align}\label{eq:lem4.2}
&\Pb\left(  \langle \bhU_k\bhS_k \bhV_k^T\btbeta^*, \bvarepsilon \rangle  \geq t\right)
 \leq  \exp\Big( - \frac{c t^2}{\sigma^2 (n+\| \btZ^k - \bX\|^2_{2,\infty}) \|\btbeta^*\|_1^2  }  \Big),
\end{align}
due to Assumption \ref{assumption:response_noise}, and
\begin{align}
\| \bhU_k\bhS_k \bhV_k^T\btbeta^*\|_2 & = \| (\btZ^k - \bX)\btbeta^* + \bX \btbeta^*\|_2 ~\leq \| (\btZ^k - \bX) \btbeta^*\|_2 + \|\bX \btbeta^*\|_2 \nonumber \\
& \leq \big(\| \btZ^k - \bX\|_{2,\infty} + \|\bX\|_{2,\infty} \big) \|\btbeta^*\|_1.
\end{align}
Finally, using Lemma \ref{lemma:hansonwright_random} and \eqref{eq:thm1.lem3.1}, it follows that for any $t > 0$
{
\begin{align}\label{eq:lem4.3}
\Pb\left(  \langle \bhU_k\bhU_k^T \bvarepsilon, \bvarepsilon\rangle   \geq \sigma^2 k + t\right) & \leq  \exp\Big( - c \min\Big(\frac{t^2}{k \sigma^4}, \frac{t}{\sigma^2}\Big)\Big),
\end{align}
}
since $ \bhU_k \bhU_k^T $ is a projection matrix and by Assumption \ref{assumption:response_noise}.

From \eqref{eq:lem4.0}, \eqref{eq:lem4.1}, \eqref{eq:lem4.2}, and \eqref{eq:lem4.3}, we conclude that w.p. at least $1-O(1/(np)^{10})$, 
\begin{align}\label{eq:lem4.5}
\langle \btZ^k (\bhbeta - \btbeta^*), \bvarepsilon\rangle  & \leq 
\sigma^2 k + C \sigma \sqrt{\log(np)} \left({ \sigma}\sqrt{k} + \sigma \sqrt{\log(np)} 
+ \|\btbeta^*\|_1 (\sqrt{n} + { \| \btZ^k - \bX\|_{2,\infty} }) \right).
\end{align}
This completes the proof of Lemma \ref{lem:thm1.4}.
~~~

\subsection{Proof of Lemma \ref{lem:subg_matrix}}\label{ssec:matrix.conc}
As mentioned earlier, the proof presented here is a natural extension of that for Theorem 4.6.1 in \cite{vershynin2018high} for the non-isotropic setting. 
Recall that \[ \norm{\bA} \, = \max_{\bx \in S^{p-1}, \by \in S^{n-1}} \langle \bA \bx, \by \rangle, \] where $S^{p-1}, S^{n-1}$ denote the unit spheres in $\Rb^p$ and $\Rb^n$, respectively. 
We start by bounding the quadratic term $\langle \bA \bx, \by \rangle$ for a finite set $\bx, \by$ obtained by placing $1/4$-net on the unit spheres, and then use the bound on them to bound $\langle \bA \bx, \by \rangle$ for all $\bx, \by$ over the spheres. 

\vspace{10pt} \noindent
{\em Step 1: Approximation. } 
We will use Corollary 4.2.13 of \cite{vershynin2018high} to establish a $1/4$-net of $\Nc$ of the unit sphere $S^{p-1}$ with cardinality $\abs{ \Nc } \le 9^p$. 
Applying Lemma 4.4.1 of \cite{vershynin2018high}, we obtain 
\begin{align*}
\| \frac{1}{n} \bA^T \bA - \bSigma \|_2 &\le 2 \, \max_{\bx \in \Nc} \Big| \langle (\frac{1}{n} \bA^T \bA - \bSigma )\bx, \bx  \rangle \Big| = 2 \max_{\bx \in \Nc} \Big| \frac{1}{n} \| \bA \bx \|_2^2 - \bx^T \bSigma \bx \Big|. 
\end{align*} 
To achieve our desired result, it remains to show that 
\[ 
\max_{\bx \in \Nc} \Big| \frac{1}{n} \| \bA \bx \|_2^2 - \bx^T \bSigma \bx \Big| \le \frac{\epsilon}{2}, 
\]
where $\epsilon = K^2 \max(\delta, \delta^2)$. 

\vspace{10pt} \noindent
{\em Step 2: Concentration. } 
Let us fix a unit vector $\bx \in S^{p-1}$ and write 
\[ 
\norm{\bA \bx }_2^2 - \bx^T \bSigma \bx = \sum_{i=1}^n \left( \langle \bA_{i, \cdot}, \bx \rangle^2 - \Ex[ \langle \bA_{i, \cdot}, \bx \rangle^2 ] \right) 
=: \sum_{i=1}^n \left(Y_i^2 - \Ex[Y_i^2] \right). 
\]
Since the rows of $\bA$ are assumed to be independent sub-gaussian random vectors with $\|\bA_{i, \cdot} \|_{\psi_2} \le K$, it follows that $Y_i = \langle \bA_{i, \cdot}, \bx \rangle$ are independent sub-gaussian random variables with $\|Y_i \|_{\psi_2} \le K$. 
Therefore, $Y_i^2 - \Ex[Y_i^2]$ are independent, mean zero, sub-exponential random variables with 
\[ 
\|Y_i^2 - \Ex[Y_i^2] \|_{\psi_1}  \le C \| Y_i^2 \|_{\psi_1}  \le C \| Y_i \|_{\psi_2}^2 \le C K^2. 
\] 
As a result, we can apply Bernstein's inequality (see Theorem \ref{thm:bernstein}) to obtain 
\begin{align*}
\Pb \Big( \Big| \frac{1}{n} \norm{ \bA \bx }_2^2 - \bx^T \bSigma \bx \Big| \ge \frac{\epsilon}{2} \Big) &= \Pb \Big( \Big| \frac{1}{n} \sum_{i=1}^n (Y_i^2 - \Ex[Y_i^2]) \Big| \ge \frac{\epsilon}{2} \Big)
\\ &\le 2 \exp( - c \min\left( \frac{\epsilon^2}{K^4}, \frac{\epsilon}{K^2} \right) n )
\\ &= 2 \exp( - c \delta^2 n) 
\\ &\le 2 \exp(-c C^2 (p + t^2)), 
\end{align*}	
where the last inequality follows from the definition of $\delta$ in \eqref{eq:matrix.conc} and because $(a + b)^2 \ge a^2 + b^2$ for $a,b \ge 0$. 

\vspace{10pt} \noindent
{\em Step 3: Union bound. } 
We now apply a union bound over all elements in the net, 
\begin{align*}
\Pb \Big( \max_{\bx \in \Nc} \Big| \frac{1}{n} \norm{ \bA \bx }_2^2 - \bx^T \bSigma \bx \Big| 
\ge \frac{\epsilon}{2} \Big) 
&\le 9^p \cdot 2 \exp(-c C^2(p + t^2)) \le 2 \exp(-t^2),
\end{align*}
for large enough $C$. 
This concludes the proof. }

\subsection{Proof of Lemma \ref{lem:subG.z}}\label{ssec:subG.z}
Recall that $\bz_i = (\bx_i + \bw_i)\circ \bpi_{i}$, where $\bw_i$ is an independent mean zero subgaussian vector with $\| \bw_i \|_{\psi_2} \leq K$ and $\bpi_i$ is a vector of independent Bernoulli variables with parameter $\rho$. 
Hence, $\Ex[\bz_i - \rho \bx_i] = \bzero$ and is independent across $i \in [n]$. 
The only remaining item is a bound on $\|\bz_i - \rho \bx_i\|_{\psi_2}$.
To that end, note that
\begin{align}
\| \bz_i - \rho \bx_i\|_{\psi_2} & = \| \bx_i \circ \bpi_i + \bw_i \circ \bpi_i - \rho \bx_i \|_{\psi_2} \nonumber \\
& \leq \| \bx_i \circ (\rho \bone - \bpi_i)\|_{\psi_2} + \| \bw_i \circ \bpi_i \|_{\psi_2}. 
\end{align}
Now, $(\rho \bone - \bpi_i)$ is independent, zero mean random vector whose absolute value is bounded by $1$, and is component-wise multiplied by  $\bx_i$ which are bounded in absolute value by $1$ as per Assumption \ref{assumption:bounded}. 
That is, $\bx_i \circ (\rho \bone - \bpi_i)$ is a zero mean random vector where each component is independent and bounded in absolute value by $1$. 
That is, $\|\cdot \|_{\psi_2} \le C$. 

For $\bw_i \circ \bpi_i$, note that $\bw_i$ and $\bpi_i$ are independent vectors and the coordinates of $\bpi_i$ have support  $\{0, 1\}$. 
Therefore, from Lemma \ref{lem:bound.psi2}, it follows that $\| \bw_i \circ \bpi_i \|_{\psi_2} \leq \| \bw_i\|_{\psi_2} \le K$ by Assumption \ref{assumption:cov_noise}. 
The proof of Lemma \ref{lem:subG.z} is complete by choosing a large enough $C$.

\vspace{5pt}
\begin{lemma}\label{lem:bound.psi2}
	Suppose that $\bY \in \mathbb{R}^n$ and $\bP \in \{0, 1\}^n$ are independent random vectors. 
	Then,
	\[	
	    \|\bY \circ \bP \|_{\psi_{2}} \leq \| \bY \|_{\psi_{2}}.	
	\]
\end{lemma}
\begin{proof}
	Given a binary vector $\bP \in \{0, 1\}^n$, let $I_{\bP} = \{ i \in [n]: P_{i} = 1 \}$. 
	Observe that 
	\[
	\bY \circ \bP = \sum_{i \in I_{\bP}} \be_i \otimes \be_i \bY.
	\] 
	Here, $\circ$ denotes the Hadamard product (entry-wise product) of two matrices.
	By definition of the $\psi_{2}$-norm,
	\begin{align*}
		\| \bY \|_{\psi_{2}} 
			&= \sup_{\bu \in \mathbb{S}^{n-1}} \| \bu^T \bY \|_{\psi_{2}}
			= \sup_{\bu \in \mathbb{S}^{n-1}}\inf \{ t > 0: \Ex_{\bY} [ \exp( | \bu^T \bY|^{2} / t^{2} ) ] \leq 2 \}.
	\end{align*}
	Let $\bu_0 \in \mathbb{S}^{n-1}$ denote the maximum-achieving unit vector (such a $\bu_0$ exists because $\inf\{\cdots\}$ is continuous with respect to $\bu$ and $\mathbb{S}^{n-1}$ is compact).
Now,	\begin{align*}
		\| \bY \circ \bP \|_{\psi_{2}}
			&= \sup_{\bu \in \mathbb{S}^{n-1}} \| \bu^T \bY \circ \bP \|_{\psi_{2}}\\
			&= \sup_{\bu \in \mathbb{\bS}^{n-1}} \inf \{ t > 0: 
				\Ex_{\bY,\bP} [ \exp (  | \bu^T \bY \circ \bP |^{2} / t^{2} ) ] \leq 2 \}\\
			&= \sup_{\bu \in \mathbb{S}^{n-1}} \inf \{ t > 0: 
				\Ex_{\bP} [ \Ex_{\bY} [  \exp (  | \bu^T \bY \circ \bP |^{2} / t^{2} )
				 ~|~ \bP ] ] \leq 2 \}\\
			&= \sup_{\bu \in \mathbb{S}^{n-1}} \inf\{ t > 0: 
				\Ex_{\bP} [ \Ex_{\bY} [  \text{exp}(  | \bu^T  \sum_{i \in I_{\bP}} \be_i \otimes \be_i \bY |^{2}
				 / t^{2} ) ~|~ \bP ] ] \leq 2 \}\\
			&= \sup_{\bu \in \mathbb{S}^{n-1}} \inf\{ t > 0: 
				\Ex_{\bP} [ \Ex_{\bY} [  \text{exp} (  |  (  \sum_{i \in I_{\bP}} \be_i \otimes \be_i \bu )^T  \bY |^{2}
				 / t^{2} ) ~|~ \bP ] ] \leq 2 \}.
	\end{align*}
	For any $\bu \in \mathbb{S}^{n-1}$, observe that
	\begin{align*}
		\Ex_{\bY} [  \text{exp} (  |  (  \sum_{i \in I_P} \be_i \otimes \be_i \bu )^T  \bY |^{2}
				 / t^{2} ) ~|~ \bP  ] 
			\leq \Ex_{\bY} [ \exp( | \bu_0^T \bY |^{2}/t^{2} )  ].
	\end{align*}
	Therefore, taking supremum over $\bu \in \mathbb{S}^{n-1}$, we obtain
	\begin{align*}
		\|\bY \circ \bP \|_{\psi_{2}}&\leq \|\bY\|_{\psi_{2}}.
	\end{align*}
\end{proof}

\subsection{Proof of Lemma \ref{lem:bound.cov}}\label{ssec:lem:bound.cov} 
Consider
\begin{align}
\Ex[\btW^T \btW] & = \sum_{i = 1}^n \Ex[(\bz_i - \rho \bx_i) \otimes (\bz_i - \rho \bx_i)] \nonumber \\
&= \sum_{i = 1}^n \Ex[\bz_i \otimes \bz_i] - \rho^2 (\bx_i \otimes \bx_i) \nonumber \\
&=  \sum_{i = 1}^n (\rho - \rho^2) \diag( \bx_i \otimes \bx_i ) + (\rho-\rho^2) \diag(\Ex[\bw_i \otimes \bw_i] ) 
+ \rho^2 \Ex[\bw_i \otimes \bw_i].
\end{align}
Note that $\| \diag(\bX^T \bX)\|_2 \leq n$ due to Assumption \ref{assumption:bounded}.
Using Assumption \ref{assumption:cov_noise}, it follows that $\|\diag(\Ex[\bw_i \otimes \bw_i])\|_2 \leq CK^2$.  
By Assumption \ref{assumption:cov_noise}, we have $\| \Ex[\bw_i \otimes \bw_i] \|_2 \leq \gamma^2$. 
Therefore,
\begin{align}
\| \Ex[\btW^T \btW] \|_2 & \leq Cn (\rho - \rho^2) (K+1)^2 + n \rho^2 \gamma^2.
\end{align}
This completes the proof of Lemma \ref{lem:bound.cov}.

\subsection{Proof of Lemma \ref{lem:est.rho}}\label{ssec:lem.est.rho}
By the Binomial Chernoff bound, for $\alpha > 1$,
\begin{align*}
\Prob{ \widehat{\rho} > \alpha \rho } &\leq \exp\left( - \frac{(\alpha - 1 )^2}{\alpha + 1} np \rho \right)	
\quad\text{and}\quad
\Prob{ \widehat{\rho} <   \rho / \alpha  } \leq \exp \left( - \frac{(\alpha - 1)^2}{2 \alpha^2} np \rho \right).
\end{align*}
By the union bound,
\[
\Prob{  \rho / \alpha  \le \widehat{\rho} \le \alpha \rho } 
\geq 
1 - \Prob{ \widehat{\rho} > \alpha \rho } -  \Prob{ \widehat{\rho} < \rho / \alpha   }.
\]
Noticing $\alpha + 1 < 2 \alpha < 2 \alpha^2$ for all $\alpha > 1$, we obtain the desired bound claimed in Lemma \ref{lem:est.rho}. 
To complete the remaining claim of Lemma \ref{lem:est.rho}, we consider an $\alpha$ that satisfies
\begin{align}
(\alpha -1)^2 & \le C \frac{\log(np)}{\rho np}, 
\end{align}
for a constant $C > 0$. Thus,
\begin{align}
 1 - C \frac{\sqrt{\log(np)}}{\sqrt{\rho np}} 
 \le \alpha 
 \le 1 + C \frac{\sqrt{\log(np)}}{\sqrt{\rho np}}. 
%
\end{align}
%
%
Then, with $\rho \geq c \frac{\log^2 np}{np}$, we have that $\alpha \leq 2$.
Further by choosing $C > 0$ large enough, we have
\begin{align}\label{eq:rho.1}
 \frac{(\rho - \hrho)^2}{\hrho^2} \leq C \frac{\log(np)}{\rho np}.
\end{align}
holds w.p. at least $1-O(1/(np)^{10})$.  
This completes the proof of Lemma \ref{lem:est.rho}.

\subsection{Proof of Lemma \ref{lem:projection}}\label{ssec:lem.projection}
By definition $\bQ \bQ^T \in \Rb^{n \times n}$ is a rank $\ell$ matrix. 
Since $\bQ$ has orthonormal column vectors, the projection operator has $\|\bQ\bQ^T\|_2 = 1$ and $\|\bQ\bQ^T\|^2_F = \ell$. 
For a given $j \in [p]$, the random vector $\bZ_{\cdot j} - \rho \bX_{\cdot j}$ is such that it has zero mean, independent components that are sub-gaussian by
Assumption \ref{assumption:cov_noise}.  
For any $i \in [n], j \in [p]$, we have by property of $\psi_2$ norm, $\|z_{ij} - \rho x_{ij}\|_{\psi_2} \leq \| \bz_i - \rho \bx_i\|_{\psi_2}$ which is bounded by
$C(K+1)$ using Lemma \ref{lem:subG.z}. 
Recall the Hanson-Wright inequality (\cite{vershynin2018high}):

\vspace{5pt}
\begin{thm} [Hanson-Wright inequality] \label{thm:hanson_wright}
Let $\bzeta \in \Rb^n$ be a random vector with independent, mean zero, sub-gaussian coordinates. 
Let $\bA$ be an $n \times n$ matrix. 
Then for any $t > 0$, 
\[
\Pb \left( \abs{ \bzeta^T \bA \bzeta - \Ex [ \bzeta^T \bA \bzeta ] } \ge t \right) \le 2 \exp\Big(-c\min \Big(\frac{t^2}{L^4 \norm{\bA}_F^2}, \frac{t}{L^2 \norm{\bA}_2} \Big)\Big),
\]
where $L = \max_{i\in [n]} \norm{ \zeta_i }_{\psi_2}$. 
\end{thm} 
Now with $\bzeta = \bZ_{\cdot j} - \rho \bX_{\cdot j}$ and the fact that $\bQ^T \bQ = \Id \in \Rb^{\ell \times \ell}$, $\| \bQ \bQ^T \bzeta \|_2^2 = \bzeta^T \bQ \bQ^T \bzeta$.
Therefore, by Theorem \ref{thm:hanson_wright}, for any $t > 0$, 
\begin{align}
\| \bQ \bQ^T \bzeta \|_2^2 & \leq \Ex\big[\bzeta^T \bQ \bQ^T \bzeta\big] + t, 
\end{align}
w.p. at least $1 - \exp\Big(-c \min \big(\frac{t}{C(K+1)^2 }, \frac{t^2}{C(K+1)^4 \ell}\big)\Big)$.  
Now, 
\begin{align}
\Ex\big[\bzeta^T \bQ \bQ^T \bzeta\big] 
& = \sum_{m=1}^\ell \Ex\big[(\bQ_{\cdot m}^T \bzeta)^2\big] \nonumber \\
& \stackrel{(a)}{=} \sum_{m=1}^\ell \Var(\bQ_{\cdot m}^T \bzeta) \nonumber \\
& \stackrel{(b)}{=}\sum_{m=1}^\ell \sum_{i=1}^n \bQ_{im}^2 \Var(\zeta_i) \nonumber \\
& \stackrel{(c)}{\leq} C(K+1)^2 \ell,
\end{align}
where $\bzeta = \bZ_{\cdot j} - \rho \bX_{\cdot j}$, and hence 
(a) follows from $\Ex[\bzeta] = \Ex[\bZ_{\cdot j} - \rho \bX_{\cdot j}] = \bzero$, 
(b) follows from $\bzeta$ having independent components and 
(c) follows from each component of $\bzeta$ having $\psi_2$-norm bounded by $C(K+1)$.
%
Therefore, it follows by union bound that for any $t > 0$, 
\begin{align}
&\Pb \Big( \max_{j \in [p]} \, \norm{\bQ\bQ^T(\bZ_{\cdot j} - \rho \bX_{\cdot j}) }_{2}^2  \ge \ell C(K+1)^2 + t \Big)
\\ & \quad \le p  \cdot \exp\Big(-c\min\Big( \frac{t^2}{C(K+1)^4 \ell}, \frac{t}{C(K+1)^2} \Big)\Big). 
\end{align}
%
%
%
This completes the proof of Lemma \ref{lem:projection}.

%

%% file: content/proof_generalization.tex
\section{Proof of Theorem \ref{thm:test_error}}\label{sec:proof_thm_4.2}
Recall that $\bXp$ and $\bZp$ denote the latent and observed testing covariates, respectively. 
We denote the SVD of the former as $\bXp = \bUp \bSp \bVp^T$.
Let $s'_{\ell}$ be the $\ell$-th singular value of $\bXp$.
Further, recall that $\btZp = (1/\hrho^\prime) \bZp$, and its rank $\ell$ truncation is denoted as $\btZ^{\prime \ell}$. 
Our interest is in bounding $\| \btZp{}^\ell \bhbeta - \bXp \btbeta^*\|_2$. Towards this, consider
\begin{align}\label{eq:thm2.1}
\| \btZp{}^\ell \bhbeta - \bXp \btbeta^*\|_2^2 & = \| \btZp{}^\ell \bhbeta - \btZp{}^\ell \btbeta^* + \btZp{}^\ell \btbeta^* - \bXp \btbeta^*\|_2^2\nonumber \\
& \leq 2 \| \btZp{}^\ell \big( \bhbeta - \btbeta^* \big) \|_2^2 + 2 \| (\btZp{}^\ell - \bXp) \btbeta^*\|_2^2.
\end{align}
We shall bound the two terms on the right hand side of \eqref{eq:thm2.1} next. 

\vspace{10pt} \noindent
{\em Bounding $ \| \btZp{}^\ell \big( \bhbeta - \btbeta^* \big) \|_2^2$. } Since $\btZp{}^\ell=  (1/\hrho^\prime) \bZp{}^\ell$, we have 
\begin{align}\label{eq:thm2.2}
 \| \btZp{}^\ell \big( \bhbeta - \btbeta^* \big)  \|_2^2 & = \frac{1}{(\hrho^\prime)^2} \|\bZp{}^\ell \big( \bhbeta - \btbeta^* \big)  \|_2^2 \nonumber \\
 & = \frac{1}{(\hrho^\prime)^2} \| \big(\bZp{}^\ell - \rho \bXp + \rho \bXp\big) \big( \bhbeta - \btbeta^* \big)  \|_2^2 \nonumber \\
 & \leq \frac{2}{(\hrho^\prime)^2} \|\big(\bZp{}^\ell - \rho \bXp) \big( \bhbeta - \btbeta^* \big)  \|_2^2 + 2\Big(\frac{\rho}{\hrho^\prime}\Big)^2  \|\bXp \big( \bhbeta - \btbeta^* \big) \|_2^2. 
\end{align}
Now, note that $\| \bZp - \bZp{}^\ell\|_2$ is
the $(\ell+1)$-st largest singular value of $\bZp$. 
Therefore, by Weyl's inequality (Lemma \ref{lemma:weyl}), we have for any $\ell \geq \rp$,  
\begin{align}\label{eq:thm2.3}
\| \bZp - \bZp{}^\ell\|_2& \leq \|\bZp - \rho \bXp\|_2.
\end{align} 
In turn, this gives
\begin{align}\label{eq:thm2.4}
\| \bZp{}^\ell - \rho \bXp\|_2 & \leq \| \bZp{}^\ell - \bZp\|_2 + \|\bZp - \rho \bXp\|_2 \leq 2 \|\bZp - \rho \bXp\|_2.
\end{align}
Thus, we have 
\begin{align}\label{eq:thm2.5}
\|\big(\bZp{}^\ell - \rho \bXp) \big( \bhbeta - \btbeta^* \big)  \|_2^2 & \leq 4 \|\bZp - \rho \bXp\|_2^2 ~ \| \bhbeta - \btbeta^* \|_2^2.
\end{align}
Recall that $\bH$ and $\bH_\perp$ span the rowspace and nullspace of $\bX$, respectively; similarly, recall that $\bHp$ and $\bHp_\perp$ are defined analogously with respect to $\bXp$. 
As a result, 
\begin{align}\label{eq:thm2.6}
\|\bXp \big( \bhbeta - \btbeta^* \big) \|_2^2 
& = \| \bXp (\bH + \bH_{\perp}) \big( \bhbeta - \btbeta^* \big) \|_2^2 \nonumber 
\\
& \le 2 \| \bXp \bH \big( \bhbeta - \btbeta^* \big) \|_2^2 + 
2 \| \bXp \bH_\perp \big( \bhbeta - \btbeta^* \big) \|_2^2\nonumber 
\\
& \leq 2 \| \bXp\|_2^2 ~ \| \bH \big( \bhbeta - \btbeta^* \big) \|_2^2 + 2 \| \bXp\|_2^2 ~ \| \bHp \bH_\perp \big( \bhbeta - \btbeta^* \big) \|_2^2.
\end{align}
Let $\bhH_r = \bhV_r \bhV_r^T$ denote the projection matrix onto the rowspace of $\btZ^r$. 
Thus, 
\begin{align}\label{eq:thm2.7}
\| \bH \big( \bhbeta - \btbeta^* \big) \|_2^2 & = 
\| \big(\bH -\bhH_r + \bhH_r\big) \big( \bhbeta - \btbeta^* \big) \|_2^2 \nonumber \\
& \leq 2 \| \bH -\bhH_r\|_2^2 ~ \|\bhbeta - \btbeta^* \|_2^2 + 2 \| \bhH_r\big( \bhbeta - \btbeta^* \big) \|_2^2.
\end{align}
From \eqref{eq:thm1.2} and above, we obtain 
\begin{align}\label{eq:thm2.8}
\| \bH \big( \bhbeta - \btbeta^* \big) \|_2^2 & \leq 
C \| \bH -\bhH_r\|_2^2 ~ \|\bhbeta - \btbeta^* \|_2^2  \nonumber \\
& \qquad  +\frac{C}{\hs_r^2} \Big( \| \bX - \btZ^r \|_{2,\infty}^2 \|\btbeta^*\|_1^2 +\langle \btZ^r (\bhbeta - \btbeta^*), \bvarepsilon\rangle \Big).
\end{align}
Thus, 
\begin{align}
\|\bXp \big( \bhbeta - \btbeta^* \big) \|_2^2 & \leq
C \|\bXp\|_2^2 ~\| \bH -\bhH_r\|_2^2 ~\|\bhbeta - \btbeta^* \|_2^2  \nonumber \\
& \qquad  +\frac{C\|\bXp\|_2^2 }{\hs_r^2} \Big( \| \bX - \btZ^r \|_{2,\infty}^2 \|\btbeta^*\|_1^2 +\langle \btZ^r (\bhbeta - \btbeta^*), \bvarepsilon\rangle \Big)
\\ &\qquad + C \| \bXp\|_2^2 ~ \| \bHp \bH_\perp  \|_2^2 ~\| \bhbeta - \btbeta^* \|_2^2. \label{eq:thm2.6.a}
\end{align}
In summary, plugging \eqref{eq:thm2.5} and \eqref{eq:thm2.6.a} into \eqref{eq:thm2.2}, we have 
\begin{align}
 \| \btZ^{\prime, \ell} \big( \bhbeta - \btbeta^* \big)  \|_2^2 & \leq %
 \frac{C}{(\hrho^\prime)^2}   \|\bZp - \rho \bXp\|_2^2 ~ \| \bhbeta - \btbeta^* \|_2^2  \nonumber \\
&\quad + C \Big(\frac{\rho}{\hrho^\prime}\Big)^2 \|\bX'\|_2^2 ~ \| \bH -\bhH_r\|_2^2 ~ \|\bhbeta - \btbeta^* \|_2^2  \nonumber \\
& \quad
	+  \frac{C\rho^2\|\bX'\|_2^2}{(\hrho^\prime)^2 \hs_r^2} \Big( \| \bX - \btZ^r \|_{2,\infty}^2 \|\btbeta^*\|_1^2 +\langle \btZ^r (\bhbeta - \btbeta^*), \bvarepsilon\rangle \Big)
	\\
	&\quad + C \Big(\frac{\rho}{\hrho^\prime}\Big)^2 \| \bXp\|_2^2 ~ \| \bHp \bH_\perp  \|_2^2 ~\| \bhbeta - \btbeta^* \|_2^2. \label{eq:thm2.9}
 \end{align}
 
\vspace{10pt} \noindent
{\em Bounding $ \| (\btZp{}^\ell - \bXp) \btbeta^*\|_2^2$. } Using inequality \eqref{eq:2.inf.ineq},
\begin{align}\label{eq:thm2.14}
\| (\btZp{}^\ell - \bXp) \btbeta^*\|_2^2 & \leq \|\btZp{}^\ell - \bXp\|_{2,\infty}^2 \|\btbeta^*\|_1^2. 
\end{align}

\vspace{10pt} \noindent
{\em Combining. } 
Incorporating \eqref{eq:thm2.9} and \eqref{eq:thm2.14} into \eqref{eq:thm2.1} with $\ell = \rp$ yields
\begin{align}\label{eq:thm2.1_new}
\| \btZ^{\prime \rp} \bhbeta - \bXp \btbeta^*\|_2^2 
&\leq \Delta_1 + \Delta_2 + \Delta_3, 
\end{align}
where
\begin{align*}
\Delta_1 &= 
\frac{C}{(\hrho^\prime)^2}   \|\bZp - \rho \bXp\|_2^2 ~ \| \bhbeta - \btbeta^* \|_2^2 + C \left( \frac{\rho s^{\prime}_1}{\hrho^\prime} \right)^2  \| \bH -\bhH_r\|_2^2 ~ \|\bhbeta - \btbeta^* \|_2^2  \nonumber \\
& \quad   + 2 \| \bXp - \btZ^{\prime \rp} \|_{2,\infty}^2 \|\btbeta^*\|_1^2, \\
\Delta_2 &= 
C \left( \frac{\rho s^\prime_1}{\hrho^\prime \hs_r} \right)^2 \left( \| \bX - \btZ^r \|_{2,\infty}^2 \|\btbeta^*\|_1^2 +\langle \btZ^r (\bhbeta - \btbeta^*), \bvarepsilon\rangle \right), 
\\
\Delta_3 &= C \left( \frac{\rho s^{\prime}_1}{\hrho^\prime} \right)^2  \| \bHp \bH_\perp \|_2^2 ~ \|\bhbeta - \btbeta^* \|_2^2. 
\end{align*}
Note that \eqref{eq:thm2.1_new} is a deterministic bound. 
We will now proceed to bound $\Delta_1$ and $\Delta_2$, first in high probability then in expectation. 

\vspace{10pt} \noindent
{\em Bound in high-probability. }
We first bound $\Delta_1$. 
First we note that by adapting Lemma \ref{lem:est.rho} with $\hrho^\prime$ in place of $\hrho$, we obtain
w.p. at least $1-O(1/(mp)^{10})$, 
\begin{align} \label{eq:rho_stuff}
    \rho/2 \leq \hrho^\prime \leq \rho.  
\end{align} 
By adapting Lemma \ref{lem.perturb.1} for $\bZp, \bXp$ in place of $\bZ, \bX$, we have w.p. at 
least $1-O(1/(mp)^{10})$, 
\begin{align}
\| \bZp - \rho \bXp \|_2 & \leq C(K+1)(\gamma+1) (\sqrt{m} + \sqrt{p}).
\end{align}
Hence, using Theorem \ref{thm:param_est} and  \eqref{eq:rho_stuff}, we have w.p. at least $1-O(1/((n \wedge m) p)^{10})$
\begin{align} 
& \frac{1}{ (\hrho^\prime)^2 m}\| \bZp - \rho \bXp \|^2_2 ~ \| \bhbeta - \btbeta^* \|_2^2
\\ & \leq 
C(K, \gamma, \sigma) 
\frac{ \log (np)}{ \rho^2 }  
\left( 1 \vee \frac{p}{m} \right)
\cdot
\left\{
	\frac{ \| \btbeta^* \|_2^2}{\snr^2}
	+ 
	\frac{ \sqrt{n} \| \btbeta^* \|_1}{(n+p) \snr^2}
	+
	\frac{r \| \btbeta^* \|_1^2}{(n+p) \snr^2}
	+
	\frac{\| \btbeta^* \|_1^2}{\snr^4}
\right\}, \label{eq:thm2.10.useful}
%
\end{align}
where $C(K, \gamma, \sigma) = C (K+1)^6 (\gamma+1)^4 (\sigma^2 + 1)$. 
%
%
Next, observe that $s'_1 = O(\sqrt{mp})$, which follows from Assumption \ref{assumption:bounded}.
Using this bound and recalling Lemma \ref{lem:thm1.1}, \eqref{eq:snr_basic_ineq}, and Theorem \ref{thm:param_est}, it follows that w.p. at least $1-O(1/(np)^{10})$
\begin{align} 
&\left( \frac{\rho s^{\prime}_1}{\hrho^\prime} \right)^2 \frac{1}{m}\| \bH - \bhH_r \|_2^2 ~ \| \bhbeta - \btbeta^* \|_2^2 
\\ & \qquad \leq C(K, \gamma, \sigma) \log(np)
\cdot
\left\{
	\frac{p \| \btbeta^* \|_2^2}{\snr^4}
	+ 
	\frac{ \sqrt{n} p \| \btbeta^* \|_1}{(n+p) \snr^4}
	+
	\frac{r p \| \btbeta^* \|_1^2}{(n+p) \snr^4}
	+
	\frac{p \| \btbeta^* \|_1^2}{\snr^6}
\right\}
\\ & \qquad \leq C(K, \gamma, \sigma) \log(np)
\cdot
\left\{
	\frac{p \| \btbeta^* \|_2^2}{\snr^4}
	+ 
	\frac{ \sqrt{n} \| \btbeta^* \|_1}{ \snr^4}
	+
	\frac{r  \| \btbeta^* \|_1^2}{ \snr^4}
	+
	\frac{p \| \btbeta^* \|_1^2}{\snr^6}
\right\}. 
\label{eq:thm2.11}
\end{align}
Next, we adapt Lemma~\ref{lem:thm1.2} for $\btZp, \bXp$ in place of $\btZ, \bX$ with $\ell = \rp$. %
If $\rho \geq c (mp)^{-1} \log^2 (mp)$, then w.p. at least $1-O(1/(m p)^{10})$
\begin{align}
&\frac{1}{m}\| \bXp - \btZ^{\prime \rp} \|_{2,\infty}^2 \|\btbeta^*\|_1^2 \\
& \quad \le 
\frac{C(K, \gamma)}{m} \cdot
\left\{ \frac{(m+p)(m +\sqrt{m}\, \log(mp))}{\rho^4 (s'_r)^2 } + \frac{r' + \sqrt{r'}\, \log(mp)}{\rho^2}  + C \frac{\log(mp)}{\rho \, p} \right \}\|\btbeta^*\|_1^2 \\
&\quad \le 
\frac{C(K, \gamma) \log(mp)}{\rho^2}  \cdot 
\left\{ \frac{1}{\snr_{\test}^2 } + \frac{r'}{m}\right\} \|\btbeta^*\|_1^2, \label{eq:test_de_noise_cov}
\end{align}
where $C(K, \gamma) = C(K+1)^4(\gamma+1)^2$. 
Note that the above uses the inequality
$$ \frac{m + p}{\rho^2(s'_{r'})^2} \le \frac{1}{\snr_{\test}^2},$$ which follows from the definition of $\snr_{\test}^2$ in~\eqref{eq:snr_test}.

%
Next, we bound $\Delta_2$. 
As per \eqref{eq:thm1.f.1}, we have w.p. at least $1-O(1/(np)^{10})$,
\begin{align}\label{eq:thm2.12}
& \| \bX - \btZ^r \|_{2,\infty}^2 \|\btbeta^*\|_1^2 
+\langle \btZ^r (\bhbeta - \btbeta^*), \bvarepsilon\rangle \nonumber \\
& \leq 
C \| \bX - \btZ^r \|_{2,\infty}^2  \|\btbeta^*\|_1^2 
+  C \sigma^2 (r + \log(np)) + C \sigma \sqrt{n\log(np)} \|\btbeta^*\|_1.
\end{align}
Recalling Lemma \ref{lem:thm1.2} and the definition of $\snr$, we have that w.p. at least $1-O(1/(np)^{10})$, 
\begin{align}\label{eq:thm2.16}
\| \bX - \btZ^r \|_{2,\infty}^2 
& \leq 
C(K, \gamma) \frac{\log(np)}{\rho^2} 
\cdot \left\{ \frac{n}{\snr^2 } + r \right\}.
\end{align}
Using \eqref{eq:singualr_val_emp_pop}, \eqref{eq:snr_basic_ineq}, \eqref{eq:rho_stuff}, we have
\begin{align}\label{eq:conditioning_num_bound}
    \left( \frac{\rho s^\prime_1}{\hrho^\prime \hs_r} \right)^2
    \le  \frac{C (s^\prime_1)^2 \rho^2}{\snr^2 (n + p)}. 
\end{align}
Therefore,~\eqref{eq:rho_stuff},~\eqref{eq:thm2.12},~\eqref{eq:thm2.16},~\eqref{eq:conditioning_num_bound}, and the bound $s'_1 = O(\sqrt{mp})$ altogether imply that w.p. at least $1-O(1/((n \wedge m) p)^{10})$, 
\begin{align} 
\frac{\Delta_2}{m} &\le 
\frac{C(K, \gamma) (s'_1)^2 \log(np) \rho^2}{ m (n+p) \snr^2}
\cdot
\left\{ 
	\frac{\sigma^2 r \| \btbeta^* \|_1^2}{\rho^2} 
	+ \sigma \sqrt{n} \| \btbeta^* \|_1
	+ \frac{n \| \btbeta^* \|_1^2}{\rho^2 \snr^2}
\right\}
\\ &\le C(K, \gamma, \sigma) \log(np)
\cdot 
\left\{ 
	\frac{r \| \btbeta^* \|_1^2}{\snr^2} 
	+ \frac{\sqrt{n} \| \btbeta^* \|_1}{\snr^2}
	+ \frac{n \| \btbeta^* \|_1^2}{\snr^4} 
\right \},  \label{eq:Delta2.1} 
\end{align}
where $C(K, \gamma, \sigma)$ is defined as in \eqref{eq:thm2.10.useful}. 
Moving on, we observe $\| \btbeta^* \|_2 \le  \| \btbeta^* \|_1$, and recall the assumption $\snr \geq C(K+1)(\gamma+1)$.
With these in mind, we incorporate \eqref{eq:thm2.10.useful}, \eqref{eq:thm2.11}, \eqref{eq:test_de_noise_cov}, and \eqref{eq:Delta2.1} into \eqref{eq:thm2.1_new} and simplify to establish
\begin{align}
	&\frac{\Delta_1 + \Delta_2}{m} 
	\le 
	C(K, \gamma, \sigma) \log((n \vee m)p)  
	\\
	&  \qquad \qquad \quad ~
    \cdot \Bigg\{
    \frac{\sqrt{n}}{\snr^2}\| \btbeta^* \|_1
    + 
    \left(
    	\frac{r \left(1 \vee \frac{p}{m} \right)}{\rho^2 \snr^2}
	+
	\frac{r'}{\snr^2_\test \wedge m}
	+ 
     	\frac{ n \vee p }{\snr^4}
    \right) \| \btbeta^* \|_1^2
    \Bigg\}. \label{eq:delta1_delta2} 
\end{align}
Finally, we bound $\Delta_3$.
Following the arguments that led to \eqref{eq:thm2.11}, we obtain w.p. at least $1-O(1/(np)^{10})$
\begin{align}
	\frac{\Delta_3}{m} &\le C \cdot p \cdot \delta_\beta \cdot \| \bHp \bH_\perp \|_2^2. \label{eq:thm2_final_6}
\end{align} 
Combining \eqref{eq:delta1_delta2} and \eqref{eq:thm2_final_6} concludes the high-probability bound.


\vspace{10pt} \noindent
{\em Bound in expectation.} 
Here, we assume that $\{\langle \bx_i, \bbeta^* \rangle \in [-b,b]: i > n\}$. 
As such, we enforce $\{\hy_i \in [-b,b]: i > n\}$.
With \eqref{eq:thm2.1_new}, this yields
\begin{align}\label{eq:thm4.2_1}
\text{MSE}_{\test} \le \frac{1}{m}\| \btZ^{\prime \rp} \bhbeta - \bXp \btbeta^*\|_2^2 
\le 
\frac{1}{m}\left( \Delta_1 + \Delta_2 + \Delta_3 \right).
\end{align}
We define $\Ec$ as the event such that the bounds in 
\eqref{eq:rho_stuff}, \eqref{eq:thm2.10.useful}, \eqref{eq:thm2.11}, \eqref{eq:test_de_noise_cov}, \eqref{eq:thm2.16}, and Lemma \ref{eq:lem.thm.1.3.0} hold. 
Thus, if $\Ec$ occurs, then combining \eqref{eq:thm2.10.useful}, \eqref{eq:thm2.11}, \eqref{eq:test_de_noise_cov},
and using the property $\| \btbeta^* \|_2 \le  \| \btbeta^* \|_1$ and assumption $\snr \geq C(K+1)(\gamma+1)$ gives  
\begin{align} 
\frac{\Ex[\Delta_1 | \cE]}{m}
&\le 
C(K, \gamma, \sigma) \log(n_{\max} p) 
\cdot 
\Bigg\{
\frac{\sqrt{n}}{\snr^2} \left( \frac{1}{\snr^2} + \frac{1 \vee \frac{p}{m}}{\rho^2 (n+p)} \right) \| \btbeta^* \|_1
\\
&\qquad \qquad  \qquad \qquad \qquad\qquad
    + \left(
    	\frac{r \left(1 \vee \frac{p}{m} \right)}{\rho^2 \snr^2}
	+
	\frac{r'}{\snr^2_\test \wedge m}
	+ 
     	\frac{ n \vee p }{\snr^4}
    \right) \| \btbeta^* \|_1^2
    \Bigg\}. 
\label{eq:delta_1_expectation}
\end{align}
Next, we bound $\Ex[\Delta_2 | \cE]$.
To do so, observe that $\bvarepsilon$ is independent of the event $\cE$. Thus, by \eqref{eq:thm1.lem3.0}, we have
\begin{align}
\Ex[ \langle \btZ^r (\bhbeta - \btbeta^*), \bvarepsilon\rangle | \cE]
   &= \Ex[\langle  \bhU_r\bhU_r^T \bX \btbeta^*, \bvarepsilon\rangle + \langle  \bhU_r\bhU_r^T \bvarepsilon, \bvarepsilon\rangle - \langle  \bhU_r\bhS_r \bhV_r^T\btbeta^*, \bvarepsilon \rangle | \cE] \nonumber\\
   &= \Ex[ \langle  \bhU_r\bhU_r^T \bvarepsilon, \bvarepsilon\rangle | \cE]
   \le C \sigma^2 r.
\end{align}
Combining the above inequality with \eqref{eq:thm2.16},
\begin{align} 
\frac{\Ex[\Delta_2 | \cE]}{m} &\le 
C(K, \gamma, \sigma) \log(np)
\cdot
\left\{
	\left(
	\frac{r}{\snr^2}
	+ \frac{n}{\snr^4} 
	\right) \| \btbeta^* \|_1^2
\right\}. \label{eq:Delta2_exp}
\end{align} 
Next, \eqref{eq:thm2_final_6} yields  
\begin{align}
	\frac{\Ex[\Delta_3 | \cE]}{m} &\le C \cdot p \cdot \delta_\beta 
	\cdot \| \bHp \bH_\perp \|_2^2. \label{eq:Delta3_exp}
\end{align} 
Due to truncation, observe that $\text{MSE}_\test$ is always bounded above by $4b^2$. 
Thus,
\begin{align}\label{eqref:thm4.2_exp_meta}
\Ex[\text{MSE}_{\test}] 
&\le \Ex[\text{MSE}_{\test} | \cE] + \Ex[\text{MSE}_{\test} | \cE^c] ~ \Pb(\cE^c) \\
&\le \frac{1}{m} \Ex[\Delta_1 + \Delta_2 + \Delta_3 | \cE] + {C b^2}
\left(1/(np)^{10} + 1/(mp)^{10} \right).
\end{align}
Plugging \eqref{eq:delta_1_expectation}, \eqref{eq:Delta2_exp}, \eqref{eq:Delta3_exp} into \eqref{eqref:thm4.2_exp_meta} and simplifying completes the proof.

%% file: content/proof_prob.tex
\section{Helpful Concentration Inequalities} \label{sec:proofs_helpful_lemmas} 
In this section, we state and prove a number of helpful concentration inequalities used to establish our primary results.

\begin{lemma}\label{lem:MGF_upper}
	Let $X$ be a mean zero, sub-gaussian random variable. Then for any $\lambda \in \Rb$, 
	\[	\Ex \exp\left( \lambda X \right) \leq \exp\left( C \lambda^2  \norm{X}_{\psi_{2}}^2 \right).	\]
\end{lemma}

\begin{lemma}\label{lem:ind_sum}
	Let $X_1, \dots, X_n$ be independent, mean zero, sub-gaussian random variables. Then,
	\[	\| \sum_{i=1}^n X_i \|^2_{\psi_{2}}  \leq C  \sum_{i=1}^n  \norm{ X_i }_{\psi_{2}}^2 .	\]
\end{lemma}

\begin{thm} [Bernstein's inequality] \label{thm:bernstein}
Let $X_1, \dots, X_n$ be independent, mean zero, sub-exponential random variables. Then, for every $t \ge 0$, we have
\begin{align*}
	\Pb \Big( \Big| \sum_{i=1}^n X_i \Big| \ge t \Big) &\le 2 \exp \Big( - c \min \Big( \frac{t^2}{\sum_{i=1}^n \norm{X_i }_{\psi_1}^2}, \frac{t}{\max_i \norm{ X_i }_{\psi_1}} \Big) \Big),
\end{align*}	
where $c > 0$ is an absolute constant. 
\end{thm}

\vspace{5pt}
\begin{lemma} [Modified Hoeffding Inequality] \label{lemma:hoeffding_random} 
Let $\bX \in \Rb^n$ be random vector with independent mean-zero sub-Gaussian random coordinates with $\| X_i \|_{\psi_2} \le K$.
Let $\ba \in \Rb^n$ be another random vector that satisfies $\|\ba\|_2 \le b$ almost surely for some constant $b \ge 0$.
Then for all $t \ge 0$, 
\begin{align*}
	\Pb \Big( \Big| \sum_{i=1}^n a_i X_i\Big| \ge t \Big) \le 2 \exp\Big(-\frac{ct^2}{K^2 b^2} \Big),
\end{align*}
where $c > 0$ is a universal constant. 
\end{lemma}

\begin{proof}
Let $S_n = \sum_{i=1}^n a_i X_i$. Then applying Markov's inequality for any $\lambda > 0$, we obtain
\begin{align*}
	\Pb \left( S_n \ge t \right) &= \Pb \left( \exp(\lambda S_n) \ge \exp(\lambda t) \right)
	\\ & \le \Ex \left[ \exp(\lambda S_n) \right] \cdot \exp(-\lambda t)
	\\ &= \Ex_{\ba} \left[ \Ex\left[ \exp(\lambda S_n) ~|~ \ba\right] \right] \cdot \exp(-\lambda t).
\end{align*}
Now, conditioned on the random vector $\ba$, observe that
\begin{align*}
	\Ex\left[ \exp(\lambda S_n)  \right] = \prod_{i=1}^n \Ex \left[ \exp( \lambda a_i X_i) \right] 
	\le \exp(CK^2 \lambda^2 \| \ba \|_2^2) \le  \exp(CK^2 \lambda^2 b^2),
\end{align*}
where the equality follows from conditional independence, the first inequality by Lemma \ref{lem:MGF_upper}, and the final inequality by assumption. Therefore, 
\begin{align*}
	\Pb \left( S_n \ge t \right) &\le \exp( CK^2 \lambda^2 b^2 - \lambda t). 
\end{align*} 
Optimizing over $\lambda$ yields the desired result:
\begin{align*}
	\Pb \left( S_n \ge t \right) &\le \exp \Big( - \frac{ct^2}{K^2 b^2} \Big). 
\end{align*} 
Applying the same arguments for $- \langle \bX, \ba \rangle$ gives a tail bound in the other direction.
\end{proof}

\vspace{5pt}
\begin{lemma} [Modified Hanson-Wright Inequality] \label{lemma:hansonwright_random} 
Let $\bX \in \Rb^n$ be a random vector with independent mean-zero sub-Gaussian coordinates with $\|X_i \|_{\psi_2} \le K$. 
Let $\bA \in \Rb^{n \times n}$ be a random matrix satisfying $\|\bA\|_2  \le a$ and $\|\bA\|_F^2 \, \le b$ almost surely for some $a, b \ge 0$.
Then for any $t \ge 0$,
\begin{align*}
	\Pb \left( \abs{ \bX^T \bA \bX - \Ex[\bX^T \bA \bX] } \ge t \right) &\le 2 \cdot \exp \Big( -c \min\Big(\frac{t^2}{K^4 b}, \frac{t}{K^2 a} \Big) \Big). 
\end{align*}
\end{lemma} 

\begin{proof}
The proof follows similarly to that of Theorem 6.2.1 of \cite{vershynin2018high}. Using the independence of the coordinates of $X$, we have the following useful diagonal and off-diagonal decomposition: 
\begin{align*}
	\bX^T \bA \bX - \Ex[\bX^T \bA \bX] &= \sum_{i=1}^n \left(A_{ii} X_i^2 - \Ex[ A_{ii} X_i^2] \right) + \sum_{i \neq j} A_{ij} X_i X_j. 
\end{align*}
Therefore, letting 
\[ p = \Pb \left( \bX^T \bA X - \Ex[\bX^T \bA \bX]  \ge t \right), \]
we can express 
\[ p \le \Pb \Big( \sum_{i=1}^n \left(A_{ii} X_i^2 - \Ex[ A_{ii} X_i^2] \right) \ge t/2 \Big) + \Pb \Big( \sum_{i \neq j} A_{ij} X_i X_j \ge t/2 \Big) =: p_1 + p_2. \]
We will now proceed to bound each term independently. 

\vspace{10pt} \noindent 
{\em Step 1: diagonal sum. } 
Let $S_n = \sum_{i=1}^n (A_{ii} X_i^2 - \Ex[A_{ii} X_i^2])$. Applying Markov's inequality for any $\lambda > 0$, we have
\begin{align*}
p_1 &= \Pb\left( \exp(\lambda S_n) \ge \exp(\lambda t/2) \right) 
\\ &\le \Ex_{\bA} \Ex\left[ \left[ \exp(\lambda S_n) ~|~ \bA \right] \right] \cdot \exp(-\lambda t/2). 
\end{align*}
Since the $X_i$ are independent, sub-Gaussian random variables, $X_i^2 - \Ex[X_i^2]$ are independent mean-zero sub-exponential random variables, satisfying
\[ 
\norm{ X_i^2 - \Ex[X_i^2] }_{\psi_1} \le C \norm{X_i^2}_{\psi_1} \le C \norm{X_i}_{\psi_2}^2 \le C K^2. 
\]
Conditioned on $\bA$ and optimizing over $\lambda$ using standard arguments, yields
\[ p_1 \le \exp \Big( -c\min\Big( \frac{t^2}{K^4 b}, \frac{t}{K^2a} \Big)\Big). \] 

\noindent 
{\em Step 2: off-diagonals. } Let $S = \sum_{i \neq j} A_{ij} X_i X_j$. Again, applying Markov's inequality for any $\lambda > 0$, we have
\begin{align*}
	p_2 &= \Pb \left( \exp(\lambda S) \ge \exp( \lambda t / 2) \right) \le \Ex_{\bA} \left[ \Ex \left[ \exp(\lambda S) ~|~ \bA \right] \right] \cdot \exp(-\lambda t / 2). 
\end{align*}
Let $\bg$ be a standard multivariate gaussian random vector. Further, let $\bX'$ and $\bg'$ be independent copies of $\bX$ and $\bg$, respectively. Conditioning on $\bA$ yields
\begin{align*}
	\Ex\left[ \exp(\lambda S) \right] &\le \Ex\left[ \exp( 4 \lambda \bX^T \bA \bX' ) \right] \qquad \text{(by Decoupling Remark 6.1.3 of \cite{vershynin2018high})}
	\\ &\le \Ex \left[ \exp( C_1 \lambda \bg^T \bA \bg') \right] \qquad \text{(by Lemma 6.2.3 of \cite{vershynin2018high})}
	\\ &\le \exp(C_2 \lambda^2 \norm{\bA}_F^2)  	\qquad \text{(by Lemma 6.2.2 of \cite{vershynin2018high})}
	\\ &\le \exp(C_2 \lambda^2 b),
\end{align*}
where $|\lambda| \, \le c / a$. Optimizing over $\lambda$ then gives
\[ p_2 \le \exp \Big(-c\min\Big( \frac{t^2}{K^4 b}, \frac{t}{K^2 a} \Big) \Big). \]

\noindent 
{\em Step 3: combining. } Putting everything together completes the proof.  
\end{proof}


%% file: content/proof_hypo_test.tex
\section{Proof of Theorem~\ref{thm:hypo}} \label{sec:proof.hypo_test}
{\em Type I error. } 
We first bound the Type I error, which anchors on Lemma \ref{lemma:type1.1}, stated below. The proof of Lemma \ref{lemma:type1.1} can be found in Appendix \ref{sec:type1_lem}.

\begin{lemma} \label{lemma:type1.1}
	Suppose $H_0$ is true. Then, 
\begin{align} \label{eq:type1.1} 
	\htau &= \|(\bH - \bhH^k) \bhH^{\prime \ell}\|_F^2 + 
	\|(\bI - \bH) (\bhH^{\prime \ell} - \bHp)\|_F^2
	\\ &\quad + 2 \langle (\bH - \bhH^k) \bhH^{\prime \ell}, (\bI - \bH) \bhH^{\prime \ell} \rangle_F.
\end{align} 
\end{lemma} 

We proceed to bound each term on the right-hand side of \eqref{eq:type1.1} independently. 

\vspace{10pt} \noindent 
{\em Bounding $\|(\bH - \bhH^k) \bhH^{\prime \ell}\|_F^2$. } 
By Lemma \ref{lem:thm1.1}, we have w.p. at least $1- \alpha_1$, 
{\small
\begin{align}  
	\| (\bhH^k - \bH) \bhH^{\prime \ell} \|_F^2 
	&\le \| \bhH^k - \bH\|_2^2 ~ \|\bhH^{\prime \ell}\|_F^2
	\le \frac{C \varsigma^2 r' \phi^2(\alpha_1)}{s^2_{r}}.  \label{eq:t.1}
\end{align} 
}
Note that we have used the fact that $\| \bhH^{\prime \ell} \|_F^2 = r'$. 

\vspace{10pt} \noindent 
{\em Bounding $\|(\bI - \bH) (\bhH^{\prime \ell} - \bHp)\|_F^2$. } 
Observe that $(\bI - \bH)$ is a projection matrix, and hence $\|\bI - \bH\| \le 1$. 
%
By adapting Lemma \ref{lem:thm1.1}, 
we have w.p. at least $1-\alpha_2$
\begin{align} \label{eq:adapt.1} 
\| \bhH^{\prime \ell} - \bHp\|_F^2 
&\le  r' \| \bhH^{\prime \ell} - \bHp\|_2^2 
\le \frac{C \varsigma^2 r' 
(\phi'(\alpha_2))^2}{(s'_{r'})^2}.
\end{align} 
Note that we have used the following: 
(i) $\| \bhH^{\prime \ell} - \bHp\|_F =
\| \sin \Theta \|_F$, where $\sin \Theta \in \Rb^{r' \times r'}$ is a matrix of principal angles between the two projectors \citep{ABSIL2006288}, which implies $\rank(\bhH^{\prime \ell} - \bHp) \le r'$;
(ii) the standard norm inequality $\| \bA \|_F \le \sqrt{\rank(\bA)} \| \bA \|_2$ for any matrix $\bA$. 
Using the result above, we have 
{\small
\begin{align}
	\| (\bI - \bH) (\bhH^{\prime \ell} - \bHp)\|_F^2 &\le 
	\|\bI - \bH\|_2^2 ~ \| \bhH^{\prime \ell} - \bHp\|_F^2
	\le \frac{C \varsigma^2 r' (\phi'(\alpha_2))^2}{(s'_{r'})^2}.  \label{eq:t.2} 
\end{align} 
}

\vspace{5pt} \noindent 
{\em Bounding $\langle (\bH - \bhH^k) \bhH^{\prime \ell}, ~(\bI - \bH) \bhH^{\prime \ell} \rangle_F$. } 
Using the cyclic property of the trace operator, we have that
{\small
\begin{align}
	 &\langle (\bhH^k - \bH) \bhH^{\prime \ell}, ~ (\bI - \bH) \bhH^{\prime \ell} \rangle_F 
	= \tr \big( (\bhH^{\prime \ell})^T (\bhH^k - \bH)  (\bI - \bH) \bhH^{\prime \ell} \big)
	\\&= \tr \big( (\bhH^k - \bH)  (\bI - \bH) \bhH^{\prime \ell} \big). \label{eq:helper_type_2_trace}
\end{align} 
}
Note that $\bhH^k - \bH$ is symmetric, and $\bI - \bH$ and $\bhH^{\prime \ell}$ are both symmetric positive semidefinite (PSD). As a result, 
Lemmas \ref{lem:thm1.1} and \ref{lemma:trace.2} yield w.p. at least $1-\alpha_1$ 
\begin{align}
	&\tr \big( (\bhH^k - \bH)  (\bI - \bH) \bhH^{\prime \ell} \big)
	\le \| \bhH^k - \bH \|_2  \tr \big( (\bI - \bH) \bhH^{\prime \ell} \big) 
	\\ &\le \| \bhH^k - \bH \|_2  \| \bI - \bH \|_2  \tr \big( \bhH^{\prime \ell} \big) 
	\le \frac{C \varsigma r' \phi(\alpha_1)}{s_{r}}.
	\label{eq:t.3} 
\end{align} 
Again, to arrive at the above inequality, we use $\| \bI - \bH\|_2 \le 1$ and $\text{tr}(\bhH^{\prime \ell}) = r'$. 

\vspace{10pt} \noindent {\em Collecting terms. } 
Collecting \eqref{eq:t.1}, \eqref{eq:t.2}, and \eqref{eq:t.3} with $\alpha_1 = \alpha_2 = \alpha/2$,  w.p. at least $1-\alpha$, 
\begin{align}
	\htau &\le 
	\frac{C \varsigma^2 r' \phi^2(\alpha/2)}{s^2_{r}}
	+
	\frac{C \varsigma^2 r' (\phi'(\alpha/2))^2}{(s'_{r'})^2}
	+ 
	\frac{C \varsigma r' \phi(\alpha/2)}{s_{r}}. 
\end{align}
Defining the upper bound as $\tau(\alpha)$ completes the bound on the Type I error. 

\vspace{10pt} \noindent
{\em Type II error. }
Next, we bound the Type II error. 
We will leverage Lemma \ref{lemma:type2}, the proof of which can be found in Appendix \ref{sec:type2_lem}. 

\begin{lemma} \label{lemma:type2}
The following equality holds: $\htau = r' - c_1 - c_2$, where 
{\small
\begin{align} 
    c_1 &= \|\bH \bH^T \bHp\|_F^2
    \\ c_2 &= \|(\bhH^k - \bH) \bhH^{\prime \ell}\|_F^2
    + \|\bH (\bhH^{\prime \ell} - \bHp)\|_F^2
    \\ &\quad + 2 \langle (\bhH^k - \bH) \bhH^{\prime \ell}, \bH \bhH^{\prime \ell} \rangle_F
    + 2 \langle \bH (\bhH^{\prime \ell} - \bHp),  \bH \bHp \rangle_F. \label{eq:type2.1} 
\end{align} 
}
\end{lemma} 
We proceed to bound each term on the right hand side of \eqref{eq:type2.1} separately. 

\vspace{10pt} \noindent 
{\em Bounding $\|(\bhH^k - \bH) \bhH^{\prime \ell}\|_F^2$. } 
From \eqref{eq:t.1}, we have that w.p. at least $1-\alpha_1$, 
\begin{align}  \label{eq:t.5}
	\|(\bhH^k - \bH) \bhH^{\prime \ell} \|_F^2 
	&\le \frac{C \varsigma^2 r' \phi^2(\alpha_1)}{s^2_{r}}. 
\end{align} 

\vspace{5pt} \noindent 
{\em Bounding $\|\bH (\bhH^{\prime \ell} - \bHp)\|_F^2$. } 
Using the inequality $\|\bA \bB \|_F \le \|\bA \| \|\bB \|_F$ for any two matrices $\bA$ and $\bB$, as well as the bound in  \eqref{eq:adapt.1}, we have w.p. at least $1-\alpha_2$, 
\begin{align}  \label{eq:t.6}
	\| \bH (\bhH^{\prime \ell} - \bHp) \|_F^2 
	&\le \frac{C \varsigma^2 r' (\phi'(\alpha_2))^2}{(s'_{r'})^2}.
\end{align}

\vspace{10pt} \noindent 
{\em Bounding $\langle (\bhH^k - \bH) \bhH^{\prime \ell}, \bH \bhH^{\prime \ell} \rangle_F$. } 
Using an identical argument used to create the bounds in \eqref{eq:helper_type_2_trace} and \eqref{eq:t.3}, but replacing $\bI - \bH$ with $\bH$, we obtain w.p. at least $1-\alpha_1$
\begin{align} \label{eq:t.7}
	 \langle (\bhH^k - \bH) \bhH^{\prime \ell}, \bH \bhH^{\prime \ell} \rangle_F 
	& \le  \frac{C \varsigma r' \phi(\alpha_1)}{s_{r}}.
\end{align}

\vspace{5pt} \noindent 
{\em Bounding $\langle \bH (\bhH^{\prime \ell} - \bHp),  \bH \bHp \rangle_F$. }
Like in the argument to produce the bound in \eqref{eq:t.3}, we use Lemmas \ref{lem:thm1.1} and \ref{lemma:trace.2} to get that w.p. at least $1-\alpha_2$, 
{\small
\begin{align}
&\langle \bH (\bhH^{\prime \ell} - \bHp), \bH \bHp \rangle_F
 = \tr \big( (\bhH^{\prime \ell} - \bHp) \bH^2 \bHp \big) 
\\&= \tr \big( (\bhH^{\prime \ell} - \bHp) \bH \bHp \big)
\le \| \bhH^{\prime \ell} - \bHp \|_2 ~ \|\bH\|_2  \tr \big( \bHp \big) 
\le \frac{C \varsigma r' \phi'(\alpha_2)}{s'_{r'}}. \label{eq:t.8} 
\end{align} 
}

\noindent {\em Collecting terms. }
Combining \eqref{eq:t.5}, \eqref{eq:t.6}, \eqref{eq:t.7}, \eqref{eq:t.8} with $\alpha_1 = \alpha_2 = \alpha/2$, and using the definition of $\tau(\alpha)$, we have that w.p. at least $1-\alpha$,
$$
	c_2 \le \tau(\alpha) + \frac{C \varsigma r' \phi'(\alpha/2)}{s'_{r'}}.
$$
Hence, along with using Lemma \ref{lemma:type2}, it follows that w.p. at least $1-\alpha$,
\begin{align}\label{eq:unonditional_type_2_bound}
	\htau &\ge r' - c_1 - \tau(\alpha) - \frac{C \varsigma r' \phi'(\alpha/2)}{s'_{r'}}. 
\end{align} 
Now, suppose $r'$ satisfies \eqref{eq:type2_cond}, which implies that $H_1$ must hold. 
Then, \eqref{eq:unonditional_type_2_bound} and \eqref{eq:type2_cond} together imply  $\Pb(\htau > \tau(\alpha) | H_1) \ge 1 - \alpha$. 
This completes the proof.

\subsection{Proof of Corollary \ref{cor:hypo}}
\begin{lemma} [Gaussian Matrices: Theorem 7.3.1 of \cite{vershynin2018high}] \label{lemma:subg_matrix_gaussian}
Let $\bA$ be a $m \times n$ random matrix where the entries $A_{ij}$ are Gaussian r.v.s with variance $\varsigma^2$. 
Then for any $t > 0$,  
$
\| \bA \|_2 \le \varsigma (\sqrt{m} + \sqrt{n} + t)
$
w.p. at least $1-2\exp(-t^2)$. 
\end{lemma} 

\begin{lemma}\label{lem:wedin_gaussian}
Let the setup of Lemma \ref{lem:thm1.1} hold.
Further, assume the entries of $\bW$ and $\bWp$ are independent Gaussian r.v.s with variance $\varsigma^2$. 
Then for any $\alpha \in (0,1)$, we have w.p. at least $1-\alpha$, 
\begin{align}
    \| \bhH^k - \bH \|_2 
    \le \frac{2 \varsigma \phi(\alpha)}{s_{r}},
    ~~
      \| \bhH^{\prime \ell} - \bHp \|_2 
    \le \frac{2 \varsigma \phi_\epost(\alpha)}{s'_{r'}}. 
\end{align}
\end{lemma}

\begin{proof}
The proof is identical to that of Lemma \ref{lem:thm1.1} except $\| \bZ - \bX \|$ is now bounded above using Lemma \ref{lemma:subg_matrix_gaussian}. 
\end{proof}
The remainder of the proof of Corollary \ref{cor:hypo} is identical to that of Theorem \ref{thm:hypo}.

\subsection{Proof of Lemma \ref{lemma:type1.1}} \label{sec:type1_lem}

Observe that
\begin{align} 
	\htau &= \|(\bI - \bhH^k) \bhH^{\prime \ell}\|_F^2 
	\\ &= \|(\bI - \bhH^k) \bhH^{\prime \ell} - (\bI - \bH) \bhH^{\prime \ell} + (\bI - \bH) \bhH^{\prime \ell}\|_F^2
	\\ & = \|(\bH - \bhH^k) \bhH^{\prime \ell} + (\bI - \bH) \bhH^{\prime \ell}\|_F^2
	\\ & = \|(\bH - \bhH^k) \bhH^{\prime \ell}\|_F^2 + \|(\bI - \bH) \bhH^{\prime \ell}\|_F^2
	 + 2 \langle (\bH - \bhH^k) \bhH^{\prime \ell}, (\bI - \bH) \bhH^{\prime \ell} \rangle_F. \label{eq:type1.0.new}
\end{align} 
Under $H_0$, it follows that $(\bI - \bH) \bHp = 0$. 
As a result, 
\begin{align}
	 \|(\bI - \bH) \bhH^{\prime \ell}\|_F^2 &=  \|(\bI - \bH) \bhH^{\prime \ell}\|_F^2
	 \\ & = \|(\bI - \bH) \bhH^{\prime \ell} - (\bI - \bH) \bHp \|_F^2
	 \\ & = \|(\bI - \bH) (\bhH^{\prime \ell} - \bHp)\|_F^2. 
\end{align} 
Applying these two sets of equalities above together completes the proof. 
 
\subsection{Proof of Lemma \ref{lemma:type2}} \label{sec:type2_lem}
Because the columns of $\bhH^{\prime \ell}$ are orthonormal, 
$
	r' = \|\bhH^{\prime \ell}\|_F^2 = \| \bhH^k \bhH^{\prime \ell}\|_F^2 + \|(\bI - \bhH^k) \bhH^{\prime \ell}\|_F^2. 
$
Therefore, it follows that
\begin{align}
	\htau &= \|(\bI - \bhH^k) \bhH^{\prime \ell}\|_F^2  
	= r' - \| \bhH^k \bhH^{\prime \ell}\|_F^2. \label{eq:type2.0_new}
\end{align}
Now, consider the second term of the equality above.  
\begin{align} \label{eq.1} 
	\|\bhH^k \bhH^{\prime \ell}\|_F^2 &= %
	\|\bhH^k \bhH^{\prime \ell} - \bH \bhH^{\prime \ell} + \bH \bhH^{\prime \ell}\|_F^2
	\\ &= \| (\bhH^k - \bH) \bhH^{\prime \ell}\|_F^2 + \|\bH \bhH^{\prime \ell}\|_F^2
	+ 2 \langle (\bhH^k - \bH) \bhH^{\prime \ell}, \bH \bhH^{\prime \ell} \rangle_F.
	\label{eq:type2.1_new} 
\end{align} 
Further, analyzing the second term of \eqref{eq:type2.1_new}, we note that 
\begin{align}
	\|\bH \bhH^{\prime \ell}\|_F^2 
	&= \|\bH \bhH^{\prime \ell}\|_F^2
	\\ & = \| \bH \bhH^{\prime \ell} - \bH \bHp + \bH \bHp\|_F^2
	\\ &= \| \bH (\bhH^{\prime \ell} - \bHp)\|_F^2 + \|\bH \bHp\|_F^2 
	+ 2 \langle \bH (\bhH^{\prime \ell} - \bHp),  \bH \bHp\rangle_F.
	\label{eq:type2.2_new} 
\end{align}
Incorporating \eqref{eq:type2.1_new} and \eqref{eq:type2.2_new} into \eqref{eq:type2.0_new}, and recalling $c_1 = \|\bH \bHp\|_F^2 = \|\bH \bH^T \bHp\|_F^2 $ completes the proof.

\subsection{Helper Lemmas}

\begin{lemma} \label{lemma:trace.1} 
Let $\bA, \bB \in \Rb^{n \times n} $ be symmetric PSD matrices. Then, $\tr(\bA \bB) \ge 0$. 
\end{lemma} 

\begin{proof}
Let $\bB^{1/2}$ denote the square root of $\bB$. Since $\bA \succeq 0$, we have
\begin{align}
	\tr(\bA \bB) &= \tr(\bA \bB^{1/2} \bB^{1/2})
	\\&= \tr(\bB^{1/2} \bA \bB^{1/2}) %
	\\& = \sum_{i=1}^n (\bB^{1/2} \be_i)^\prime \bA  (\bB^{1/2} \be_i) \ge 0.
\end{align} 
\end{proof} 

\begin{lemma} \label{lemma:trace.2} 
If $\bA \in \Rb^{n \times n}$ is a symmetric matrix and $\bB \in \Rb^{n \times n}$ is a symmetric PSD matrix, then
$
	\tr(\bA \bB) \le \lambda_{\max}(\bA) \cdot \tr(\bB),
$
where $\lambda_{\max}(\bA)$ is the top eigenvalue of $\bA$. 
\end{lemma} 

\begin{proof}
Since $\bA$ is symmetric, it follows that  
$\lambda_{\max}(\bA) \bI - \bA \succeq 0$. 
As a result, applying Lemma \ref{lemma:trace.1} yields
$
	\tr((\lambda_{\max}(\bA) \bI - \bA) \bB) = \lambda_{\max}(\bA) \cdot \tr(\bB) - \tr(\bA \bB) \ge 0. 
$
\end{proof}

%% file: content/appendix_lower_bound.tex
\section{Towards a Lower Bound on Model Identification} \label{sec:lower_bound}
We now take a first step towards establishing a lower bound on PCR's parameter estimation error in Lemma~\ref{lemma:lower_bound} below. 
Recall that Theorem~\ref{thm:param_est} implies that PCR faithfully recovers the model parameter $\btbeta^*$ provided $\snr$ grows sufficiently fast.
Conversely, if $\snr = O(1)$, then Lemma~\ref{lemma:lower_bound} suggests the parameter estimation error is lower bounded by an absolute constant. 
To establish our result, we show that the Gaussian location model problem \citep{YihongWu} is an instance of error-in-variables regression.
\begin{lemma}\label{lemma:lower_bound}
Let $n = O(p)$ and $\snr = O(1)$.
Then,
\begin{align} \label{eq:lower_bound_param_est}
\inf_{\bhbeta} \sup_{\btbeta^* \in \mathbb{B}_2} \Ex\| \bhbeta - \btbeta^* \|^2_2 = \Omega(1),
\end{align}
where $\mathbb{B}_2 = \{\bv \in \Rb^p : \| \bv \|_2 \le 1 \}$.
\end{lemma}

We make several important remarks.
First and foremost, our result stated in Lemma~\ref{lemma:lower_bound} is only a partial correspondence with that stated in Theorem~\ref{thm:param_est}. 
The minimax bound in Lemma~\ref{lemma:lower_bound} is stated with $\rho = 1$, i.e., it does not capture the refined dependence on $\rho$. 
Meanwhile, \eqref{eq:snr} and \eqref{eq:thm1.main} suggest that the error decays as $\rho^{-4}$. 
While this dependency on $\rho$ may not be optimal, similar dependencies have appeared in error bounds within the error-in-variables literature, e.g., \cite{loh_wainwright} and references therein. 
Establishing the optimal dependence with respect to $\rho$ is interesting future work.

Moreover, Lemma~\ref{lemma:lower_bound} considers the constraint set $\mathbb{B}_2$, which contrasts with that considered in the main body of this work. 
Finally, as seen in the proof below, our reduction argument utilizes a specific choice of $\bX$ while the main body of this work considers a fixed design matrix that the practitioner is unable to choose. 
Closing the gap on these limitations would significantly enhance the current lower bound, and we leave a formal treatment of this problem as important future work.

%% file: content/proof_lower_bound.tex
\subsection{Proof of Lemma~\ref{lemma:lower_bound}} \label{sec:proof_lower_bound}
Broadly, we proceed in three steps:
(i) stating the Gaussian location model (GLM) and an associated minimax result; 
(ii) reducing GLM to an instance of error-in-variables regression; 
(iii) establishing a minimax result on the parameter estimation error of error-in-variables using the GLM minimax result. 

{\em Gaussian location model.}
Below, we introduce the GLM setting through a well-known minimax result. 
\begin{lemma}[Theorem 12.4 of \cite{YihongWu}]\label{lemma:GLM_minimax}
Let $\btheta \sim \Nc(\btheta^*, \sigma^2\bI_p)$, where $\bI_p \in \Rb^{p \times p}$ is the identity matrix and $\btheta, \btheta^* \in \Rb^p$.
Given $\btheta$, let $\bbhtheta$ be any estimator of $\btheta^*$.
Then,
\begin{align}
    \inf_{\bbhtheta} \sup_{\btheta \in \mathbb{B}_2} \Ex \| \bbhtheta - \btheta^* \|_2^2 = \Theta( \sigma^2p \wedge 1). 
\end{align}
\end{lemma}

{\em Reducing GLM to error-in-variables.} 
We will now show how an instance of GLM can be  reduced to an instance of error-in-variables.
Towards this, we follow the setup of Lemma \ref{lemma:GLM_minimax} and define $\bbeta^* = \btheta^*$, $\bbeta = \btheta$, and $s = 1/\sigma$.
For convenience, we write $\bbeta = \bbeta^* + \bEta$, where the entries of $\bEta$ are independent Gaussian r.v.s with mean zero and variance $1/s^2$; 
hence $\bbeta \sim \Nc(\bbeta^*, (1/s^2) \bI_p)$. 
Now, recall that the error-in-variables setting reveals a response vector
$y = \bX \bbeta^* + \bvarepsilon$ and covariate $\bZ = \bX + \bW$, where the parameter estimation objective is to recover $\bbeta^*$ from $(\by, \bZ)$. 
%
Below, we construct instances of these quantities using $\bbeta, \bbeta^*$ as follows: 
\begin{itemize}
    \item [(i)] Let the SVD of $\bX$ be defined as $\bX = s \bu \otimes \bv$, where $\bu = (1, 0, \dots, 0)^T \in \Rb^n$ and $\bv = \bbeta^*$. Note by construction, $\rank(\bX) = 1$ and $\bbeta^* \in \text{rowspan}(\bX)$. 
    
    \item [(ii)] To construct $\by$, we first sample $\bvarepsilon \in \Rb^n$ whose entries are independent standard normal r.v.s. 
    Next, we define $\by = s \bu + \bvarepsilon$. 
    From (i), we note that $\bX \bbeta^* = s\bu$ such that $\by$ can be equivalently expressed as $\by = \bX \bbeta^* + \bvarepsilon$. 
    
    \item [(iii)] 
    Let $\bZ = s \bu \otimes \bbeta$. By construction, it follows that $\bZ = \bX + s \bu \otimes \bEta$.
    Note that $\bW = s \bu \otimes \bEta$ is an $n \times p$ matrix whose entries in the first row are independent standard normal r.v.s and the remaining entries are zero. 
\end{itemize}

{\em Establishing minimax result.} 
As stated above, the error-in-variables parameter estimation task is to construct $\bhbeta$ from $(\by, \bZ)$ such that $\| \bhbeta - \bbeta^*\|_2$ vanishes as $n, p$ grow. 
Using the above reduction combined with Lemma \ref{lemma:GLM_minimax}, it follows that 
\begin{align}
    \inf_{\bhbeta} \sup_{\bbeta^* \in \mathbb{B}_2} \Ex\| \bhbeta - \bbeta^* \|_2^2 = \Theta( p/s^2 \wedge 1). 
\end{align}
To attain our desired result, it suffices to establish that $p/s^2 = \Omega(1)$. 
%
By \eqref{eq:snr} and under the assumption $n = O(p)$, we have that $s^2 \le 2 \snr^2 (n + p) \le c \snr^2 p$ for some $c > 0$.
As such, if $\snr = O(1)$, then the minimax error is bounded below by a constant.